\documentclass[preprint,12pt]{amsart}

\usepackage{amssymb,amsmath,amscd,latexsym,amsthm,hyperref}
\usepackage{graphicx}
\usepackage{amsfonts}
\usepackage{color}
\usepackage{pdfsync}
\usepackage[all]{xy}
\usepackage{appendix}
\allowdisplaybreaks
\input epsf

\usepackage[all]{xy}
\newcommand{\commentout}[1]{}

\newcommand{\f}[1]{{\mathfrak{#1}}}
\newcommand{\Span}{\mathrm{span}}

\newcommand{\rank}{\mathrm{Rank\, }}

\renewcommand{\det}{\mathrm{det}}

\theoremstyle{plain}
\numberwithin{equation}{section}
\newtheorem{theorem}{Theorem}[section]
\newtheorem{corollary}[theorem]{Corollary}

\newtheorem{lemma}[theorem]{Lemma}
\newtheorem{proposition}[theorem]{Proposition}

\theoremstyle{definition}

\newtheorem{remark}[theorem]{Remark}
\newtheorem{example}[theorem]{Example}

\def\text#1{\;\;\;\;{\rm \hbox{#1}}\;\;\;\;}
\def\qquad{\quad\quad}

\def\msy#1{{\mathbb #1}}
\def\C{{\msy C}}
\def\N{{\msy N}}

\def\Z{{\msy Z}}
\def\K{{\msy K}}

\def\T{{\msy T}}
\def\R{\mathbb{R}}
\def\ga{\alpha}

\def\e{\epsilon}
\def\fe{{\mathfrak e}}
\def\ff{{\mathfrak f}}

\def\frak#1{\mathfrak #1}
\def\fa{{\mathfrak a}}

\def\fc{{\mathfrak c}}
\def\fg{{\mathfrak g}}
\def\fh{{\mathfrak h}}

\def\fk{{\mathfrak k}}
\def\fl{{\mathfrak l}}
\def\fm{{\mathfrak m}}
\def\fn{{\mathfrak n}}
\def\fp{{\mathfrak p}}
\def\fq{{\mathfrak q}}
\def\fs{{\mathfrak s}}
\def\ft{{\mathfrak t}}
\def\fu{{\mathfrak u}}
\def\fz{{\mathfrak z}}

\newcommand{\bc}{\mathbf{c}}
\newcommand{\bt}{\mathbf{t}}

\newcommand{\bC}{\mathbf{C}}

\def\to{\rightarrow}

\def\Im{{\rm Im}\,}

\def\Ad{{\rm Ad}}

\def\ad{{\rm ad}}

\def\pr{{\rm pr}}

\newcommand{\id}{\mathrm{id}}

\def\GL{\mathrm{GL}}
\newcommand{\SL}{\mathrm{SL}}

\newcommand{\so}{\mathfrak{so}}
\newcommand{\su}{\mathfrak{su}}

\newcommand{\SU}{\mathrm{SU}}
\newcommand{\U}{\mathrm{U}}

\renewcommand{\sp}{\mathfrak{sp}}
\newcommand{\gl}{\mathfrak{gl}}

\newcommand{\lsl}{\mathfrak{sl}}
\renewcommand{\sl}{\mathfrak{sl}}

\def\cD{{\mathcal D}}

\def\cO{{\mathcal O}}
\def\cP{{\mathcal P}}

\theoremstyle{definition}
\newtheorem*{Basic Example}{Basic Example}

\def\rmS{{\rm S}}
\def\rmU{{\rm U}}

\def\SU{\rmS\rmU}

\def\sideremark#1{\ifvmode\leavevmode\fi\vadjust{\vbox to0pt{\vss
 \hbox to 0pt{\hskip\hsize\hskip1em%
 \vbox{\hsize2cm\tiny\raggedright\pretolerance10000
 \noindent #1\hfill}\hss}\vbox to8pt{\vfil}\vss}}}

\newcommand{\chC}{{\mathfrak{c}^{\C}_{h}}}
\newcommand{\gh}{{\mathfrak{g}_h}}
\newcommand{\ghC}{{\mathfrak{g}^{\C}_{h}}}
\newcommand{\ghCnu}{{\mathfrak{g}_{h}^{\C [\nu ]}}}
\newcommand{\ghnu}{{\mathfrak{g}_{h}^{[\nu ]}}}
\newcommand{\ghone}{\mathfrak{g}_h^{(1)}}
\newcommand{\gone}{\mathfrak{g}^{(1)}}

\newcommand{\ghCZ}{{\mathfrak{g}_{h}^{\C [0]}}}
\newcommand{\ghCOne}{{\mathfrak{g}_{h}^{\C [1 ]}}}
\newcommand{\ghCTwo}{{\mathfrak{g}_{h}^{\C [2]}}}

\newcommand{\ghZ}{{\mathfrak{g}_{h}^{[0]}}}
\newcommand{\gZ}{{\mathfrak{g}_{h}^{[0]}}}

\newcommand{\ghOne}{{\mathfrak{g}_{h}^{[1]}}}
\newcommand{\ghTwo}{{\mathfrak{g}_{h}^{[2]}}}
\newcommand{\ghEven}{{\mathfrak{g}_{h}^{[{\rm even}]}}}

\newcommand{\kh}{\mathfrak{k}_h}
\newcommand{\khZ}{\mathfrak{k}_h^{[0]}}

\newcommand{\khTwo}{\mathfrak{k}_h^{[2]}}
\newcommand{\khC}{{\mathfrak{k}_{h}^{\C}}} 

\newcommand{\phKa}{\mathfrak{p}_{h\kappa}^*}
\newcommand{\php}{{\fp_{h+}}}
\newcommand{\phm}{{\fp_{h-}}}
\newcommand{\phPm}{\mathfrak{p}_{h\pm}}
\newcommand{\PhPm}{\mathbb{P}_{h\pm}}

\newcommand{\phPOne}{\mathfrak{p}_{h+}^{[1]}}
\newcommand{\phPTwo}{\mathfrak{p}_{h+ }^{[2]}}
\newcommand{\phPmZ}{\mathfrak{p}_{h\pm}^{[0]}}
\newcommand{\phPmOne}{\mathfrak{p}_{h\pm}^{[1]}}
\newcommand{\phPmTwo}{\mathfrak{p}_{h\pm}^{[2]}}
\newcommand{\phone}{\mathfrak{p}^{(1)}_h}
\newcommand{\phtwo}{\mathfrak{p}^{(2)}_h}
\newcommand{\phtwoC}{\mathfrak{p}^{(2)\C}_h}
\newcommand{\phoneP}{\mathfrak{p}^{(1)}_{h+}}
\newcommand{\phoneM}{\mathfrak{p}^{(1)}_{h-}}
\newcommand{\phonePM}{\mathfrak{p}^{(1)}_{h\pm}}
\newcommand{\nP}{\mathfrak{n}_+}
\newcommand{\nM}{\mathfrak{n}_-}
\newcommand{\ph}{{\fp_h}}
\newcommand{\ah}{{\fa_h}}
\renewcommand{\th}{{\mathfrak{t}_h}}

\newcommand{\phC}{\fp_{h}^{\C} }
\newcommand{\Gh}{{G_h}}
\newcommand{\GhC}{\mathbb{G}_{h}}
\newcommand{\GC}{\mathbb{G}}
\newcommand{\Kh}{K_h}
\newcommand{\KhC}{{\mathbb{K}_{h}}}
\newcommand{\KC}{{\mathbb{K}}}
\newcommand{\Php}{\mathbb{P}_{h+}}
\newcommand{\Phm}{\mathbb{P}_{h-}}
\newcommand{\phpm}{{\mathfrak{p}_{h\pm}}}

\newcommand{\Dh}{\cD_h}
\newcommand{\NhC}{{\mathbb{N}_{h}}}
\newcommand{\AhC}{{\mathbb{A}_{h}}}

\newcommand{\DP}{{\cD_{h+}}}
\newcommand{\DPt}[1]{\cD_{h+}^{#1}}
\newcommand{\Oh}{\Omega_h}
\newcommand{\ch}{{\mathfrak{c}_h}}

\newcommand{\zh}{{\mathfrak{z}_{\kh}}}
\newcommand{\Sh}{{\Sigma_h}}
\newcommand{\ahp}{{\mathfrak{a}_h^+}}
\newcommand{\PhIe}{{Q_{h}(I,\epsilon )}}
\newcommand{\PhIeC}{{\mathbb{Q}_{h}(I,\epsilon)}}

\newcommand{\qk}{\mathfrak{q}_{h\kappa}}
\newcommand{\qRk}{\mathfrak{q}_\kappa}

\newcommand{\ghbKTwo}{\mathfrak{g}_{h\kappa}^{(2)}}

\newcommand{\ghKappa}{\mathfrak{g}_{h\kappa}}
\newcommand{\ghKa}{\mathfrak{g}_{h\kappa}^*}

\newcommand{\khKa}{\mathfrak{k}_{h\kappa}^*}
\newcommand{\khKaC}{\mathfrak{k}_{h\kappa}^{*\C}}

\newcommand{\fbh}{\mathfrak{b}_h}
\newcommand{\fbhC}{\mathfrak{b}_h^{\mathbb{C}}}

\newcommand{\nhOne}{\mathfrak{n}_{h\kappa}^1}
\newcommand{\nhOneC}{\mathfrak{n}_{h\kappa}^{1\C}}
\newcommand{\nhTwo}{\mathfrak{n}_{h\kappa}^2}
\newcommand{\nhTwoC}{\mathfrak{n}_{h\kappa}^{2\C}}
\newcommand{\nh}{\mathfrak{n}_{h\kappa}}
\newcommand{\mhZero}{\mathfrak{m}_{h\kappa}^0}
\newcommand{\mhone}{\mathfrak{m}_{h\kappa}^{(1)}}
\newcommand{\mhtwo}{\mathfrak{m}_{h\kappa}^{(2)}}
\newcommand{\mhtwoC}{\mathfrak{m}_{h\kappa}^{(2)\C}}
\newcommand{\Qkappa}{Q_{h\kappa}}
\newcommand{\Nkappa}{N_{h\kappa}}
\newcommand{\MkappaZ}{M_{h\kappa}^0}
\newcommand{\MkappaP}{M_{h\kappa}^0}

\newcommand{\nOne}{\mathfrak{n}_{\kappa}^1}
\newcommand{\nTwo}{\mathfrak{n}_{\kappa}^2}

\newcommand{\mZero}{\mathfrak{m}_{\kappa}^0}

\newcommand{\mtwo}{\mathfrak{m}_{\kappa}^{(2)}}

\newcommand{\Sbh}{\Sigma (\ghC,\mathfrak{b}_h^\C)}
\newcommand{\Sc}{\Sigma_c}

\newcommand{\Scc}{\Sigma_{cc}}

\newcommand{\Scn}{\Sigma_{cn}}

\newcommand{\SccP}{\Sigma_{cc}^+}

\newcommand{\ScnP}{\Sigma_{cn}^+}
\newcommand{\Deh}{\Delta_h}
\newcommand{\Dhn}{\Delta_{hn}}
\newcommand{\Dhc}{\Delta_{hc}}

\newcommand{\tDh}{\tau_{\Dh}}
\newcommand{\thh}{\theta_h}
\newcommand{\dthh}{\dot \theta_h}
\newcommand{\dtau}{\dot\tau}
\newcommand{\dsh}{\dot\sigma_h}
\newcommand{\kpr}{\dot\kappa^\prime_{I,\epsilon}}

\newcommand{\dt}{\dot\tau}

\newcommand{\dtheta}{\dot\theta}
\newcommand{\Ght}{G_{h0}^\tau}
\newcommand{\ght}{\mathfrak{g}_{h}^{\dot\tau}}

\newcommand{\ws}{\widetilde{\sigma}}
\newcommand{\wt}[1]{\widetilde{#1}}
\newcommand{\wXi}{\widetilde{\Xi}}

\begin{document}

\title{Extensions of real bounded symmetric domains}
\author{Gestur {\'O}lafsson}
\address{Department of Mathematics, Louisiana State University, Baton Rouge, LA 70803}
\email{olafsson@math.lsu.edu}
\author{Robert J. Stanton}
\address{1937 Beverly Rd, Columbus, OH 43221}
\email{stanton.2@osu.edu}

\begin{abstract}
For a real bounded symmetric domain, $G/K$, we construct various natural enlargements to which several aspects of 
harmonic analysis on $G/K$ and $G$ have extensions. Our starting point is the realization of $G/K$ as a totally
real submanifold in a bounded domain $G_h/K_h$. We describe the boundary orbits and relate them to the
boundary   orbits of $G_h/K_h$. We relate the crown and the split-holomorphic crown of $G/K$ to the
crown $\Xi_h$ of $G_h/K_h$. We identify an extension of a representation of $K$ to
a larger group $L_c$ and use that to extend sections of vector bundles over the Borel compactification of $G/K$ to its closure. Also, we show there is an analytic extension of $K$-finite matrix coefficients of $G$ to a specific Matsuki cycle space.
\end{abstract}
\date{}
\keywords{
Structure of semisimple symmetric spaces, bounded domains, duality of symmetric spaces, extension of 
representations, crown}
\subjclass[2000]{Primary 32M15,  22E46, Secondary: 22E50, 53C35}
\maketitle
\section*{Introduction}
\setcounter{section}{0}

\noindent
\'Elie Cartan was the first to prove the existence of a compact real form of a complex semisimple Lie algebra. 
This can be considered the introduction of duality into the theory of Riemannian symmetric spaces. Subsequently, even
in the more general context of symmetric spaces, various people have identified several types of
duality. In this paper we explore some of the consequences of a type of duality
involving compactly causal spaces and noncompactly causal spaces or, said
geometrically, involving Hermitian and split-Hermitian
spaces\footnote{split-Hermitian
and split-complex are called para-Hermitian in \cite{K85,K87} and elsewhere.}
We describe here, in heuristic form, various results to which one is lead
(to conjecture) from this viewpoint. Some of this is, without a doubt, known to experts. Thus, 
as we use standard terminology, we relegate precise definitions and careful notation 
to subsequent sections, for now we take a more casual approach.

Let $\Gh$ be a semisimple Hermitian Lie group of noncompact type with maximal compact subgroup $\Kh$, i.e. $\Gh/\Kh$ is a Hermitian Riemannian symmetric space. Let $\tau$ be an involution commuting with $\theta$ and such that, $G$, the fixed point set of $\tau$ has Riemannian symmetric space, $G/K$, a real form of $\Gh/\Kh$. Denote by $\gh$ the Lie algebra of $\Gh$ and by $\ghC$ its complexification. Of course $\tau$ induces an involution $\dot\tau$ on $\gh$. We let $\dot\tau$ also denote the complex linear extension of $\dot\tau$ to $\ghC$; while we let $\dot\eta$ be its complex ${\it conjugate-linear}$ extension to $\ghC$. The associated holomorphic (resp. anti-holomorphic) involutions on $\GhC$ are denoted $\tau$ (resp. $\eta$). Then $\Gh$ (resp. $\Kh$) is the fixed point set of $\tau$ in $\GhC$ (resp. $\KhC$), and let $G_c$ (resp. $L_c$) be the fixed point set of $\eta$  in $\GhC$ (resp. $\KhC$). Then $G_c$ is a semisimple split-Hermitian Lie group, i.e. $G_c/L_c$ is a split-complex pseudo-Riemannian symmetric space with an integrable bi-Lagrangian structure. Now $\tau$ restricts to $G_c$ giving an involution ${\tau}_c$ having fixed point set $G$ such that $G/K$ is a split-real form of $G_c/L_c$. One could repeat the above with $G_c$ and $\dot{\tau}_c$ the complex linear extension to ${\mathfrak g}_c^{\mathbb C}$. Notice that ${\mathfrak g}_c^{\mathbb C}\cong \ghC$ but not equal. Various properties of  $\{\Gh,G,\eta\}$ and $\{G_c,G,{\tau}_c\}$ are the main focus of this paper. Detailed discussion about $G_c/L_c$ and its compactification can
be found in the work of Kaneyuki \cite{K85,K87}.

We begin with several decompositions involving $\{\Gh,G,\eta\}$. 

From Harish-Chandra we have the open subset
\begin{equation} \tag{a}\Gh \mathbb {B}_h =\textrm {exp }\DP\, \KhC\,\Phm \subset \GhC,
\end{equation}
 then applying $\eta$ we should obtain  similarly the open containment
 \begin{equation}\tag{a'} GP_{min} = \textrm{exp}\,\cD_+ \,L_c\,N_{c-}\subset G_c.
 \end{equation}

  From \cite{KS} we have the complex open neighborhood of $\Gh$
 \begin{equation} \tag{b} \Gh\, \textrm{exp} \, i\Omega_h \,\KhC \subset \GhC,
 \end{equation}
 then with $\eta$ we should obtain an open neighborhood of $G$ 
 \begin{equation} \tag{b'} G\, \textrm{exp}\, i\Omega_h^{\dot\eta}\,  L_c \subset G_c.
 \end{equation}

 Also from \cite{KS} we have the holomorphic extension of the Iwasawa decomposition
 \begin{equation}\tag{c} \textrm{exp}\,i\Omega_h\, \Gh \subset \KhC\AhC\NhC,
 \end{equation}
 so that with $\eta$ we get
 \begin{equation}\tag{c'} \textrm{exp}\,i\Omega_h^{\dot\eta}\,G \subset L_c\AhC^{\eta}\NhC^{\eta}\subset G_c.
 \end{equation}
 
 The Akhiezer-Gindikhin crown of $\Gh$ is an open subset
 \begin{equation}\tag{d} \Xi_h : = \{\Gh \textrm{exp}\,i\Omega_h \KhC\}/\KhC\subset \GhC/\KhC,
 \end{equation}
 so with $\eta$ we should get for the \lq \,real\,\rq crown of $G$
 \begin{equation}\tag{d'} \Xi \cong \{G\, \textrm{exp}\, i\Omega_h^{\dot\eta}\,  L_c\}/L_c\subset G_c/L_c.
 \end{equation}

Now in \cite {KSII} and for a real form $G/K$, the existence of an open subset $\Xi_0\subset \Xi$ is shown such that
\begin{equation}\tag{e} \Xi_0 \,\textrm{ is biholomorphic to } \Gh/\Kh.
\end{equation}
A straightforward variation of that argument shows that 
 \begin{equation} \tag{e'}  \Xi_0 \, \textrm{ is split--biholomorphic to an open subset of } G_c/L_c.
 \end{equation}

 From various sources we have the crown of $\Gh$ is biholomorphic to an open subset of flag manifolds
 \begin{equation} \tag{f} \Xi_h  \subset \GhC/{\KhC\Phm} \times \GhC/\KhC\Php,
 \end{equation}
 so that applying $\eta$ we have for the crown of $G$ an open subset
 \begin{equation}\tag{f'} \Xi\subset G_c/L_cN_{c-} \times G_c/L_cN_{c+}.
 \end{equation}

In the various parts of the text we will identify several of these fixed point sets for $\eta$. The intent of this summary is to motivate several results. Now we give a more careful outline of the paper. The bounded Hermitian domain $\Gh/\Kh$ has a boundary that is a finite union of $\Gh$ orbits whose geometric structure is described in considerable detail in \cite {Sa}. We summarize this in \S1 so that in \S2 and \S3 using $\eta$ we may give a similar description of the boundary $G$ orbits for $G/K$. This geometric description was crucial in \cite {MSIII} to describe the decomposition of a natural holomorphic extension of homogeneous vector bundles to the boundary along these $\Gh$ orbits.  For the $\R$-form $G/K$ an extension of homogeneous vector bundles over $G/K$ to the boundary will be needed and a geometric description of their decomposition on the orbits. An extension of homogeneous vector bundles over $G/K$ is the content of \S4, \S5 and \S6. In \S7 we give a proof of the open neighborhood (c') using both $\eta$ and the main result in \cite{Ma}. Using this, the holomorphic extension of the Iwasawa decomposition (c) from \cite{KS}, together with $\eta$ we obtain then in  \S7  an analytic extension of $K$-finite matrix coefficients of irreducible representations of $G$ to $D = L_c \, \textrm{exp}\, i\Omega_h^{\dot\eta}\, G$.
 
\section{Bounded Symmetric Domains: Complex Case}\label{Section1}
\noindent 
We recall some facts about bounded symmetric domains in $\C^n$. This goes back
to \cite{KW65a,KW65b,W69,W72}, but for structure theory our reference is \cite{Sa}, 
although we shall alter his presentation to suit our needs; for analysis see 
\cite {KSII}, \cite {MSIII}.

\subsection{Notation}\hfill

Let $\Dh$ be a bounded symmetric domain in $\C^n$. The identity component of the group of holomorphic  automorphisms of $\Dh$ is a connected noncompact semisimple Lie group that we shall denote by $\Gh$ \footnote{The subscript $h$ will be used for objects related to the Hermitian symmetric space.}. The group $\Gh$ acts transitively, and the isotropy at any base point is a maximal compact subgroup of $\Gh$. We fix one and denote it by $\Kh$, so that $\Dh\simeq\Gh/\Kh$.  The Lie algebra of $\Gh$ (resp. $\Kh$) is denoted by $\gh$ (resp. $\kh$), while the superscript ${}^\C$ denotes a complexification of the indicated Lie algebra.  For a cleaner presentation we assume that $\Gh$ is simple, and that it is contained in a simply connected complex Lie group $\GhC$ whose Lie algebra is $\ghC$. The analytic subgroup of $\GhC$ corresponding to $\khC$ is denoted $\KhC$. The reason for requiring $\GhC$ to be simply connected comes from the following result, see \cite[Thm. 8.2, p. 320 and p. 351]{He78}. 
\begin{proposition}\label{le:conn} Let $G$ be a connected simply connected semisimple Lie group with finite center and $\sigma :G\to G$ an involutive homomorphism. Then
$G^\sigma :=\{a\in G\mid \sigma (a)=a\}$ is connected.
\end{proposition}
\begin{proof} In \cite{He78} $G$ is assumed to be compact; in \cite{Ra} $G$ is just simply connected. If $G$ is semisimple with finite center here is an easier argument. Let $\theta $ be a Cartan involution commuting with $\sigma$, 
$\fg =\fk\oplus \fp$ the associated Cartan decomposition, and $K=G^\theta$. Then $K$ is compact, connected, and simply connected by hypothesis. Furthermore  $G^\sigma = K^\sigma \exp (\fp^\sigma)$. Since
$K^\sigma$ is connected, the claim follows.
\end{proof}

If $\fh$ is a Lie algebra and if $\dot\varphi : \gh\to \fh$ is a Lie algebra homomorphism, then we denote by the same letter the
\textit{complex linear} extension, i.e. $\dot\varphi :\ghC\to \fh^\C$. Similarly on the group level, if $\tau : \Gh\to H$ is an analytic homomorphism, and
if  $H$ is contained in a complex Lie group $\mathbb{H}$ with Lie algebra $\fh^\C$, then
we will denote by the same letter the holomorphic extension, i.e. $\tau : \GhC\to \mathbb{H}$. This extension always
exists as we are assuming
that $\GhC$ is simply connected. The same convention will be used for other Lie groups without comment.

Let $\thh: \Gh\to \Gh$ be the Cartan involution corresponding to $\Kh$, i.e. $\thh^2=\id$ and $\Gh^{\thh}=\Kh$. Denote by $\dthh: \gh \to \gh$ the derived involution. Then $\kh =\{X\in \gh\mid \dthh (X)=X\}$ and with $\ph :=\{X\in \gh\mid \dthh (X)=-X\}$, one has $\gh=\kh\oplus \ph$.  The subspace $\ph$ can be identified with the tangent space of $\Dh$ at $e\Kh$. As  $\Dh$ is a complex domain, there is a complex structure $J: \ph \to \ph$. Moreover, $J$ extends to a derivation of $\gh$ which, as $\gh$ is semisimple, must be inner. Since  $J$ commutes with $\ad \kh|_{\ph}$, the derivation is represented by an element $Z_h$  in $\zh$, the center of $\kh$, i.e. $J = \ad Z_h|_{\ph} $. As we also assume that $\Gh$ is simple, one knows that $\zh$ is one dimensional, hence $J$ is essentially unique. 

As $ (\ad Z_h|_{\ph})^2=-1$, $\ad Z_h$ has eigenvalues $0$, $i$, and $-i$. For the respective eigenspaces we have $\ghC (\ad Z_h;0)=\khC$, and we set $\phPm :=\ghC(\ad Z_h;\pm i)$.  Then $\phPm$ is a complex abelian subalgebra of dimension $n$; $\KhC$ acts on $\phPm$; and $\phC=\php\oplus \phm$ as a $\KhC$-module. The $\KhC$-modules $\php$ and $\phm$ are contragredient and, as the center acts by a different constant, inequivalent.

Denote by $\Php$, resp. $\Phm$, the analytic subgroup of $\GhC$ corresponding to the Lie algebra  $\php$, resp. $\phm$. Then
$\PhPm$ is abelian, simply connected and $\exp : \phPm \to \PhPm$ is a holomorphic diffeomorphism and group homomorphism. We denote the inverse of $\exp|_{\php}$ by $\log$.

\begin{proposition}  $\Php \KhC\Phm$ is open and dense in $\GhC$, and the multiplication map
\[\Php \times \KhC\times \Phm \to \Php\KhC\Phm \, ,\quad (p_+, k,p_-)\mapsto p_+k p_-\]
is a holomorphic diffeomorphism. We denote the inverse by
\begin{equation}\label{eq-inv}
a\mapsto (p_+(a),{k}_h(a), p_-(a))\, .
\end{equation}
\end{proposition}

We consider the usual generalized flag manifold $\cP_h=\GhC/\KhC\Phm$ and a basepoint $x_o=e\KhC \Phm$. The $\Gh$ orbit of the basepoint, $\Gh\cdot x_o$, is $\Gh/\Kh\simeq\Dh$. On the other hand, the Bruhat cell $\Php\cdot x_o$ is open and dense in $\cP_h$. By means of  $\log$ one obtains a holomorphic isomorphism $\Php\cdot x_o\simeq \Php \simeq \php$, denoted by $g\cdot x_o \mapsto z(g\cdot x_o)$, such that for $p \in \Php\cdot x_o$, $k\in \KhC$ and
$X\in \php$
\begin{enumerate}
\item $z(k\cdot p)=\Ad (k)z(p)$
\item $z(\exp (X)\cdot p)=X+z(p)$.
\end{enumerate}

Restricted to $\Gh\cdot x_o$ the map has image $\DP\subset \php $, the Harish-Chandra bounded realization of $\Dh$.  

 \begin{theorem}\label{thm:1.3}
$\displaystyle
\php\supset \DP\simeq \Dh \simeq \Gh/\Kh \subset \cP_h =\GhC/\KhC\Phm. 
$
\end{theorem}

In a moment we will discuss the boundary components of $\Gh/\Kh$. For that we note that we can
take a closure in $\php$ or the closure in $\cP_h$. It is a priori not clear that those two closures should
be isomorphic. It is however clear that the closure in $\cP_h$ is $\Gh$-invariant, but
it is not clear that $\Gh$ acts on the closure in $\php$. In Lemma \ref{lem:ClosureSame} we
show that  $c({\DP})$, the closure of $\DP$ in $\php$, viewed as a subset of $\cP_h$ is
the same as the closure in $\cP_h$. In particular, $G_h$ acts on
$c({\DP})$\footnote{Note that this is not correct for the
unbounded realization of $\Gh/\Kh$ as the example of the upper half-plane shows.}.
Let $\partial \DP := c({\DP})\setminus \DP$ be the topological boundary
of $\DP$ in $\fp_h^+$. The action of $\Gh$ on $\DP$ extends to one on $\partial \DP$ which then decomposes into a finite disjoint union of $\Gh$-orbits. In a later section we shall give a complete parameterization  of the orbits and determination of the isotropy. This is well known, e.g. \cite {Sa},  but we include the proof because of its importance for our treatment of real domains.
 
\subsection{Essential Structure Theory - $\C$ forms}\hfill
 
Let $\ch$ be a Cartan subalgebra of $\gh$ containing $Z_h$, hence $\ch\subset \kh$. Let $\Deh$ be the set of roots of $\chC$ in $\ghC$. Since $\ch\subset \kh$, $\dthh|_{\ch}=\id $. Then $\dthh ({\ghC}_\alpha )={\ghC}_\alpha $, and as $\dim_\C \ghC_\alpha =1$,
either $\ghC_\alpha \subset \khC$ in which case one calls $\alpha$ a compact root, or $\ghC_\alpha\subset \phC $ and $\alpha$ is called noncompact. Denote by $\Dhc$ the set of compact roots, and by $\Dhn$ the set of noncompact roots. Then
\begin{align}\Dhc &=\{\alpha\in\Deh\mid \alpha (Z_h)=0\},\\
\Dhn &= \{\alpha \in \Deh\mid\alpha (Z_h)=\pm i\}.
\notag\end{align}
We choose the set of positive roots, $\Deh^+$, so that $\{\alpha \mid \alpha (Z_h)=i\}\subset \Deh^+$. 
Denote by $W_h=W(\Deh)$ the Weyl group generated by
reflections $s_\alpha$, $\alpha \in \Deh$, and denote by $W_{hc}$ the
subgroup generated by $s_\alpha$, $\alpha \in \Dhc$. As $\alpha (Z_h)=0$ for
all $\alpha \in \Dhc$ it follows that  $\Deh^+$ is invariant under   $W_{hc}$.

Recall that $\alpha,\beta\in \Deh$ are called \textit{strongly orthogonal} if $\alpha \pm \beta\not\in \Deh$. In the usual way one constructs a maximal set $\{\gamma_1,\ldots ,\gamma_{r_h}\}$ of strongly orthogonal roots in $\Dhn^+$. 

Denote by $\dsh:\ghC\to \ghC$ the conjugation
 with respect to $\gh$. For each $j=1,\ldots ,r_h$ choose $E_j\in \ghC_{\gamma_j}$ and set $F_j=\dsh (E_j)\in \ghC_{-\gamma_j}$. One can normalize $E_j$ so that with $H_j=[E_j,F_j]\in i\ch$ one has $\gamma_j (H_j)=2$. Let $Z_j:= iH_j$, $X_j:= E_j+F_j$, and $Y_j:=i(E_j-F_j)$. We set
\begin{equation}\label{def-th}
\th:=\bigoplus_{j=1}^{r_h}\R H_j\subset i\gh \text{ and } \ah:=\bigoplus_{j=1}^{r_h} \R X_j\subset \ph\, .
\end{equation}
Then $\ah$ is maximal abelian in $\ph$.

More generally, for $I\subseteq \{1,\ldots ,r_h\} $ and $\epsilon \in \{-1,1\}^{\# I}$ let $Z (I,\epsilon ):=\sum_{j\in I}\epsilon_j \, iH_j\,\in \ch $, $E(I,\epsilon ):=\sum_{j\in I}\epsilon_j E_j\,\in\ph_+$, $F(I,\epsilon ):=\sum_{j\in I}\epsilon_j F_j\,\in\ph_-$. Similarly $H (I,\epsilon ):=-iZ (I,\epsilon )$,
  $X(I,\epsilon ):= E(I,\epsilon )+F(I,\epsilon)\,\in\ph$ and $Y (I,\epsilon ):=i(E(I,\epsilon )-F(I,\epsilon))\,\in\kh$. We set $H_0=-iZ_h$. If $I=\{1,\ldots ,b\}$  then we write $E(b,\epsilon)$ instead of $E(I,\epsilon)$, etc., and if
furthermore $\epsilon =1$ then we simply write $E(b)$ etc.

Then $Z (I,\epsilon)$, $X (I,\epsilon)$ and $Y (I,\epsilon)$ generate a subalgebra of $\gh$ isomorphic to $\su (1,1)$.  These determine equivalence classes of holomorphic disks in $\Dh$ which, since we are in a homogeneous space, lift to equivalence
classes of compatible homomorphisms $\dot\kappa:\su (1,1))\to \gh$ together with their holomorphic extensions $\dot\kappa:\sl (2,\C)\to \ghC$, here $\sl (2,\C) :=\su (1,1)^\C$ . Amongst these, the 
homomorphisms associated to  $Z(b)$, $X (b)$ and $Y (b)$ play a critical role and will be referred to as basic homomorphisms. 
In passing we note that  $H (I,\epsilon )$, $E (I,\epsilon)$ and $F (I,\epsilon)$ generate a subalgebra of $\ghC$ isomorphic to $\sl (2,\R)$ which is the Cayley transform of the $\su (1,1)$ described above. Much of this notation is not needed in the complex case but will be needed when we do the real case. 

The word \lq Essential\rq \,in the subsection title refers to the fact that we have fixed the structure theory, whereas if, as in \cite {Sa}, one chooses first the geometry of holomorphic disks, one would have a different but equivalent choice of structure theory.

\begin{Basic Example} $\SU (1,1)$\hfill

The prototype bounded domain is $\SU (1,1)/\U(1)\simeq\{z\in \C\mid |z|<1\}$. In this case we have $\GhC=\SL (2,\C)$.
The conjugation $\dot\sigma_1$ \footnote{The subscript $1$ will be used for objects related to this basic case.} on $\sl (2,\C) (:=\su (1,1)^\C)$ with respect to $\su (1,1)$, and the holomorphic extension of the
 standard Cartan involution $\dot\theta_1$ of $\su (1,1)$ are given by
\[\dot\sigma_1\left(\begin{pmatrix} a & b\\ c & d\end{pmatrix}\right)=\begin{pmatrix} \overline{d} & \overline{c}\\
\overline{b} & \overline{a}\end{pmatrix}\text{ and }
\dot\theta_1\left(\begin{pmatrix} a & b\\ c & d\end{pmatrix}\right)=\begin{pmatrix} a & -b\\ -c & d\end{pmatrix}.\]
Let
\[h=\begin{pmatrix} 1 & 0 \\ 0 & -1\end{pmatrix}, \quad e=\begin{pmatrix} 0 & 1 \\ 0 & 0\end{pmatrix}
\text{ and } f=\begin{pmatrix} 0 & 0 \\ 1 & 0\end{pmatrix}.\]
Then $[h,e]=2e$, $[h,f]=-2f$, $[e,f]=h$, $\dot\sigma_1(h)=-h$ and $\dot\sigma_1 (e)=f$. Taking $Z_1=ih$ gives $\php=\C e$, $\khC=\C h$, and $\phm =\C f$.

A computation gives
\[\begin{pmatrix} 1 & z\\ 0 & 1\end{pmatrix} \begin{pmatrix} k & 0 \\
0 & k^{-1}\end{pmatrix} \begin{pmatrix} 1 & 0\\ y & 1\end{pmatrix}
=\begin{pmatrix} k+zy/k & z/k\\ y/k & 1/k \end{pmatrix}.\]
Thus
\[\Php \KhC\Phm =\left\{\left. \begin{pmatrix} a & b\\ c & d\end{pmatrix}\in \SL (2,\C)
\, \right|\, d\not= 0\right\}. \]
Identifying $\Php\cong \php \cong \C$, $\Phm\cong \phm\cong \C$ and  $\KhC\cong\C^*$
the maps in (\ref{eq-inv}) are given by
\begin{eqnarray*}
p_+\left(\begin{pmatrix} a & b\\ c & d\end{pmatrix}\right) &=& \begin{pmatrix} 1 &  b/d \\ 0 & 1\end{pmatrix}\mapsto b/d,\\
k_h \left(\begin{pmatrix} a & b\\ c & d\end{pmatrix}\right) &= &\begin{pmatrix} 1/ d & 0\\ 0 & d\end{pmatrix}\mapsto 1/d,\\
p_-\left(\begin{pmatrix} a & b\\ c & d\end{pmatrix}\right)&=&\begin{pmatrix} 1 & 0 \\ c/d & 1 \end{pmatrix}\mapsto c/d\, .
\end{eqnarray*}
This gives the Harish-Chandra realization of $\Dh$ as $\DP=D_1=\{z\in \C \mid |z|<1\}$. As
\[\begin{pmatrix} a & b\\ c & d\end{pmatrix}\begin{pmatrix} 1 & z \\ 0 & 1\end{pmatrix}=
\begin{pmatrix} a  & az+b\\ c & cz+d\end{pmatrix},\]
it follows that the action of $\SU (1,1)$ on $\Dh$ is the usual action $g\cdot z = \frac{az + b}{cz+d}$.
\end{Basic Example}

To return to the general situation we let in $\su (1,1)$
\begin{equation}\label{def-X1Y1}
Z_1=ih =\begin{pmatrix} i & 0 \\ 0 & -i\end{pmatrix},\quad  X_1:=e+f=\begin{pmatrix} 0 & 1\\ 1 & 0\end{pmatrix},  \text{and}Y_1:= i(e-f) = \begin{pmatrix} 0 & i \\ -i & 0\end{pmatrix}.
\end{equation}
For $I$ and $\epsilon $ as above,
\begin{equation}\label{def-kj}
\dot \kappa_{I,\epsilon} : Z_1\mapsto Z (I,\epsilon) \, ,\quad X_1\mapsto X (I, \epsilon), \, \quad \text{ and }\quad Y_1\mapsto Y (I,\epsilon)
\end{equation}
defines a Lie algebra homomorphism of $\sl (2,\C)$ into $\ghC$ such that
\begin{equation}\label{eq-kappaSigma}
\dot\kappa_{I,\epsilon} \circ\dot \sigma_1=\dsh \circ \dot\kappa_{I,\epsilon}\text{ and  }
\dot \kappa_{I,\epsilon}\circ \dot\theta_1=\dthh \circ\dot  \kappa_{I,\epsilon}\,.
\end{equation}
It follows, in particular, that
\[\su (1,1)_{I,\epsilon}:=\dot \kappa_{I,\epsilon} (\su (1,1))=\Span\{Z(I,\epsilon),X(I,\epsilon),Y(I,\epsilon)\}\subseteq \gh.\]
These are the lifts of holomorphic disks embedded into $\Dh$ and are those called standard homomorphisms.

Similarly
\begin{equation}\label{eq:SLkappa}
\sl (2,\R)_{I,\epsilon}:=\dot \kappa_{I,\epsilon}(\sl (2,\R))=\Span\{H (I,\epsilon),E(I,\epsilon),F(I,\epsilon)\}\, .
\end{equation}

As $\SL (2,\C)$ is simply connected,  there exists a group homomorphism $\kappa_{I,\epsilon} : \SL (2,\C)\to \GhC$ such that
$d\kappa_{I,\epsilon}=\dot \kappa_{I,\epsilon}$. In particular, $\kappa_{I,\epsilon} (\SU (1,1))\subseteq \Gh$. We set
$\dot \kappa_j=\dot \kappa_{\{j\},1}$ and $\kappa_j=\kappa_{\{j\},1}$. We also note that if $I\cap J=\emptyset$ then
$[\dot \kappa_{I,\epsilon}(\sl (2,\C)),\dot \kappa_{J,\epsilon^\prime}(\sl (2,\C))]=\{0\}$ and similarly for $\kappa_{I,\epsilon}$ and
$\kappa_{J,\epsilon^\prime}$. In particular, $\kappa_1,\ldots ,\kappa_{r_n}$ is a maximal family  of commuting standard homomorphisms
$\SL (2,\C)\to \GhC$.

A simple matrix calculation shows that
\[\exp \left(\frac{\pi i}{4}Y_1\right)=\frac{1}{\sqrt{2}}\begin{pmatrix} 1 & -1\\ 1 & 1\end{pmatrix}
\text{ and } \Ad \left(\exp \left(\frac{\pi i}{4}Y_1\right)\right)h=X_1\, .
\]
Thus if we set
\[\bc_{I,\epsilon} :=\exp \left(\frac{\pi i}{4}Y (I,\epsilon )\right)=\kappa_{I,\epsilon}\left(\frac{1}{\sqrt{2}}\begin{pmatrix} 1 & -1\\ 1 & 1\end{pmatrix}\right)
\text{ and } \bC_{I,\epsilon} := \Ad (\bc_{I,\epsilon})\, ,\]
then
\begin{equation}\label{eq-ZjXj}
\bC_{I,\epsilon} H (I,\epsilon) =X (I,\epsilon )\, .
\end{equation}

\begin{lemma}
Let $\bC =\bC_{\{1,\ldots ,r_h\},(1,\ldots ,1)}$. Then $\bC (\th)=\ah$.
\end{lemma}

\begin{theorem}[Moore]\label{th-Moore}
Let $\beta_j:= [\bC^{-1}]^{t}(\gamma_j|_{\th})$.
  There are two possibilities for the restricted roots $\Sh= \Sigma (\gh ,\ah)$:
\medskip

\noindent
\textbf{Case I:} $\displaystyle{ \Sh = \pm \left\{\left. \beta_i, \frac{1}{2}\left(\beta_j\pm \beta_k\right)\, \right|\,
i=1,\ldots , r_h\, ,\,\, 1\le j<k\le r_h\right\}}$,\\ 
\noindent
\textbf{Case II:} $\displaystyle{ \Sh = \pm \left\{\left.\beta_i, \frac{1}{2}\beta_i, \frac{1}{2}\left(\beta_j\pm \beta_k\right)\, \right|\,
i=1,\ldots , r_h\, ,\,\, 1\le j<k\le r_h\right\}}$.
\medskip

\noindent
The first case occurs if and only if $\Dh$ is a tube type domain.
 
\end{theorem}

We will use
\[\ahp=\left\{\left. \sum_{j=1}^{r_h} x_jX_j\,\right|\,  x_1>x_2>\ldots >x_{r_h}> 0\right\}\]
as a positive Weyl chamber. The corresponding set of positive roots are obtained by taking $+$ in front of
the parenthesis in Case I and II above.

\subsection{Boundary orbits}\hfill
 
 Using $\SU (1,1)$-reduction, eq. (\ref{eq-ZjXj}) is the main step in the proof that
 $\Dh \simeq \Gh/\Kh$ is diffeomorphic to a  bounded domain in $\php$. Indeed let
$\Oh=\sum_{j=1}^{r_h} (-1,1)E_j\subset \php$. For $\bt\in \R^{r_h}$ let $a_{\bt }:=
\exp \sum_{j=1}^{r_h}t_jX_j$. By a calculation in $\SU (1,1)$  we have
\begin{equation}\label{eq-actXj}
a_{\bt} \cdot \sum_{\nu =1}^{r_h}\xi_\nu  E_\nu
=\sum_{\nu =1}^{r_h} \frac{\cosh (t_\nu )\xi_{\nu}+\sinh (t_\nu)}{\sinh (t_\nu )\xi_\nu +\cosh (t_\nu)}\, E_\nu\, .
\end{equation}
In particular,
\[a_{\bt} \cdot 0=\sum_{\nu=1}^{r_h} \tanh (t_\nu ) E_{\nu}\, .\]
 Thus we have
\begin{equation}
\Dh \simeq \Gh/\Kh \simeq \DP=\Ad (\Kh )\Oh\subset \php, 
\end{equation}
the Harish-Chandra bounded realization of $\Dh$. 
Now it is clear that $\Gh$ acts on $\partial \DP$.  For $b\in \{1,\ldots ,r_h\}$ recall $E(b):=E_1+\ldots +E_b$
 and set $\cO_h(b):= \Gh\cdot E(b)$.
\begin{theorem}\label{th-boundh} Let $z\in \partial \DP$. Then there exists $b\in \{1,\ldots ,r_h\}$ and $g\in \Gh$ such that
$z=g \cdot E(b)$. In particular,
\[\partial \DP =\dot \bigcup_{b=1}^{r_h} \cO_h(b)\, .\]
Thus, the boundary orbits are parameterized by  $\{1,\ldots ,r_h\}$.
\end{theorem}
\begin{proof}
Let $\{z_n\}$ be a sequence in $\DP$ such that $z_n\to z$. As $\ah$ is maximal abelian in $\ph$, there exists $k_n\in\Kh$ and $t_{jn}\in (-1,1)$ such that $z_n=k_n\exp \sum_{j=1}^{r_h}t_{jn}E_j$. By applying a Weyl group element we can assume that $t_{1n}\ge t_{2n}\ge \ldots \ge t_{r_hn}\ge 0$. As $K$ and $[-1,1]$ are compact we can assume (by going to subsequences) that $k_n\to k\in K$ and $t_{jn}\to t_j\in [-1,1]$. Let $b$ be such that $t_1,\ldots ,t_b=1$ and
$1> t_j\ge 0$ for $j>b$. Let $s_j=- \tanh^{-1} (t_j)$ for $j >b$. By (\ref{eq-actXj}) we have
\[\exp \sum_{j>b}s_jX_j\cdot  \sum_{j=1}^{r_h} t_jE_j =E_1+\ldots + E_b=E(b)\, .\]
We can therefore take $g=\exp ( \sum_{j>b}s_jX_j)k^{-1}$.
\end{proof}

The closure  of ${\cD_h}$ in $\cP_h$ appears to be 
bigger than  $c({\DP})$ the closure of $\DP$ in $\php$. In fact we show,

\begin{lemma}\label{lem:ClosureSame} 
The closure of ${\cD_h}$ in $\cP_h$ is the same
as the closure, $c({\DP})$, of $\DP$ in $\exp (\php)\cdot x_o$. 
In particular, the action of $\Gh$ extends to $c({\DP})$.
\end{lemma}
\begin{proof} It is clear that the closure in $\exp (\php)\cdot x_o$ is contained in the
closure in $\cP_h$. As in the proof above, let $z= \lim_{j} k_ja_j\cdot x_o$ be in the closure of $\cD_h$ in $\cP_h$. Again let $k$ be a limit of a subsequence of $\{k_j\}$ and recall that $k_j$ and $k$ normalize
$\php$. As $\pm E_j\in \php$ it follows that $z\in k\cdot \exp \php\cdot x_o=\exp \php\cdot x_0$.
\end{proof}

Consequently, we do not have to distinguish if we are talking about the closure of $\cD_h$ in $\cP_h$, 
or the closure of ${\DP}$ in $\php$.

\subsection{Isotropy of boundary orbits}\hfill

We come to the determination of the isotropy of the various orbits in the boundary. Again,
we provide more details than are needed in the complex case, but they will be used later in the real case. Let

\begin{align}
\PhIe&:=\{g\in \Gh\mid g\cdot E (I,\epsilon) =E (I,\epsilon)\}\text{ and }\label{eq-Pba}\\
 \PhIeC&:=\{g\in\GhC\mid g\cdot E (I,\epsilon) =E (I, \epsilon)\}\, . \label{eq-Pbb}
\end{align}
Then the boundary orbit  $\cO_h (I,\epsilon ):=\Gh\cdot E (I,\epsilon)$ is isomorphic to  $\Gh/\PhIe$.
If $I=\{1,\ldots ,b\}$ and $\epsilon \in \{-1,1\}^b$ then we simply write $E(b,\epsilon) , Q_h(b,\epsilon),\cO_h(b,\epsilon)$, etc.
If $\epsilon =(1,\ldots ,1)$ then we do not include it in the notation.

As before, for a standard homomorphism $\dot \kappa : \sl (2,\C)\to \ghC$, i.e. (\ref{def-kj}),  we write $E_\kappa$ for $\dot\kappa (e)$,
$X_\kappa =\dot \kappa (e+f)$ etc. The corresponding homomorphism $\SL (2,\C)\to \GhC$ and the restriction to
$\SU (1,1)$ is denoted by $\kappa$. The following is valid for an arbitrary standard homomorphism.
To avoid even more burdensome notation we will use subscripts involving $\kappa$ only when it seems useful. We remark that
\cite{Sa} (Chapter 2 and 3) refers to a standard homomorphism as one $\dot\kappa : \lsl (2,\R)\to \gh$.
There should be no confusion from the terminology herein as the two are related by the Cayley transform introduced 
earlier, see \cite[p. 107--109]{Sa} for a detailed discussion.

Given a standard homomorphism let $\pi_\kappa :=\ad \circ \dot\kappa $. Then $\pi_\kappa$ is a finite dimensional representation of
$\sl (2,\C)$. As the irreducible representations of $\sl (2,\C)$ are  determined by their dimension with the  $1$-dimensional
representation being the trivial representation, the $2$-dimensional representation being the natural
representation of $\sl (2,\C)$ acting on $\C^2$, and the $3$-dimensional representation being the adjoint representation of
$\sl (2,\C)$ acting on itself.  The corresponding highest weights are $0$, $1$, and $2$. According
to \cite{Sa} Lemma 1.1 p. 90, every irreducible $\sl (2,\C)$ representation occurring
in $\pi_\kappa$ has dimension less than or equal to $3$. Following \cite{Sa}, for $\nu\in\{0,1,2\}$
denote by $\ghCnu$ (resp. $\ghnu$) the corresponding isotypic subspace. Then (as $ \lsl (2,\R)$ is split)
\begin{equation}\label{eq-gdecomp}
\gh = \ghZ\oplus \ghOne\oplus \ghTwo,
\end{equation}
and similarly for the complexification $\ghC$. From \cite{Sa} \S 1, Chapter 3 we obtain

\begin{lemma} Let $\kappa :\SU (1,1)\to \Gh$ be a standard homomorphism.  Let
\[Z_\kappa := \dot\kappa (Z_1)\text{ and } Z^{(1)}_\kappa := Z_h-\frac{1}{2}Z_\kappa \, .\]
Then the following conditions are equivalent:
\begin{enumerate}
\item $Z^{(1)}_\kappa=0$;
\item $\ghOne =\{0\}$ and $\ghZ$ is compact.
\end{enumerate}
\end{lemma}

Notice that each of the spaces $\ghCnu$ is $\dsh$ and $\dthh$ stable. As $\dot\kappa$ is standard, it intertwines the respective Cartan involutions and conjugations. Hence we have similar decompositions for $\kh$, $\khC$, $\ph$, $\phC$, and   $\phpm$.  In particular, we have
\begin{equation}\label{eq-phpm}
\phpm =\phPmZ\oplus \phPmOne \oplus \phPmTwo \, .
\end{equation}
Also, $\ghCZ (=\mathfrak{z}_{\ghC}( \dot\kappa (\lsl (2,\C) ))$ is a subalgebra, as is
\begin{equation}\label{eq-gEven}
\ghEven =\ghZ\oplus \ghTwo\, .
\end{equation}
Since $\ghEven$ is $\dthh$-stable, it follows that $\ghEven$ is a reductive subalgebra. Furthermore $Z_h\in \ghEven$,
so each non-compact ideal of $\ghEven$ is of Hermitian type.

Next we decompose $\ghEven $ into ideals, $\ghEven= \bigoplus_j \gh_j $, such that $\gh_0$ is the maximal compact ideal, while $\gh_j$ is simple and noncompact for $j\ge 1$. It follows that the maximal abelian ideal of $\ghEven$ is contained in $\gh_0$, and each $\gh_j$, $j\ge 1$, is of Hermitian type. Define
\begin{eqnarray}
\ghone&:=&\bigoplus_{\gh_j\subseteq \ghZ}\gh_j\nonumber\\
\mathfrak{l}_2&:=&\gh_0\nonumber\\
\ghKappa&:=&\bigoplus_{\gh_j\subseteq\ghZ,\, j\ge 1}\gh_j\label{eq-decomp}\\
\ghKa&:=&\bigoplus_{\gh_j\varsubsetneq \ghZ}\gh_j\nonumber\\
&=&\khKa\oplus\phKa\, .\nonumber
\end{eqnarray}
Then $\ghone\subseteq\ghZ, \ghTwo \subseteq \ghKa$ and $\ghEven = \ghone \oplus \ghKa$. The corresponding analytic subgroups of $\Gh$ will be denoted by the respective upper case Latin letter.

Finally we arrive at the parabolic subalgebra corresponding to $\kappa$. Recall that $X_\kappa =\dot \kappa (e+f)$. Let $
\mhZero:= \gh (\ad X_\kappa ;0)$, $\nhOne:= \gh (\ad X_\kappa ;1)$,  $\nhTwo :=\gh (\ad X_\kappa ;2)$,
$\nh:= \nhOne\oplus \nhTwo$ and $\qk:= \mhZero \oplus \nh$\footnote{Note that our notation here differs from
\cite[p.95]{Sa} where $\nhOne$ is denoted by $V_\kappa$ and $\nhTwo$ is
denoted by $U_\kappa$.} 
Then
$\qk$
is a maximal parabolic subalgebra of $\gh$. Denote by $\Qkappa=\MkappaZ \Nkappa$ the
corresponding maximal parabolic subgroup in $\Gh$. 
It will be useful to give a more detailed description of $\Qkappa$, the nilradical $\Nkappa$,
the structure of $\MkappaZ$ and the connected component
of $\MkappaZ$.

Let  $F_\kappa$ be the finite abelian group generated $\gamma_\kappa = \exp (\pi i X_\kappa)$. We have the
Levi factor
\begin{equation}
\MkappaZ:= F_\kappa (\MkappaZ)_0 = (\MkappaZ\cap \Kh)(\MkappaZ)_0. 
\end{equation}
We note that by \cite[p. 287]{Va} every $\Ad (m)$, $m\in \MkappaZ$, is in $\Ad (\fm_{h\kappa}^{0\C} )$ but not necessarily,
as the set $F_\kappa$ shows, in $\Ad (\mhZero)$. 
 
Now $F_\kappa$ preserves the decomposition
(\ref{eq-gdecomp}), and as $F_\kappa \subset \Kh$ it also preserves (\ref{eq-phpm}).
Finally
\[\Ad (\gamma_\kappa)|_{\ghEven}=\id\text{ and } \Ad (\gamma_\kappa )|_{\ghOne}=-\id\, .\]
Thus each $\mathfrak{p}^{[\nu ]}_{h\pm}$ is a $K_h\cap \MkappaP$-module.

Consider next the Lie algebras defined in (\ref{eq-decomp}) and their relationship to the Levi factor $\mhZero$. Recall $\bc_\kappa =\exp (\frac{\pi i}{4}Y_\kappa)$ and  $\bC_\kappa =\Ad (c_\kappa)$. Set $\ghbKTwo =\bC_\kappa^{-1}({\khKaC})\cap\gh$.
Then $\mhZero = \mhone \oplus
\mhtwo\oplus \R(X_\kappa)$  where
$\mhone =\mathfrak{l}_2\oplus
\ghone$,
$\mathfrak{l}_2$ is a compact ideal in $\mhZero$ and
$\ghone $ is of Hermitian non-compact type having
$Z^{(1)}_\kappa$
defining the almost complex structure, and $\mhtwo\oplus \R(X_\kappa) = \ghbKTwo $.
Let $M^{(i)}_\kappa$ be the connected
subgroup with Lie algebra $\frak m^{(i)}_\kappa$.  Then we have
$(M^0_\kappa )_o =M^{(1)}_\kappa M^{(2)}_\kappa$ exp $\R (X_\kappa)$ and
$M^0_\kappa =F_\kappa M^{(1)}_\kappa  M^{(2)}_\kappa$exp $\R(X_\kappa) $.

\begin{lemma}\label{le:Complex} The following holds true:
\begin{enumerate}
\item $\ad\,{\ghone}|_{\nhOne}=0$ and $\ad\,{\ghbKTwo}|_{\nhOne}$ is faithful. The orbit $G_h^{(2)}\cdot E_\kappa$ is a self-dual cone.
\item Let $I_o=-\ad (Y_{\kappa})\circ \dthh|_{\nhOne}=\dthh\circ \ad (Y_\kappa)|_{\nhOne}=2\ad\,(Z^{(1)}_\kappa)|_{\nhOne}$. Then $I_o$ defines a complex structure on $\nhOne$. We have
\[\nhOneC (I_o;i)=\nhOneC \cap C_\kappa^{-1}(\phC )\text{ and }\nhOneC (I_o;-i)=\nhOneC \cap C_\kappa^{-1}(\khC)\, .\]
\end{enumerate}
\end{lemma}

Now we have all the notation to give a detailed description of the stabilizer of $E_\kappa\in \DP$ and hence the isotropy of the orbit  $\cO_h(\kappa)$, see \S 1, Chapter 3 and Proposition 8.5, p. 142 in \cite{Sa}.

\begin{theorem}\label{thm:MainS2} Let $\kappa :\SU (1,1)\to \Gh$ be a standard homomorphism.
Then one has the following:
\begin{enumerate}
\item $Z_\kappa \in \khTwo$.
\item $Z_\kappa^{(1)} \in \khZ$.
\item $\ghEven = \gh (\ad Y_\kappa; 0)\oplus \gh (\ad Y_\kappa , 2)\oplus \gh (\ad Y_\kappa ;- 2)
=\gh (\bC_\kappa^4;1)$ as a 
$\MkappaZ \cap \Kh$-module.
\item $\nhTwoC = \bC_\kappa^{-1}(\phPTwo )$.
\item $\ghOne = \gh (\ad Y_\kappa ; 1)\oplus \gh (\ad Y_\kappa ; -1)=\gh (\bC_\kappa^4;-1)$.
 \item $\ad Z_\kappa^{(1)}|_{\ghZ}=\ad Z_h|_{\ghZ}$.  In particular, the $Z_h$-element in the Hermitian type Lie algebra
$\ghZ$ is $Z_\kappa^{(1)}$.
\item If $Z_\kappa^{(1)}\not= 0$ then the stabilizer of $E_\kappa$ in $\Gh$ is $Z_{\Gh}(X_\kappa,Z^{(1)}_\kappa)N_{h\kappa}$. Hence there is a fibration
 
$$M_{h\kappa}^{(1)}/K_h\cap M^{(1)}_{h\kappa} \to \cO_h(\kappa)\to G/Q_{h\kappa}\cong K_h/K_h\cap M^{0}_{h\kappa} $$  
with typical fiber a Hermitian symmetric space. 
\item If $Z_\kappa^{(1)}=0$  then the stabilizer of $E_\kappa$ in $\Gh$ is $Z_{\Gh}(X_\kappa)N_{h\kappa} = Q_{h\kappa}$. Hence the orbit $\cO_h(\kappa) \cong G/Q_{h\kappa}\cong K_h/K_h\cap M^{0}_{h\kappa}$. In particular,
in this case $\cO_h(\kappa)$ is compact.
\end{enumerate}
\end{theorem}

Next consider the Cartan decomposition of $\mhZero$ corresponding to the Cartan involution
$\dot\theta|_{\mhZero}$. First we have
$\mathfrak{g}^{(1)}_{h}=\mathfrak{k}^{(1)}_h\oplus \phone$,
and $(\phone)^{\C}=\phoneP\oplus \phoneM$, with
$\phonePM$ simultaneous $\pm i$ eigenspaces of $\ad Z^{(1)}_\kappa$ and $\ad\, Z_h$.  Moreover we have the
identification $\phPmZ=\phonePM$.

Now consider $\mhtwo$, the other summand of $\mhZero$,  with Cartan
decomposition $\kh\cap \mhtwo \oplus \phtwo$. Note that
\begin{equation}
\mathfrak{k}_{h\kappa}^*=\kh\cap\mhtwo \oplus \khTwo\, .
\end{equation}
\begin{lemma} $Z_\kappa$ is in the center of $\khKa$ and
$\bC_\kappa^{-1}(\khKaC) = \mhtwoC  \oplus \C X_\kappa$.
\end{lemma}

Define now
\[L_\kappa =Z_{\Kh}(Z_\kappa )=Z_{\Kh}(Z_h^{(1)})\, .\]
\begin{lemma}
The Lie algebra $\mathfrak{l}_\kappa$ of $L_\kappa$ decomposes into ideals as
\[\mathfrak{l}_{\kappa}=\kh \cap \mhone\oplus \khKa\]
and $Z_\kappa$ defines an almost complex structure on $\Kh/L_{\kappa}$.
\end{lemma}

This gives yet another fibration in Theorem 1.10 (7),(8) here with base K\"ahlerian, namely
$$L_\kappa/K_h\cap M^{0}_{h\kappa}\to G/Q_{h\kappa}\cong K_h/K_h\cap M^{0}_{h\kappa} \to \Kh/L_{\kappa}.$$
We have now according to \cite{Sa}:
\begin{lemma} $\bC_\kappa\circ\dthh\circ \bC_\kappa^{-1}(\phPOne) =
\khC (\ad Z_\kappa; i)$ as $\Kh\cap \MkappaZ$-modules.
\end{lemma}

For convenience we summarize these various identifications in the next
statement.

\begin{proposition}   We have the following $\Kh\cap \MkappaZ$-isomorphisms:
\begin{enumerate}
\item $\phPmZ \cong\phonePM$.
\item $\phPOne \cong \khC (\ad Z_\kappa; i)$.
\item $\phPmTwo \cong \phtwoC \oplus\C X_\kappa$.
\end{enumerate}
\end{proposition}

\section{Bounded Symmetric Domains: Real Case}\label{section2}\hfill

\noindent
In this section we consider homogeneous real forms of $\DP$, i.e. fixed point sets of anti-holomorphic automorphisms.
We continue to assume that $\Gh$ is simple or of the form  $\Gh=G\times G$ where $G/K$ is
a bounded symmetric domain in $\C^n$ with $G$ simple. Thus either $\gh$ is simple, or $\gh=(\fg,\fg)$ with
$\fg$ simple and $\tau (X,Y)=(Y,X)$. We use \cite{HO96,GO90,GO91} as standard references although
the perspective will be slightly different in this section.  We will present a parallel presentation for the material
for real domains vis \`a vis the complex case. The first observation in the real case will be a replacement for $\GhC$. 
This will be the Lie group $G_c$ to be described shortly.

\subsection{Real Bounded Symmetric Domains and Related Subgroups of $\GhC$}\hfill

 Let $\tau : \Gh \to \Gh$ be a non-trivial involution commuting with $\thh$. Let $\dot \tau : \gh \to \gh$
be the derived involution which then commutes with $\dot\theta_h$. 
Finally, we let $G:=\Ght$. Then $G$ is a connected, reductive subgroup of $\Gh$ having Lie algebra
$\fg:=\ght=\{X\in \gh\mid \dot \tau (X)=X\}$. 
With the usual notation, set $\fq_h:=\{X\in\gh\mid \dot\tau (X)=-X\}$. Then $\gh=\fg \oplus \fq_h$. 

As $\dt $ and $\dot \thh$ commute, it follows that $\dot\theta := \dot\thh|_{\fg}$ defines 
a Cartan involution on $\fg$ and $\fg=\fk\oplus \fp$ with $\fk :=\fg\cap \kh$ and $\fp:=\fg\cap \ph$.
Also, $\fq_h=\fq_{hk}\oplus \fq_{hp}$ with $\fq_{hk}=\fq_h\cap \kh$ and $\fq_{hp}=\fq_h\cap \ph$.

As $\tau$ and $\thh$ commute, $\tau $ induces an involution  on $\Gh/\Kh \simeq \Dh$ denoted \
$\tDh : \Dh\to \Dh$ such that $\tDh (g\cdot z)= \tau (g)\cdot \tDh (z)$. Via the biholomorphism
$\Dh\simeq\DP$,  $\tau $ induces an involution denoted $\sigma_h^+ :\DP\to \DP$ such that
$\sigma_h^+  (g\cdot z)= \tau (g)\cdot \sigma_h^+  (z)$. We assume that $\sigma_h^+ $ is anti-holomorphic, i.e., defines
a conjugation on $\DP\subset \php$. Then $\cD_+:=\DPt{\sigma_h^+ }$ is a totally real submanifold as follows from

\begin{lemma}\label{le-tauZo} $\dot\tau (Z_h)=-Z_h$.
\end{lemma}

Let $K:=G\cap \Kh=G^\theta$. Then $K$ is maximal compact in $G$ with Lie algebra $\fk$. We have
(see Lemma \ref{thm:1.3} for notation)
\begin{equation}\label{DR}
\cD_+\simeq G/K \hookrightarrow \Gh/\Kh \simeq \DP
\end{equation}
is a realization of the Riemannian symmetric space $G/K$ as a bounded totally real domain in $\php$. 

We come to the substitute for $\GhC$. Denote by $\dot \eta :=\dot \sigma_h\circ\dot\tau$ the {\it conjugate
linear} extension of $\dot\tau$ to $\ghC$ and, as usual, $\eta$ the corresponding involution on
$\GhC$. Set $\fg_c=(\ghC)^{\dot\eta}$ and let $G_c$ \footnote{The subscript $c$ will be used for
objects related to this group.} be the corresponding analytic subgroup of $\GhC$. 
By Lemma \ref{le:conn}, $G_c=\GhC^{\eta}$ as $\GhC$ is assumed simply
connected. $\fg_c$ is a real semisimple subalgebra of $\ghC$ which is stable under $\dt$ and $\dot\thh$.
Clearly 
\[\fg=\fg_c^{\dot \tau}=\{X\in \fg_c\mid \dot \tau (X)=X\}=\fg_c\cap \fg_h\]
and with $i\fq_h=\{X\in\fg_c\mid \dot\tau (X)=-X\}$, then 
\[\fg_c=\fg\oplus i\fq_h= (\fk\oplus i\fq_{hp})\oplus (\fp \oplus i\fq_{hk})\]
where $\fk=\fk_h\cap \fg$ and $\fp=\fp_h\cap \fg$.

On the
other hand, the involution $\dot \theta_c:=\dt\circ \dthh|_{\fg_c}$ defines a Cartan involution on
$\fg_c$ with corresponding decomposition 
$\fg_c=\fk_c\oplus \fs_c$
and corresponding Cartan involution $\theta_c$ on $G_c$ (we reserve the notation $\fp_c$ for a parabolic subalgebra).
Then $\fk_c=\fk\oplus i\fq_{hp}$ and $\fs_c=\fp\oplus i\fq_{hk}$.  Notice that $\dot\theta_c$ agrees with the conjugate linear extension of $\dot\theta_h$ restricted to $\fg_c$.
To streamline the notation we let $\fq_c:= i\fq_h$ so that $\fg_c=\fg\oplus\fq_c$. Then $\fq_c=\fq_c\cap \fk_c\oplus \fq_c\cap \fs_c = \fq_{ck} \oplus \fq_{cp}$, with
$\fq_{ck}=i\fq_{hp} $ and $\fq_{cp}=i \fq_{hk} $, i.e., the elliptic and hyperbolic parts have been interchanged.
In the special case that $\cD_+$ is a bounded complex domain, then $G_c=\GC$, the complexification of $G$. 
\begin{lemma} $\sigma_h^+=\dot\eta|_{\DP}$. In particular, $\cD_+=\DP\cap \fg_c$.
\end{lemma}
\begin{proof} Recall that $H_0=-iZ_h$. From Lemma \ref{le-tauZo} and $\dot \eta=\dot \sigma_h\circ\dot\tau$ we get $\dot \eta (H_0)=H_0$. 
As $\fp_{h\pm}=\ghC (\ad H_0;\pm 1)$ and $\khC=\ghC (\ad H_0;0)$ it
follows that $\dot \eta (\phpm)=\phpm $ and $\dot \eta (\khC)=\khC$, similarly $\eta (\PhPm)=\PhPm$ and $\eta (\KhC)=\KhC$. For $g\in G_h$ and $0\in \fp_{h+}$ write $g\cdot 0=Z\in\DP$. Then
$g=\exp (Z) k_\C(g)p_-(g)$ and
\[\tau (g)=\eta (g)=\exp (\dot \eta (Z))\eta (k_\C (g))\eta (p_-(g)).\]
From this the claim follows.
\end{proof}

\subsection{Essential Structure Theory - $\R$ forms}\hfill

We shall refine our choice of Cartan subalgebra $\ch\subset \gh$ to take into account the involution $\dot\tau$ and the associated decomposition $\gh=\fg \oplus \fq_h$. We still require $\ch$ to contain $Z_h$ but now choose the Cartan subalgebra $\ch$ such that $\fbh:=\ch\cap \fq_{hk}$ is maximal abelian in $\fq_{hk}$. Thus all the notation from \S 1.2 remains in force here so will be used freely when applicable. 

Denote by $\Sbh$ the set of roots of $\fbhC$ in $\ghC$. Set $\fa_c:=i\fbh\subset \fs_c$.
Recall that $\sigma_h^+$ is anti-holomorphic and $\dt (Z_h)= -Z_h$.

\begin{lemma} $\fbh$ is maximal abelian in $\fq_{hk}$ and $\fq_h$\,; $\fa_c$ is maximal abelian in $\fs_c$ and in $\fq_{cp}=\fs_c\cap \fq_c$.
\end{lemma}
\begin{proof} The first claim is by construction. Since $Z_h\in\fbh$  one has $\fz_{\ghC}(\fbh^\C)\subset \kh^\C = \fk^\C\oplus\fq_{hk}^\C,$ while $\fq_{cp} = i\fq_{hk}$. Hence  $\fs_c\cap \fz_{\ghC}(\fbh^\C) \subset  i\fq_{hk}$.
\end{proof}

\begin{corollary}\label{co-roots} $\Sigma (\fg_c,\fa_c)$, the set of restricted roots of $\fa_c$ in $\fg_c$, are all restrictions from the complex space $\Sbh$ to the real form $\fa_c$.
\end{corollary}

Let $\Sigma_c:=\Sigma (\fg_c,\fa_c)$. We will view $\Sigma_c$ either as the set of roots of $\fa_c$ in $\fg_c$ or the roots of
$\fbhC$ in $\ghC$ without further comment.
Recall that $H_0=-iZ_h\in\fa_c$. Then $\ad (H_0)$ has three eigenvalues: $0,\pm  1$. We set

\begin{eqnarray}
\mathfrak{l}_c \,&:=& \fg_c(\ad H_0;0)=\khC\cap \fg_c=\fk\oplus i\fq_{hk}\label{de-mPrime}\\
\nP&:=&\fg_c (\ad H_0;1)=\php\cap \fg_c\label{de-nP}\\
\nM&:=&\fg_c (\ad H_0;-1)=\phm\cap \fg_c\label{de-nM}\\
\cP_c&:=&G_c/ L_cN_{-}
\end{eqnarray}
where  $N_{-}$ denotes the analytic subgroup of $G_c$ with Lie algebra   $\nM$ and
$L_c:=Z_{G_c}(H_0)$. Note that $L_c$ has Lie algebra $\fl_c$, but that $L_c$ is not
necessarily connected. For future reference we set $\fl'_c := [\fl_c,\fl_c]$. We
also note that 
\[\fl_c=\mathfrak{k}_{h}^{\C\dot \eta},\quad \fn_{\pm}=\mathfrak{p}_{h\pm}^{\dot \eta}\quad
\text{and} \quad \fl_c\oplus
 \fn_{\pm}=(\fk_h^\C\oplus \fp_{h\pm})^{\dot\eta}.\]

The set $\Sc$ of restricted roots decomposes accordingly into two disjoint sets
\begin{align*}
\Scc&:=\{\alpha \in \Sc\mid {\fg_c}_\alpha\subset \mathfrak{l}^\prime_c,\, \alpha \not\equiv 0 \}\\
&\,=\{\alpha\in\Sc\mid \alpha (H_0)=0, \,
\alpha\not\equiv 0\}\\
&\,=\{\beta|_{\fa_c}\mid \beta\in \Dhc, \beta|_{\fa_c}\not\equiv 0\},
\end{align*}
and
\begin{align*}
\Scn&:=\{\alpha \in \Sc\mid \fg_{c\alpha} \subset \phC\cap \fg_c\}\\
&\,=\{\alpha\in\Sc\mid \alpha (H_0)=\pm 1\}\\
&\,=\{\beta|_{\fa_c}\mid \beta \in \Dhn \}\, .
\end{align*}

If $\alpha \in \Scn$ then $\alpha(H_0)=\pm 1$. We choose the system of positive roots in $\Sc$ such that
\[\ScnP=\{\alpha\in\Sc \mid \alpha (H_0)=1\}=\Dhn^+|_{\fa_c}=\{\beta|_{\fa_c}\mid \fg_{\beta}\subseteq \fp_{h+}\}\]
 and
$\SccP\cup \{0\}= \Dhc^+|_{\fa_c}$.

From $\mathfrak{l}_c  =\fk\oplus i\fq_{hk}$ notice that $K$ is a maximal compact subgroup of  $ (L_c)_o$ and preserves $\fn_{\pm}$;
$\fa_c$ is maximal abelian in $\mathfrak{l}_{cp}$; and $\Scc$ is the set of restricted roots of $\fa_c$ in $\mathfrak{l}_{c}$ .

\begin{lemma} Let $\fm_c$ denote the centralizer of $\fa_c$ in $\fk_c$. Then $\fm_c\subset \fk$.
\end{lemma}
\begin{proof} Since $H_0\in\fa_c$ one  has $\fm_c\subseteq \fz_{\fg_c}(H_0)=\fk\oplus i\fq_{hk}$.
\end{proof}

Denote by $W_{cc}$ the Weyl group generated by the roots in $\Scc$.
\begin{lemma}\label{Wcc}

$W_{cc}=N_K(\fa_c)/Z_K (\fa_c)$ and $W_{cc}(\ScnP)=\ScnP$.

\end{lemma}
 
As $Z_h\in\fbh$ and $\dt (Z_h)= -Z_h$  it follows that $\Dhn$, $\Dhn^+$ and $\Dhc$ are stable under the involution $\dot\tau^\sharp: \beta \mapsto -\beta\circ \dot\tau$. Moreover $\dot\tau^\sharp|_{i\ch^*}=\dot\eta^t |_{i\ch^*}$. Via the identification of $\Sigma (\fg_c,\fa_c)$ with $\Sbh$ we extend $\dot\tau^\sharp$ to $\Sc$.

The dichotomy present in Moore's classification of restricted roots in the complex case is reflected in the next several results.
\begin{lemma}[\cite{GO91}, Lemma 3.2] Let $\gamma \in \Dhn^+$. If $\dot\tau^\sharp(\gamma)\not= \gamma$ then $\gamma$ and $\dot\tau^\sharp(\gamma)$ are strongly orthogonal. In particular, if $\{\gamma_1,\ldots ,\gamma_{r_h}\}$ is a set of strongly orthogonal roots in $\Dhn^+$, then either
$\dot\tau^\sharp(\gamma_j)= \gamma_j$, or $\gamma_j$ and $\dot\tau^\sharp(\gamma_j)$ are strongly orthogonal.
\end{lemma}

\begin{lemma}\label{le-allOrNone} We have either  $\dot\tau^\sharp(\gamma_j)=\gamma_j$ for all $j=1,\ldots ,r_h$, or $\dot\tau^\sharp(\gamma_j)\not=\gamma_j$ for all
$j=1,\ldots, r_h$.
\end{lemma}

This follows from the classification in \ref{se-Classification}. The classification also shows that $\dot\tau^\sharp\gamma_j
\not= \gamma_j$ only for the following five cases:
\begin{enumerate}
\item $\Gh=G\times G$ is not simple and $G/K$ is embedded into $G/K\times \overline{G/K}$ diagonally.
\item $\gh=\so (2,n)$ and $\fg=\so (1,n)$, $n\ge 3$.
\item $\gh = \su (2p,2q)$ and $\fg=\mathfrak{sp} (p,q)$.
\item $\gh = \sp (2n,\R)$ and $\fg=\mathfrak{sp}(n,\C)$.
\item $\gh =\mathfrak{e}_{6(-14)}$ and $\fg=\mathfrak{f}_{4(-20)}$.
\end{enumerate}
As we will see later, the cases $\dot\tau^\sharp(\gamma_j)=\gamma_j$ and $\dot\tau^\sharp(\gamma_j)\not=\gamma_j$ are very different from the point
of view of the underlying geometry.

In the case $\dot\tau^\sharp(\gamma_j)=\gamma_j$ we set $r=r_h$, while in the case $\dot\tau^\sharp(\gamma_j)\not= \gamma_j$ we set $r=r_h/2$.  In the latter case we order the strongly orthogonal roots so that $\dot\tau^\sharp(\gamma_{2j-1})
=\gamma_{2j}$, $j=1,\ldots ,r$, see \cite{{GO}91}, Section 3, for more details and discussion.

\begin{lemma}\label{cor:2.11} Assume that $r=r_h$. Then we can choose $E_j$ and $F_j=\dot\sigma_h(E_j)$ such that  $\dot\tau (E_j)=F_j$,
$\dot\eta (E_j)=E_j$, and $\dot\eta (F_j)=F_j$. So with $X_j = E_j + F_j$ and $Y_j = i(E_j - F_j)$, then  $\dot\eta(X_j )=\dot\tau (X_j)=X_j$ and
$\dot\eta (Y_j)=\dot\tau (Y_j)= -Y_j$. In particular,
$$\fa:=\fa_h = \bigoplus_{j=1}^r \R X_j\subset\fp, \text{  and}$$
$$\fa_{hq}: = \bigoplus_{j=1}^r \R Y_j\subset\fq_{hp}.$$
Moreover since $\rank G/K = \rank \Gh/\Kh=r$, $\fa$ is maximal abelian in $\fp$ and in $\fp_h$, while $\fa_{hq}$ is  maximal abelian in $\fq_{hp}$ and in $\fp_h$. 
\end{lemma}

\begin{proof} As $\dot\tau^\sharp(\gamma_j)=\gamma_j$ it follows that $\dot\eta (\ghC_{\gamma_j})=\ghC_{\gamma_j}$ so
\[\ghC_{\gamma_j}=\ghC_{\gamma_j}\cap \fg_c\oplus i ( \ghC_{\gamma_j}\cap \fg_c).\]
Thus we can  choose $E_j\in \fg_{c\gamma_j}$ such that $-B_c(E_j,\dot\theta_c(E_j))=1$, where $B_c$ denotes the Killing form on $\fg_c$. Then $[E_j,-\dot\theta_c(E_j)]=H_j$. Notice that $-\dot\theta_c(E_j)=\dot\tau(E_j)$ as $E_j\in \php$. Furthermore,
$E_j\in \fg_c$ and hence $E_j=\dot\sigma_h(\dot\tau (E_j))$ or $\dot\tau (E_j)=\dot\sigma_h(E_j)=F_j$.
\end{proof}

Similarly in the other case we have

\begin{lemma}\label{le-rnot} Assume that $r\not= r_h$. Then we can choose $E_j$ and $F_j=\dot\sigma_h(E_j)$
such that $\dot\tau (E_{2j-1})=
F_{2j}$,  and  $\dot\tau (E_{2j})=F_{2j-1}$  for $1\le j \le r$, hence $\dot\eta (E_{2j-1})= E_{2j}$ 
and $\dot\eta (F_{2j -1}) = F_{2j}$. So with $X_l = E_l + F_l$ and $Y_l = i(E_l - F_l)$, then
$\dot\tau (X_{2j-1})=X_{2j} = \dot\eta  (X_{2j-1})$ while  $\dot\tau (Y_{2j-1})=  - Y_{2j} = \dot\eta (Y_{2j-1})$. One has $\fa_h = \fa \oplus \fa_h^q$ with
\begin{equation}\label{def-aq}
\fa =\fa_h\cap \fg=\bigoplus_{j=1}^r \R (X_{2j-1}+ X_{2j})\quad \text{ and }\quad \fa_h^q=
\bigoplus_{j=1}^r \R (X_{2j-1}-X_{2j})\subset \fq_{hp}\, .
\end{equation}
Moreover, $\fa$ is maximal abelian in $\fp$ and $\rank G/K=\frac{1}{2}\rank \Gh/\Kh=r$.
\end{lemma}

To allow for uniform treatment of the cases we introduce the notation
$E_j^\prime =E_j$, $F_j^\prime = F_j$,  $X^\prime_j=X_j$ etc. in case $r=r_h$, and
$E^\prime_j=E_{2j-1}+E_{2j}$, $F^\prime_j=F_{2j-1}+F_{2j}$, $X^\prime_j=X_{2j-1}+X_{2j}$, etc. in
case $r\not= r_h$. Then in all cases we have $\dot\tau(E_j^\prime) = F_j^\prime$ and
\[\fa=\bigoplus_{j=1}^r \R X^\prime_j\, .\]
The order in $\fa^*$ is obtained from the lexicographic order with respect to the basis $\{X^\prime_1,\ldots ,X^\prime_r\}$.

Similarly we will need an extension of this notation to include subsets and signs. So for $I\subset \{1,\ldots ,r\}$ and $\epsilon\in\{-1,1\}^{\# I}$, if $r=r_h$ set $I^\prime = I$ and $\epsilon^\prime=\epsilon$; 
otherwise, set $I^\prime =\{2j-1,2j\mid j\in I\}= (2I-1)\cup 2I$ with $\epsilon^\prime_{2j-1}=\epsilon^\prime_{2j}=\epsilon_j$. Then we will have $E^\prime(I^\prime ,\epsilon^\prime )$ equal to either $E(I,\epsilon)$ in the first case, and to $E(2I - 1,\epsilon) + E(2I,\epsilon)$ in the second case. 

\begin{remark}\label{re:Realization}
We mention that all the
classical irreducible Riemannian symmetric
spaces, with a possible extension by the abelian group $\R^+=\{t\in \R\mid t>0\}$, arise in this way as a real form of a bounded
symmetric domain in $\C^n$, see Tables 3 and 4.
The Riemannian symmetric spaces that do not occur this way are those that correspond to the symmetric pairs: $(\fe_{6(2)}, \su (6) \times \su(2))$, $(\fe_{6(6)} , \sp (4))$, $(\e_{7(7)},\su(8))$,
$(\e_{7(-5)},\so (12) \times \su (2))$, $(\e_{8(8)},\so (16))$, $(\e_{8(-24)},\e_{7} \times \su (2))$, $(\ff_{4(4)}, \sp (3) \times \su(2))$,
and $(\fg_{2(2)},\su(2) \times \su(2))$, namely, those with a quaternionic K\"ahler metric or associated to a split exceptional group.

The extra factor $\R^+$ occurs in the cases exactly where
$\DP\simeq \R^k+i\Omega$ is a tube type domain and (up to finite coverings)  $G\simeq \GL (\Omega)_o$ is the
automorphism group of the symmetric cone $\Omega$; moreover, here $r = r_h$. These are not all the tube domains, but those for which $\gh = \fg_c$. The simplest case is when $\Gh=\SU (1,1)$ and
$G=\{\exp t X_1\mid t\in \R\}$ (see (\ref{def-X1Y1}) for the notation). In this  case $K$ is trivial and
$\exp tX_1$ acts on $\cD_+ $ by 
\[\exp tX_1\cdot x=\frac{x+\tanh (t)}{x\tanh (t)+1}\]
 according to
(\ref{eq-actXj}).  In the general case the $\R^+$ factor is $\exp \R H_0$ where
$H_0=X_1+\cdots +X_r$. The element $H_0 (=-iZ_h$) is centralised by $K$. Let $\fa =\bigoplus_{j=1}^r \R X_j$, which by Lemma \ref{cor:2.11} 
is maximal abelian in $\fp$, and set
$A=\exp \fa$. Then $G=KAK$. It follows that the action of the $\R^+$ is given by
\[\exp (tH_0)\cdot (ka\cdot 0)= k\exp \sum (t_j+t)X_j\cdot 0= k \cdot \sum_{j=1}^r \tanh (t_j+t)E_j .\]

The Lie algebra $\fg$ is simple except for the aforementioned tube type cases and the case $\gh=\so (2,p+q)$, $\fg =\so (1,p)\times \so (1,q)$, $p,q\ge 2$. If $G/K$ itself is a bounded symmetric domain in $\C^n$, then $\Gh=G\times G$
and $G/K$ is embedded diagonally into
$\DP\times \overline{\DP}$ ($ \overline{\DP}$ the conjugate structure). This is the only case where $\Gh$ is not simple.

Non-uniqueness of the bounded realization occurs  if $\fg=\so (1,p )$, then one can take $\gh = \su (1,p)$ or
$\gh = \so (2,p+1)$; while for $\fg = \sp (2,2)$ one has the choices $\gh = \su (4,4)$ or $\mathfrak{e}_{6(-14)}$.

\end{remark}

\subsection{Boundary Orbits of $\cD$}\hfill

The $G$ orbit of the basepoint, $G\cdot x_o$, is $\cD$, an open domain in $\cP_c =G_c/ L_cN_{-}$. On the other hand, the Bruhat cell $N_+\cdot x_o$ is open and dense in $\cP_c$. As in the complex case, by means of  $\log$ one obtains an analytic isomorphism $N_+\cdot x_o\simeq N_+ \cong\nP$,  
\begin{equation}
\cD \simeq G/K \simeq \cD_+ \subset \nP, 
\end{equation}
the Harish-Chandra bounded realization of $\cD$. 

There are two possible ways to consider the closure of $\cD$ and the corresponding boundary 
orbits: we can consider the closure in the generalized flag manifold $\cP_c $, 
or in the open dense set $\exp \fn_+\cdot x_0 \simeq \fn_+$. As for the complex case,
Lemma \ref{lem:ClosureSame}, the two agree. 

\begin{lemma} Denote by $c({\cD})$ the closure of $\cD$ in $\cP_c$. Then $c({\cD})$ is also the closure of $\cD$ in 
$\GhC/\KhC\Phm $ and the closure of $\exp \cD_+\cdot x_0$ in $\exp \fn_+\cdot x_0$. In
particular, the action of $G$ on $\cD_+$ extends
to the boundary of $ \cD_+$.
\end{lemma}
\begin{proof} This follows from the fact that $\cP_c$ is compact and hence closed in $\GhC/\KhC\Phm$.
\end{proof}

\begin{remark} The above statement is also a consequence of the fact that $\cP_c=(\GhC/\KhC\Phm)^{\eta}$.
The rest then follows from  Lemma \ref{lem:ClosureSame} by taking $\eta$-fixed points.
\end{remark} 

Denote by $\partial \cD =c({\cD})\setminus \cD$ the topological boundary of $\cD$.

\begin{proposition} Let $I\subseteq \{1,\ldots ,r\}$ and let $\epsilon \in \{-1,1\}^{\# I}$.
\begin{enumerate}
\item If $r=r_h$, then $\dot\eta (E (I,\epsilon ))=E(I,\epsilon)$ and $\dot\eta (\cO_h(I,\epsilon))=\cO_h (I,\epsilon)$.
\item If $r\not=r_h$, then $\dot\eta (E (2I-1,\epsilon ))=E(2I,\epsilon)$ and $\dot\eta (\cO_h (2I-1,\epsilon))=
\cO_h(2I,\epsilon)$.
\item Uniformly in all cases we have $\dot\eta (E^\prime(I^\prime ,\epsilon^\prime ))=E^\prime(I^\prime ,\epsilon^\prime)$ and
$\dot\eta (\cO_h(I^\prime ,\epsilon^\prime))=\cO_h(I^\prime ,\epsilon^\prime).$
\end{enumerate}
\end{proposition}
\begin{proof} This follows from Lemma \ref{cor:2.11} and Lemma \ref{le-rnot} as $\tau (\Gh)=\Gh$.
\end{proof}

Clearly $\partial \cD = (\partial \DP)^{\dot\eta}$ and each $\cO_h (I^\prime ,\epsilon^\prime)^{\dot\eta}$ is $G$-invariant although the orbits are yet to be determined. However from Theorem \ref{th-boundh} we can conclude

\begin{lemma}\hfill
\samepage
\begin{enumerate}
\item If $r=r_h$, then $\displaystyle \partial \cD =\dot\bigcup_{b=1}^r \cO_h(b)^{\dot\eta}$.
\item If $r\not=r_h$ then, $\displaystyle \partial \cD = \dot \bigcup_{b=1}^r\cO_h(2b)^{\dot\eta}$.
\end{enumerate}
\end{lemma}

Indeed more can be said in both cases, but we start with some simple observations about the strongly 
orthogonal roots $\gamma_j$ and the maximal abelian subspace $\fa_c$. 

In the case $r=r_h$ we have
$\gamma_j\circ \tau = -\gamma_j$ and $\ft_h=\bigoplus \R H_j \subset \fa_c$ (see (\ref{def-th}) for
notation). Let $\alpha_j=\gamma_j|_{\fa_c}$. Then   $\{\alpha_1 ,\ldots ,\alpha_r\}$ is a maximal set
of strongly orthogonal roots in $\Sigma_{cn}$.

In the case $r=r_h/2$ we have $\dim \ft_h\cap \fq_c= \frac{1}{2}\dim \ft_h$ and $r_h$ is even. We
let $\alpha_j = \gamma_{2j}|_{\fa_c} =\gamma_{2j-1}|_{\fa_c}$, $j=1,\ldots ,r$.  Then the set
$\{\alpha_1 ,\ldots , \alpha_r \}$ is a maximal set of strongly orthogonal roots
in $\Sigma_{cn}$.

\begin{lemma}[\cite{NO00} Lemma 2.23]\label{transitive}Let $\{\beta_1,\dots, \beta_r\}\subset \ScnP$ be a maximal set of strongly orthogonal roots. Then given a permutation $\beta_i \to \beta_{\sigma(i)}$ there is an element $k\in K$ that implements it, in particular $Ad (k)(H_{\beta_i}) = H_{\beta_{\sigma(i)}}$.
\end{lemma}

Suppose that $r=r_h$. Now $\gamma_j\in \ScnP$ and $E_j\in  \fg_{c\gamma_j}$. Then $E(b) =E_1 +\cdots + E_b$ with $E_j\in \fg_{c\gamma_j}$.  It follows from Lemma \ref{transitive} that $\dim_{\R} \fg_{c\gamma_j}$ is independent of $j$, so denote it by $a$. Also $\fg_{2\gamma_j}=0$ as can be seen from Lemma \ref{transitive}  and Moore's Theorem. If $a>1$ then $Z_{K_c}(\fa_c)$ acts transitively on the unit sphere in $\fg_{c\gamma_j}$ ( \cite{Wa}  Theorem 8.11.3, p. 265). But in this case the unit sphere is connected so $Z_{K_c}(\fa_c)_o$ acts transitively. 
We also know from Lemma 2.5 that the Lie algebra of $Z_{K_c}(\fa_c)_o$ is contained in $\fk$. Hence $Z_{K_c}(\fa_c)_o\subset K$.
It follows that $E_{j} $ and $-E_{j}$ are conjugate under $Z_K(\fa_c)\subset G$.
 Now apply this argument to each  of the analytic subgroups of $G_c$ corresponding
to the Lie algebra generated by $\R E_j\oplus \R H_j\oplus \R E_{-j}$ to see that we can find $k_j\in G_j$
such that $\Ad (k_j)E_j = -E_j$. But the groups $G_i$ and $G_j$ commute if $i\not= j$, thus with $k=\prod_j^b k_j^{(1+\epsilon_j)/2}$ we have $\Ad (k) E(b,\epsilon)=E_b$.

\begin{lemma}\label{le:MinusOne} The following are equivalent:
\begin{enumerate}
\item there exists $m\in N_K(\fa)$ such that $\Ad (m)|_{\fa}=-1$;
\item there exists $m\in K$ such that
$\Ad (m)E_j^\prime=-E_j^\prime $, $j=1, \ldots ,r$; 
\item there exists $m\in K$ such that $\Ad (m)F_j^\prime =-F_j^\prime$, $j=1,\ldots ,r$.
\end{enumerate}
\end{lemma}
\begin{proof} As noted following Lemma \ref{le-rnot}, $\dot\tau (E_j^\prime  ) = F_j^\prime$. Since $\tau|_{K}=\id$, it follows that (2) and (3) are equivalent.
Assume that there exists $m\in K$ such that $\Ad (m)|_{\fa}=-1$. Then $\Ad (m)(E_j^\prime +F_j^\prime)=
- E_j^\prime - F_j^\prime$, $j=1,\ldots ,r$. As $K\subset L_c$ we have
$\Ad (m)\fn_\pm =\fn_\pm$. Hence $\Ad (m)E_j^\prime = -E_j^\prime$ and $\Ad (m)F_j^\prime = -F_j^\prime$.
On the other hand, if (2) and (3) hold then, as $X_j^\prime = E_j^\prime + F_j^\prime$, $\Ad (m)X_j^\prime = -X_j^\prime$.
Since $\fa = \bigoplus_j \R X_j^\prime$, the claim follows.
\end{proof}
 
 \begin{remark}  It follows from Lemma \ref{transitive} that it is
enough to assume that (2) and (3) above hold for one $j$.
\end{remark}

\begin{corollary}\label{co:MinusOne}  Assume that $r=r_h$. If $-1$ is not in the Weyl group $W= N_K(\fa )/Z_K (\fa)$, then $E(b,\epsilon)$ is
not conjugate to $E(b,\epsilon^\prime)$ if
$\epsilon\not=\epsilon^\prime$.
\end{corollary}

\begin{theorem}\label{th-Orbit1}
Assume that $r=r_h$ and let $1\le b\le r$.
\begin{enumerate}
\item If $-1\in W$ then $(\cO_h(b))^{\dot\eta}=G\cdot E(b)=: \cO(b)$ is one $G$-orbit and
\[\partial \cD =\dot \bigcup_{b=1}^r \cO(b)\, .\]
\item If $-1\not\in W$ then $\displaystyle (\cO_h(b))^{\dot\eta}=\dot \bigcup_{\epsilon\in \{-1,1\}^b}G\cdot E(b,\epsilon)$ and
\[\partial\cD =\dot\bigcup_{b=1}^r\dot\bigcup_{\epsilon\in\{-1,1\}^b} \cO(b,\epsilon)\quad \text{ with }\quad
\cO(b,\epsilon):=G\cdot E(b,\epsilon)\, .\]
\end{enumerate}
\end{theorem}
\begin{proof} Let $z\in (\cO_h(b))^{\dot\eta}\subseteq \partial \cD$. Using the familiar argument we
can choose $k_j\in K$ and $a_j\in A$ such that
$k_ja_j\cdot 0\to z$. Again, $k_j$ has a convergent subsequence, so we can assume that $k_j\to k\in K$.
Replace $z$ by $w=k^{-1}z$
in the same $G$-orbit. Write $a_j=\exp \sum_{\nu =1}^r t_{\nu, j}X_\nu$. Then
\[a_j\cdot 0=\sum_{j=1}^r\tanh (t_{\nu ,j})E_j\, .\]
As $a_j\cdot 0\to w$ it follows that there exists a set $I$ such that $\tanh (t_{\nu,j})\to \epsilon_\nu\in \{-1,1\}$ for $\nu \in I$ and
$\tanh (t_{\nu,j})\to x_\nu\in (-1,1)$ for $\nu\not\in I$. Hence

\[w=\sum_{\nu\in I}\epsilon_\nu E_\nu +\sum_{\nu\not\in I}x_\nu E_\nu\, .\]

If $t_\nu\in\R$ is so that $x_\nu=\tanh (t_\nu)$ then $\exp(-\sum_{\nu\not\in I}t_\nu X_\nu)\cdot \sum_{\nu\not\in I}x_\nu E_\nu
=0$ so we can assume that $w= \sum_{\nu\in I}\epsilon_\nu E_\nu $. As $E_\nu\in \fg_{c\gamma_\nu}$ from Lemma \ref{transitive} we can assume that there exists a $b$ and $\epsilon\in \{-1,1\}^b$ such that
$w=E(b,\epsilon)$. The claim now follows from Lemma \ref{le:MinusOne} and Corollary \ref{co:MinusOne}.
\end{proof}

\begin{theorem}\label{th:Orb}
Assume that $r\not=r_h$. Let $1\le b\le r$. Then $\cO(2b)^{\dot\eta}=G\cdot E(2b) = G\cdot E^\prime(b)$ is one $G$-orbit. In particular,
with $\cO(b)=G\cdot E^\prime(b)$ we have
\[\partial \cD =\dot\bigcup_{b=1}^r \cO(b)\, .\]
\end{theorem}
\begin{proof} Let $z\in (\cO(2b))^{\dot\eta}$. By replacing  $X_{2j-1} + X_{2j}$ with $X_j^\prime$ we see as above that
we can assume that $z=\sum_{\nu =1}^b \epsilon_\nu E_\nu^\prime$ for some $b$.
As before let
$\alpha_i=\gamma_{2i-1}|_{\fa_c}$. Then $\alpha_i\in \ScnP$ and 
\[\alpha_i|_{\fa_c}=\gamma_{2i}|_{\fa_c}=\dot\tau^\sharp(\gamma_{2i-1})|_{\fa_c}.\]
It follows that $\dim \fg_{c\alpha_i}\ge 2$. We also have $2\alpha_i\not\in\Sc$. Thus $Z_K(\fa_c)$ acts transitively on spheres in
$\fg_{c\alpha_j}$ which implies that $E_i^\prime$ and $-E_i^\prime$, which are both in $\fg_{c\alpha_i}$ are conjugate via $Z_K(\fa_c)$. Thus we can take $\epsilon_i=1$ for all $i$.
The roots $\alpha_i$ and $\alpha_j$ are conjugate by
$s_{(\alpha_i-\alpha_j)/2}\in W_{cc}$ and $E_\nu^\prime\in \fg_{\alpha_\nu}$. It follows that we can assume that $J=\{1,\ldots ,b\}$ for some $b\le r$.
\end{proof}

 \subsection{Isotropy of $E(b,\epsilon)$}\hfill

In this section we describe the stabilizer in $G$ of $E(b,\epsilon)$, respectively $E(2b)$, on the boundary of
$\cD$. On the way we give some extra information about the structure
of each part in the stabilizer. Our notation for subgroups of $G$ will be the same as that used
for $\Gh$ except we drop the subscript \lq\lq h\rq\rq and $L=Z_G(H_0)$. Our standard homomorphism will
always been assumed to be of the form $\kappa_{I,\epsilon}$ for $I=\{1,\ldots ,b\}\subset \{1,\ldots ,r\}$.  We
define $I^\prime$ as in the  earlier subsection and then write
 $\kappa$ instead of $\kappa_{I^\prime,\epsilon}$ wherever the exact form does not matter. As before,
we write $E_\kappa$, $H_\kappa$, $X_\kappa$, $\cO(\kappa )$ etc. for $E(b ,\epsilon)$, $H(b,\epsilon)$,
$X (b,\epsilon)$, $\cO(I^\prime ,\epsilon)$. We have $\dot\eta (E_\kappa )=E_\kappa$.
Hence if $\GhC^{E_\kappa}$ is the stabilizer of $E_\kappa$ in $\GhC$ then
the stabilizer $G_c^{E_\kappa}$ of $E_\kappa$ in $G_c$ is $G_c^{E_\kappa}=\left(\GhC^{E_\kappa}\right)^\eta$
and the stabilizer $G^{E_\kappa}$ in $G$ is $(G_c^{E_\kappa})^\tau$. Same argument holds also for
the Lie algebra of the stabilizers.

\begin{Basic Example} $\SU (1,1)$ - cont.\hfill

We return to the prototype example, \S 1.3, and introduce an anti-holomorphic involution. Consider the map $\dtau_1 : \sl (2,\C)\to \sl (2,\C)$ given by the matrix multiplication
\[\dot\tau_1\left(\begin{pmatrix} x & y \\ z & -x\end{pmatrix}\right) =X_1 \begin{pmatrix} x & y \\ z & -x\end{pmatrix}  X_1
= \begin{pmatrix} -x & z \\ y & x\end{pmatrix} \text{ where } X_1=\begin{pmatrix} 0 & 1\\ 1 & 0\end{pmatrix}.\]
Clearly $\dot\tau_1$ is complex linear, whereas for $X\in \su (1,1)$  one has $\dot \tau_1 (X)=\overline{X}$, in particular $\dot\tau_1(\su (1,1))=\su (1,1)$. Recall that the conjugate linear extension of $\dtau_1$ from $\su(1,1)$ to $\sl(2,\C) (= \su(1,1)^\C)$ is denoted  $\dot\eta_1$,  and so on $\su(1,1)$ is also given by complex conjugation, as is $\eta_1$ on $SU(1,1)$. For the involution $\dtau_1$, the Lie subalgebra of $\sl(2,\C)$ denoted $\fg_c$ is $\sl (2,\R)\cong \su(1,1)$. Thus for the subgroup $G\subset SU(1,1)\cap SL(2,\R)$ we have
 \[G=\left\{\left. \begin{pmatrix} \cosh (t) & \sinh (t)\\ \sinh (t)&\cosh(t)\end{pmatrix}\, \right|\, t\in \R\right\}
 \text{ and }\]
 \[ \cD_+ =\DPt{\sigma_h^+ }= (-1,1)\subset \DP = D_1\subset \php.\]

Since $\gh \cong \fg_c$ we know (cf. Table 4) that $\fg$ has an $\R$-factor and that $r = r_h$ (cf.  Lemma 2.10). As regards compatibility of the involutions, 
\begin{equation}\label{eq-kappatau}
\kappa \circ  \tau_1= \tau\circ  \kappa\text{ and  }
\kappa \circ \theta_1=\theta \circ \kappa\,,
\end{equation}
consequently\footnote{We remark that the results in this subsection are valid for all standard homomorphisms 
satisfying (\ref{eq-kappatau}).}
$$\kappa\circ \eta_1=\eta \circ \kappa.$$  
Moreover, with $\kpr:=
\dot\kappa_{I^\prime ,\epsilon^\prime}$ we similarly have $\kpr \circ \dot\tau_1=\dot\tau \circ \kpr$.
\end{Basic Example}

Earlier we recalled the decomposition obtained from $\pi_\kappa :=\ad \circ \dot\kappa $:
\begin{equation}\label{eq:g0etc}
\gh = \ghZ\oplus \ghOne\oplus \ghTwo\quad \text{and}\quad \ghC = \ghCZ\oplus \ghCOne\oplus \ghCTwo.
\end{equation} 

\begin{lemma} If $\pi$ is a finite dimensional representation of $\sl (2,\C)$ then $\pi$ and $\pi\circ \dot\tau$ are equivalent.
\end{lemma}
\begin{proof} This is well known. We assume that $\pi$ is irreducible, then $\pi$ is uniquely determined by its dimension. As the
dimension of $\pi$ and $\pi\circ \dot\tau $ are equal and $\pi\circ \dot\tau$ is irreducible the result follows.
\end{proof}

It follows from this Lemma that the decompositions in (\ref{eq:g0etc}) are preserved under $\dot\tau $ and
$\dot\eta$. In particular, where the superscript refers to $\dot\eta$-fixed points, respectively intersection:
\[\fg_c=\fg_c^{[0]}\oplus \fg_c^{[1]}\oplus \fg_c^{[2]}, \,\, 
\fg = \fg^{[0]}\oplus \fg^{[1]}\oplus \fg^{[2]} \text{ and }
\fg_h=  \fg^{[0]}\oplus \fg^{[1]}\oplus \fg^{[2]}\oplus  \fq_h^{[0]}\oplus \fq_h^{[1]}\oplus \fq_h^{[2]}.\]

\begin{remark} Recall that we have defined $\kappa$ such that it defines a homomorphism $\su(1,1)\to
\fg_h$. But as pointed out in \ref{eq:SLkappa} 
one can, by extending $\kappa$ to $\sl (2,\C)$ and then restrict to $\sl (2,\R)$, view $\kappa$ as a
homomorphism $\sl (2,\R)$ into $\fg_c$. Then the first decomposition in (\ref{eq:g0etc}) is the isotypic decomposition of
the representation $\ad_{\fg^c} \circ \kappa$ of $\sl (2,\R)$. The second decomposition is then obtained
by taking the $\dot\tau$-fixed point in each of the spaces $\fg_c^{[j]}$.  We will discuss that in more details
in the next section. Note that the spaces $\fg^{[j]}$, $j=1,2$,
are not necessarily $\kappa (\sl (2,\R))$-invariant.
\end{remark} 
  
As $\dtau (X_\kappa)=\dot\eta (X_\kappa)=X_\kappa$ it follows that the eigenspaces of
$\ad X_\kappa$ are $\dtau$ and $\dot\eta$ stable and compatible with the decomposition
$\gh= \fg\oplus \fq_h$
and $\ghC=\fg_c\oplus i\fg_c$. In short, all the essential structure from the previous sections
is invariant under $\dtau$ and $\dot\eta$. In particular,
\begin{equation}\label{eq:MoreDecomp}
\mhZero =\mZero \oplus \mhZero\cap \fq_h,\,\, \nhOne= \nOne\oplus \nhOne\cap \fq_h, \,\,
\nhTwo = \nTwo \oplus \nhTwo\cap \fq_h , \text{ and }\qk=\qRk \oplus \qk\cap \fq_h\, .
\end{equation}

Let $H_0$ and $Z_h $ be as before. Let $H_\kappa = \kappa (H_1)$ and
$H_\kappa^{(1)}=H_0-\frac{1}{2}H_\kappa$ and note that 
$H_\kappa ,H_\kappa^{(1)}\in \fa_c\subset \fq_c\cap \fs_c$.
Complexifying the decomposition in (\ref{eq:MoreDecomp}) and then taking $\dot\eta$ and $\dot\tau$ fixed points
we get
\[\fn_{c\kappa}^j=\fn_{h\kappa}^{j\C}\cap \fg_c=\fg_c(\ad X_\kappa ; j),\quad j=1,2,\]
and
\[\fn_\kappa^j=\fn_{h\kappa}^j\cap \fg=\fg_c(\ad X_\kappa ; j)^{\dot\tau}=\fg (\ad X_\kappa ; j),\quad j=1,2. \]

For the complexification of the Levi factor of the maximal parabolic subalgebra $\fq_{h\kappa}$ and
its intersection with $\fg_c$ we also have with $\fl_{c\kappa} =\fz_{\fg_c}(X_\kappa)$ and with the obvious 
notation: 
\[\fp_{c\kappa} = \fl_{c\kappa}\oplus \fn_{c\kappa}= (\fl_{c\kappa}^{(1)}\oplus \fl_{c\kappa}^{(2)})\oplus \R X_\kappa
\oplus (\fn_{c\kappa}^1\oplus \fn_{c\kappa}^2) \]
and
\[\fp_{\kappa} = \fl_{\kappa}\oplus \fn_{\kappa}= (\fl_{\kappa}^{(1)}\oplus \fl_{\kappa}^{(2)})\oplus \R X_\kappa
\oplus (\fn_\kappa^1\oplus \fn_\kappa^2) \]
semidirect products.

Let $L_{c\kappa} =Z_{G_c}(X_\kappa)=L^{(1)}_{c\kappa}L^{(2)}_{c\kappa}A_\kappa$ where   $L^{(1)}_{c\kappa}$ is the 
analytic subgroup of $G_c$ with Lie algebra $\fl_{c\kappa}^{(1)}$,
$L^{(2)}_{c\kappa}=Z_{G_c}(X_\kappa, H_0),$ and $A_\kappa = \exp \R X_\kappa$.
We use analogous notation for $\fg$ and $G$ 
dropping the index $c$. Up to connected components for $L^{(1)}_\kappa$, those
Lie algebras, respectively Lie groups, 
are obtained by taking $\dot\tau$, respectively $\tau$ fixed points. Finally we
let 
\[P_{c\kappa}=N_{G_c}(\fn_{c\kappa}) =L_{c\kappa}^0A_\kappa N_{c\kappa} \quad\text{and}\quad 
P_\kappa = N_{G}(\fn_\kappa )
=L_\kappa^0 A N_\kappa.\]

Theorem  \ref{thm:MainS2}, parts (7) and (8) now imply:

\begin{lemma}\label{thm:MainR} The following holds true
\begin{enumerate}
\item  $P_{c\kappa}$ is a maximal parabolic subgroup of $G_c$.
\item If $H^{(1)}_\kappa\not= 0$ then $H^{(1)}_\kappa$ is central in
$\fl_\kappa^{(1)}\cap \fs_c$ and $L_{c\kappa}^{(1)}/L_{\kappa}$ is, up to compact factors, a split-Hermitian
symmetric space.
\item $P_\kappa$ is a parabolic subgroup in $G$.
\item If $H^{(1)}\not=0$ then $L_\kappa^{(2)}/K\cap L^{(2)}_\kappa$ is the
fixed point set of the conjugation $\eta$ in the Hermitian symmetric space
$M_{h\kappa}^{(1)}/K_h\cap M_{h\kappa}^{(1)}$ and we have a fibration
\[L^{(1)}_\kappa /K\cap L^{(1)}_\kappa \to \cO (\kappa ) \to
K/K\cap L^{(2)}_\kappa .\]
\item If $H^{(1)}_\kappa = 0$ then the stabilizer of $E_\kappa$ in $G$ is $P_\kappa$ and
$\cO(\kappa)=G/P_\kappa = K/K\cap L^{(2)}$ is a compact symmetric $R$-space.
\end{enumerate}
\end{lemma}
 
  \section{Finer Structure of $Q_\kappa$}

\noindent
In this section we discuss the finer structure of the stabilizer of $E_\kappa$. This material will not be used
in this article but we still think it is worth including.
Recall from \S1.5 that $\ghZ =\mathfrak{z}_{\gh}( \dot\kappa (\su(1,1) )$ is a subalgebra which has ideal  $\ghKappa=\bigoplus_{\gh_j\subseteq\ghZ,\, j\ge 1}\gh_j$.

 \begin{lemma} Let $V\subset \fg_h^{[k]}$ be an irreducible $\ghKappa$-module. Then exactly one of the following holds:
\begin{enumerate}
\item  $\dtau (V) = V$ and $\dtau|_V=\id$. In this case $V\subset \fg$ and the action of $\ghKappa$ is trivial.
\item   $\dtau (V) = V$ and $\dtau|_V=-\id$. In this case $V\subset \fq_h$ and the action of $\ghKappa$ is trivial..
\item $\dtau (V)=V$ and $\dtau|_V\not= \pm \id$. Then $\dim V> 1$ and $\dim V\cap \fg =1$. If $\dim V=2$, then
$V\cap \fg\subset \fn^{1}$ or $V\cap \fg\subset \fn^{-1}$. If $\dim V=3$, then  $V\cap \fg \subset
\gZ \cap \fg $.  
\item $\dtau (V)\not= V$. Then $\dim V>1$ and $V\cap \dtau (V)=\{0\}$ and $\dtau|_{V\oplus \dtau (V)}: V\oplus \dtau (V)\to V\oplus \dtau (V)$ is
given by $\dtau (X,Y)=(\dtau (Y),\dtau (X))$ and $(V\oplus \dtau (V))\cap \fg =\{X+\dtau (X)\mid X\in V\}$ is three dimensional.
\end{enumerate}
\end{lemma}
\begin{proof} It is clear that exactly one of the conditions (1) to (4) must hold. In the case where $\dim V=2$ or $\dim V=3$ the
action of $\ad |_{\ghKappa }$ and $\ad\circ \dtau|_{\ghKappa}$ on $V$ are different as $e$ and $f$ act differently
on $\R^2$ and $\sl (2,\R)$. Thus, if $\tau|_V=\pm \id$, we must have that the action of $\ghKappa$ is trivial
as $\ad\circ \dtau_1=\tau\circ \ad$. Then (1) and (2) follow.

Assume that
$\tau (V)= V$ and $\tau|_V\not=\pm \id$. Then clearly $\dim V>1$. Assume that $\dim V=2$. Then $V=\Im (\id +\dot \tau)\oplus
\Im (\id -\tau)=V(\dtau, 1)\oplus V(\dtau ,-1)$ and each of the eigenspaces is one dimensional. As $\dtau_1(X_1)=X_1$,
$\R^2 =\R^2(X_1,1)\oplus \R^2(X_1,-1)$. Since
$\dot \kappa\circ\dot \tau_1=\dot \tau \circ\dot\kappa$, it follows that $\ad X_\kappa|_{V(\dot\tau ,1)}=\pm 1$. If $\dim V=3$, then the
action is the standard $\su (1,1)$ action on its Lie algebra and 
\[\su (1,1)\cap \fg =\R X_1=\su (1,1)(\ad X_1,0)=\su (1,1)^{\dot\tau_1}.\]

For (4) we note that $V\cap \dtau (V)$ is invariant. As $V$ is assumed irreducible, we either have $V=\dtau( V)$ or
$V\cap \dtau (V)=\{0\}$. The rest is now obvious.
 
\end{proof}
\begin{lemma} Assume that (3) above holds and $\dim V=2$.  Then $\theta (V)\cap V=\{0\}$. Furthermore,
\begin{enumerate}
\item $\theta (V)$ is $\ghKappa$-stable.
\item $\theta (V(\dtau ,\pm 1))=\theta( V)(\dtau ,\pm 1)$.
\item $\theta (V(\ad X_\kappa ,\pm 1))=\theta (V)(\ad X_\kappa , \mp 1)$.
\item If $0\not= X\in V(\ad X_\kappa, \pm 1)$ then $\theta (X)\in V(\ad X_\kappa , \mp 1)$
and $[X,\theta X]\in \fm^0\cap \fp$.
\end{enumerate}
\end{lemma}
\begin{proof} Fix $v\in V(\dtau ,1)$. If $[X_\kappa,v]=v$ then $[X_\kappa ,\dtheta (v)]=-\dtheta (v)$,
hence $v$ and $\dtheta (v)$ are linearly independent. As $\dim V(\dtau ,1)=1$ it follows that $\dtheta (v)\not\in V$. Similarly, if
$[X_\kappa ,v]=-v$ then $[X_\kappa,\dtheta (v)]=v$ and $v\not\in V$. It follows that $V\cap \theta (V)=\{0\}$.

\end{proof}

The conclusion from this is
\begin{corollary} If $\nhOne\not= \{0\}$, then $\nOne\not=\{0\}$  and $\dim \nOne=\frac{1}{2}\dim \nhOne$. Furthermore, $\dtau|_{\nhOne}$
defines a conjugation on $\nhOne$ so $\nOne$ is a totally real subspace.
\end{corollary}

\begin{remark} This follows also from the following observation. Lemma \ref{le:Complex} states that $I_o=-\ad (Y_\kappa ) \circ \dthh$ defines
a complex structure on $\nhOne$. $\dtau$ commutes with $\dthh$ and anti-commutes with $\ad (Y_\kappa)$. Hence 
$I_o\dtau = -\dtau I_o$ which shows that $\dtau|_{\nhOne}$ is conjugate linear. Hence $\nOne=\nhOne\cap \fg$ is a real form
for $\nhOne$ and $\nhOne = \nOne \oplus I_o\nOne$.
\end{remark}

\begin{lemma} Let $V\subset \gh$ be one dimensional or a  simple ideal. Then either $\dtau (V)=V$, or $\dtau (V)\cap V=\{0\}$
and we have the ``group case'' where $V\times  \dtau (V)$ is an ideal, $V$ and $\dtau (V)$ commute, and 
$(V\times \dtau(V) )^{\dtau}=\{(X,\dtau (X))\mid X\in V\}$.
\end{lemma}
\begin{proof} If $V\cap \dtau (V)\not=\{0\}$ then $V\cap \dtau (V)$ is an ideal in $V$. As $V$ is either one dimensional or
simple it follows that $V=\dtau (V)$.  The rest is obvious.
\end{proof}

\begin{lemma} $\dtau \f(l_2)=\fl_2$ and $\fl_2\cap \fg$ is an ideal in $\mZero$. Let $L_2$ be the analytic subgroup of $\Gh$ with Lie algebra $\fl_2$.
Then $L_2/G\cap L_2$ is a compact symmetric space.
\end{lemma}
\begin{proof} $\fl_2$ is the maximal compact ideal of $\mhZero$. As $\fl_2+\dtau (\fl_2)$ is a compact ideal it follows that
$\dtau \f(l_2)=\fl_2$.
The rest of the Lemma is now obvious.
\end{proof}
\begin{lemma} Assume that $\ghone\not=\{0\}$. We have $\dtau (\ghone) = \ghone$ and $\dtau (Z_\kappa^1)=-Z_\kappa^1$.  Let
$G_h^1$ be the analytic subgroup of $G_k$ with Lie algebra $\ghone$. Then $G_h^1$ is $\thh$ and $\tau$ invariant. If
$K_h^1=(G_h^1)^{\thh} = K_h\cap G_h^1$ then $K_h^1$ is maximal compact in $G_h^1$, $G_h^1$ is a bounded
domain, $\tau $ defines a conjugation on $G_h^1/K_h^1$ and $(G\cap G_h^1)/( G\cap K_h^1)=(G_h^1/K_h^1)^\tau$ is
a real form of $G_h^1/K_h^1$. 
\end{lemma}

\begin{lemma} $\fg^{(1)}:=\fg\cap \ghone$ is an ideal in $\mZero$.
\end{lemma}
\begin{proof} We have $[\mZero ,\fg^{(1)}]\subset \ghone\cap \fg=\gone$.
\end{proof}

The next result follows easily from the above.
\begin{lemma} $\dtau (\mhtwo )=\mhtwo$ and $\mtwo = \mhtwo \cap \fg$ is an ideal in $\mZero$.
\end{lemma}
 
As $\dtau (X_\kappa)=X_\kappa$ we have $F_\kappa \subset \Gh^{\tau}$. Let $\tilde F_\kappa := F_\kappa \cap G$.
 
\section{Lift from $K$ to $(L'_c)_0$}

\noindent
One of the results in the paper (\S 6) will be an extension of sections of homogeneous vector bundles over $G/K$ to its closure, and hence the boundary orbits. A key step in the proof will be a lift of irreducible representations of $K$ to $L_c$. In this section we will do the lift from $\fk$ to $\fl'_c$, i.e. from $K$ to $(L'_c)_o$. Subsequently we will treat the full $L_c$. A glance at Table 5 shows the real forms $G$ divided into three types.  In subsequent subsections the proof of the lift will be done for each type. 

\subsection{The case OCCC}\hfill

We shall use the terminology of $\sigma$-normal system of roots  for which a convenient reference is \cite {WrI} p. 21-24. For this subsection only we shall denote by $G$ a non-compact connected semisimple Lie group with Lie algebra $\fg$, later the results will be applied to $(L'_c)_o$ in Table 5. The Killing form on $\fg$ induces a non-degenerate symmetric bilinear form on $\fg^*$ for which we use $\langle \cdot,\cdot\rangle$.  Let  $\theta$ be a Cartan involution and write $\fg=\fk\oplus \fs$ for the Cartan decomposition of $\fg$. Let $\fa$ be a maximal abelian subspace in $\fs$ and, as usual, let $\fm=\fz_{\fk} (\fa)$,
 and extend $\fa$ to a Cartan subalgebra
$\ft=\ft_+\oplus \fa$ of $\fg$. Denote by $\Delta = \Delta (\fg^\C,\ft^\C)$ the set of roots of $\ft^\C$ in $\fg^\C$. Clearly $\Delta$ is a reduced system of roots. Our assumption in this subsection is that all Cartan subalgebras in $\fg$ are conjugate, to be denoted  OCCC.

\begin{lemma} $\ft_+$ is a Cartan subalgebra of $\fk$ and $\fm$.
\end{lemma} 

\begin{proof} For a Cartan subalgebra $\fc$ of $\fg$ let
\[\fc_{R}=\{X\in \fc\mid (\forall \alpha\in \Delta (\fg^\C,\fc^\C))\,\, \alpha (X)\in \R\}\]
and
\[\fc_{I}=\{X\in \fc\mid (\forall \alpha\in \Delta (\fg^\C,\fc^\C))\,\, \alpha (X)\in i\R\}\, .\]
Then $\fc=\fc_I\oplus \fc_R$ and the dimensions $\dim \fc_I$ and $\dim \fc_R$ are constant on each conjugacy class. In particular, for $\fc = \ft$,
$\ft_R=\fa$ and $\ft_I=\ft_+$.

If $\ft_+$ is not a Cartan subalgebra of $\fk$, then $\ft_+$ extends to a Cartan subalgebra $\tilde\ft_+$ of $\fk$ which in turn extends to a Cartan subalgebra $\tilde\fc$ of $\fg$ such that $\ft_+$ is a proper subspace of
$\tilde\ft_+$, or $\ft_+\subsetneq\tilde\fc_I$ which is not possible by the above discussion.
\end{proof}

It follows that $\ft$ is a fundamental Cartan subalgebra as well as a maximally split Cartan subalgebra. As $\ft=\ft_+\oplus \fa$ we can restrict roots from $\Delta$ to either $\ft_+$ or $\fa$. Denote by $\Sigma =\Sigma (\fg,
\fa)$ the set of (restricted) roots of $\fa$ in $\fg$, i.e. $\Sigma =\{\beta |_\fa \mid \beta\in \Delta\}\setminus \{0\}$. For $\alpha \in \Sigma$ and $\Delta (\alpha ):=\{\beta\in \Delta \mid \beta|_\fa =\alpha\}$ we let $\fg_\alpha \subset \fg$ be the restricted root space, and set
\[m_\alpha : =\dim \fg_\alpha =\# \Delta (\alpha )\, .\] That $\fg$ has one conjugacy class of Cartan subalgebra is equivalent to all multiplicities $m_\alpha$, $\alpha\in \Sigma$ are even. 
Next we define the involution that will serve as the $\sigma$ of the $\sigma$-normal system. Let $\ft_{\R} = i\ft_+\oplus \fa$. For $\lambda \in \ft_{\R}^*$ let
\[\lambda^\theta := \lambda \circ \theta ,\,\, \lambda^\sharp = -\lambda^\theta ,  \lambda^+:= 
\frac{1}{2}\left(\lambda +\lambda^\theta\right) =\frac{1}{2}\left(\lambda -\lambda^\sharp\right),\]
 and
\[ \lambda^-:=\frac{1}{2}\left(\lambda -\lambda^\theta\right) =
\frac{1}{2}\left(\lambda +\lambda^\sharp\right) .\]
We identify $\lambda^+$ with $\lambda|_{\ft^+}$ and similarly write $\lambda^-$ for $\lambda|_\fa$. 

If $\alpha \in \Delta$ then $\alpha^\theta, \alpha^\sharp$ are in $\Delta$ because $\fg^\C_{\alpha^\theta}=\theta (\fg^\C_\alpha)$ and $\fg^\C_{\alpha^\sharp}=
\theta (\fg^\C_{-\alpha})$. Also $\theta$ and $^\sharp$ are isometries for $\langle\cdot,\cdot\rangle$. 

It is also clear that
\[\Delta_{\bullet}:=\{\alpha \in \Delta \mid \alpha^\theta= \alpha\}=\Delta (\fm^\C,\ft_+^\C)\, = \{\alpha \in \Delta \vert \alpha^-= 0\} .\]
\begin{lemma}\label{L-2} Assume OCCC. Then $\beta^\sharp\not= \beta$ for all $\beta \in\Delta$. In fact, $\beta^\theta - \beta \notin \Delta$.
\end{lemma}
\begin{proof} Let $\beta \in \Delta$. Suppose that $\beta^\sharp = \beta$. Then $\beta^+= 0$, hence $\beta\in \Sigma$. But then 
\[\Delta (\beta )=\{\beta \}\cup \{\gamma \in \Delta \mid \gamma^\sharp\not= \gamma\,,\,\, \gamma|_\fa=\beta\}.\]
Hence $m_\beta$ is odd which contradicts OCCC. 

If $\beta^\theta = \beta$ then $\beta^\theta - \beta = 0$, so is not a root. Assume that $\beta^\theta \not= \beta$ and that $\gamma = \beta^\theta - \beta \in \Delta$. Then $\gamma ^\theta= -\gamma$ so that $\gamma^+ = 0$, i.e. $\gamma$ is a real root. But $\ft$ is fundamental so there are no real roots.
\end{proof}

\begin{corollary} $(\Delta,\theta)$ is a normal $\sigma$-system of roots per \cite {WrI}.
\end{corollary}

From this, various properties of the roots will follow. The OCCC condition will impose some additional constraints which we will identify in the next few results.

\begin{lemma}\label{L-3} Let $\alpha\in\Delta$. Then $\alpha^+\in \Delta (\fk^\C,\ft^\C_+)$.
\end{lemma}
\begin{proof} Let $X_\alpha =X_\alpha^+ + X_\alpha^-\in\fg^\C_{\alpha}$. Here $X_\alpha^{\pm}=\frac{1}{2}\left(X_\alpha \pm \theta (X_\alpha )\right)$.
If $H\in \fa$ then
\[ [H,X_\alpha]=\alpha (H)(X_\alpha^++X_\alpha^-)=[H,X_\alpha^+]+[H,X^-_\alpha ]\, .\]
It follows that 
\[[H,X^\pm_\alpha]=\alpha (H)X_\alpha^{\mp}\, .\]
Thus $X_\alpha^\pm \not= 0$. But the same argument shows that for $H\in \ft_+$ we have $[H,X_\alpha^\pm]=\alpha (H)X_\alpha^\pm$ and therefore
$\fk^\C_{\alpha^+}\not= \{0\}$.
\end{proof}
\begin{lemma} Assume OCCC. Let $\alpha\in\Delta\setminus \Delta_\bullet$. Then $\alpha$ and $\alpha^\theta$ are strongly orthogonal.
\end{lemma}
\begin{proof} We have $\alpha- \alpha^\theta = 2\alpha^-$. By the above $2\alpha^-\not\in\Delta$. Similarly we have $\alpha+\alpha^\theta
=2\alpha^+$. We just saw that $\alpha^+\in \Delta (\fk^\C,\ft_+^\C)$. As $\ft_+^\C$ is a Cartan subalgebra of $\fk^\C$ it follows that $2\alpha^+\not\in\Delta^+$.
\end{proof}
\begin{corollary} Assume OCCC. If $\alpha\in\Delta\setminus \Delta_\bullet=\{\alpha\in\Delta\mid \alpha\not=\alpha^\theta\}$. Then $\|\alpha^+\|
=\|\alpha^-\|$.
\end{corollary}
\begin{proof} This follows from the last lemma which implies that $\alpha$ and $\alpha^\theta$ are orthogonal or $\langle \alpha ,\alpha^\theta\rangle =\|\alpha^+\|^2-\|\alpha^-\|^2=0$.
\end{proof}

\begin{lemma} Let $\Delta_\sharp=\{\alpha^+\mid \alpha^\theta\not=\alpha\}$ (not counted with multiplicities). Then $\Delta (\fk^\C,\ft^\C_+)=\Delta_\sharp\dot{\bigcup} \Delta_\bullet$.
\end{lemma}

\begin{proof} It is clear that the union is disjoint. Let $\Sigma^+$ be a set of positive roots in $\Sigma$ and, as usual, $\fn=\bigoplus_{\gamma\in\Sigma^+}\fg_\gamma$. Then
\[\fk^\C=\fm^\C\oplus\bigoplus_{\gamma\in\Sigma^+}\{X_\gamma +\theta (X_\gamma)\mid X_\gamma \in \fg_\gamma^\C\}\, .\]
As $\Sigma =\{\alpha |_\fa\mid \alpha^\theta \not=\alpha\}$ the claim follows now using the same argument as in the proof of Lemma \ref{L-3}.

\end{proof}
The following set of simple roots is adapted from \cite {WrI} p. 21-24 with slightly different notation. Let $\ell_+:=\dim \ft_+$, $\ell_2:=\dim \fa$ and $\ell =\ell_++\ell_2=\dim \ft_\R$. We choose a lexicographical ordering in $\ft_\R^*$ with respect to  a basis $H_1,\ldots H_{\ell}$ so that $H_1,\ldots ,H_{\ell_+}$ is a basis for
$i\ft_+$. Let $\Delta^+$ be the corresponding set of positive roots and $\Pi$ the set of simple roots. Then by Lemma \ref{L-2} and \cite {WrI} there exists $\ell_1$ such that the following holds:
\begin{enumerate}
\item $\Pi_\bullet =\{\alpha_1,\ldots ,\alpha_{\ell_1}\}$ is a set of simple roots for $\Delta_\bullet$ (contained in $\Delta_\bullet^+=\Delta_\bullet \cap \Delta^+$). Furthermore $\Pi_\bullet =\{\alpha\in\Pi\mid \alpha^\theta =\alpha\}$.
\item $\ell = \ell_1+2\ell_2$.
\item If $1\le \nu \le \ell_2$ then $\alpha_{\ell_1+\nu}^\theta = \alpha_{\ell_1+\ell_2+\nu}$ and
$\alpha^\theta_{\ell_1+\ell_2+\nu}=\alpha_{\ell_1+\nu}$.
\end{enumerate}

\begin{lemma} $\Pi_c=\{\alpha_1,\ldots ,\alpha_{\ell_1},\alpha_{\ell_1+1}^+,\ldots ,\alpha_{\ell_1+\ell_2}^+\}$ is a
simple system in $\Delta ^+(\fk^\C,\ft^\C_+)$.
\end{lemma} 
Let $\Psi =\{\mu_1,\ldots ,\mu_{\ell}\}$ denote the set of fundamental weights for $\Pi$. 
\begin{lemma} Let $\Psi_c:=\{\mu_j^+\mid j=1,\ldots ,\ell_1+\ell_2\}$ (where we identify 
$\mu_j$ with $\mu_j^+$ for $j=1,\ldots ,\ell_1$). Then
$\Psi_c$ is the set of fundamental weights corresponding to the simple system $\Pi_c$.
\end{lemma}
\begin{proof} We have to show that
\[\frac{2\langle \mu^+_\nu ,\alpha_\sigma^+\rangle}{\langle \alpha_\sigma^+,
\alpha_\sigma^+\rangle}=\delta_{\nu,\sigma}\, .\]
This is clear for $\nu=1,\ldots ,\ell_1$ as in this case $\mu_\nu=\mu_\nu^+$. Assume now that $\ell_1+1\le \nu \le \ell_1+\ell_2$.
Then for $1\le \sigma \le \ell_1$ we have
\[0=\langle \mu_\nu , \alpha_\sigma \rangle=\langle \mu_\nu^+,\alpha_\sigma\rangle\, .\]
Assume  $\ell_1+1\le \sigma \le \ell_1+\ell_2$ and write $\sigma = \ell_1+j$, $1\le j\le \ell_2$. then
\begin{eqnarray*}
\langle \mu_\nu^+,\alpha_\sigma^+\rangle &=& \langle \mu_\nu^+,\frac{1}{2}(\alpha_\sigma +\alpha_\sigma^{\theta} ) \rangle\\
&=&\frac{1}{2}\langle \mu_\nu ,\alpha_\sigma +\alpha_\sigma^{\theta}\rangle
\end{eqnarray*}
because $\theta $ is an involution. As $\alpha_{\sigma}^\theta = \alpha_{\ell_1+\ell_2+j}$ and
$\|\alpha_\sigma^+\|^2=\frac{1}{2}\|\alpha_\sigma\|^2=\frac{1}{2}\|\alpha_\sigma^\theta\|^2$ we get
\begin{eqnarray*}
\frac{2\langle \mu_\nu^+, \alpha_\sigma^+\rangle}{\|\alpha_\sigma^+\|^2}
&=&\frac{2\langle \mu_\nu , \frac{1}{2}\left(\alpha_\sigma +\alpha_\sigma^\theta \right) \rangle}{\|\alpha_\sigma^+\|^2}\\
&=&\frac{2\langle \mu_\nu, \frac{1}{2}\alpha_\sigma\rangle}{\frac{1}{2}\|\alpha_\sigma\|^2}
 +\frac{2\langle \mu_\nu , \frac{1}{2}\alpha_{\ell_1+\ell_2+j}\rangle}{\frac{1}{2}\|\alpha_{\ell_1+\ell_2+j}\|^2}\\
 &=&\delta_{\nu,\sigma}\, 
 \end{eqnarray*}
\qedhere
\end{proof}

Denote by $\Lambda^+(K)$ the set of highest weights of irreducible representations of $K$ and similarly by $\Lambda^+(G)$ the space of highest weights of
irreducible finite-dimensional representations of $G$. If $\mu\in \Lambda^+(K)$ then we denote the corresponding irreducible representation of
$K$ by $\sigma_\mu$. If $\mu \in \Lambda^+(G)$ then the corresponding irreducible representation of $G$ with highest weight $\mu$ is denoted by 
$\tau_\mu$.

Let $\widetilde{G}$ be the universal covering of $G$ and let $\widetilde{K}$ denote the
analytic subgroup of $\widetilde{G}$ corresponding to the Lie algebra $\fk$. Then 
$\widetilde{K}$ is simply connected and locally isomorphic to $K$. Furthermore,
the center of $\widetilde{G}$, $Z(\widetilde{G})$, is contained in $\widetilde{K}$.

\begin{theorem} Let $\mu=\sum_{j=1}^{\ell_1+\ell_2}k_j\mu^+_j \in \Lambda^+(K)$. Set $\widetilde{\mu}:=
\sum_{j=1}^{\ell_1+\ell_2}k_j \mu_j$. Then $\widetilde{\mu}\in\Lambda^+(\widetilde{G})$,
 descends to be in $\Lambda^+(G)$. Moreover $\sigma_\mu$ is contained in
$\tau_{\widetilde{\mu}}|_K$ with multiplicity one.
\end{theorem}
\begin{proof} It is clear that $\widetilde{\mu} \in\Lambda^+(\widetilde{G})$ and $\mu\in\Lambda^+(\widetilde{K})$.
Denote by 
$\widetilde{\tau}_{\widetilde{\mu}}$ respectively
$\widetilde{\sigma}_\mu$ the corresponding representation of $\widetilde{G}$, respectively
$\widetilde{K}$. Clearly $\widetilde{\sigma}_\mu$ is contained in
$\widetilde{\tau}_{\widetilde{\mu}}|_{\widetilde{K}}$. Let $\widetilde{Z}$ be the kernel of the canonical projection
$\widetilde{G}\to G$. Then $\widetilde{Z}\subset \widetilde{K}$ and $K \simeq \widetilde{K}/\widetilde{Z}$.
 Since $\mu\in \Lambda^+(K)$ it follows that $\widetilde{\sigma}_\mu|_{\widetilde{Z}}=\id$.
As $\widetilde{Z}$ is central in $\widetilde{G}$ and $\widetilde{\tau}_{\widetilde{\mu}}$ is
irreducible one has $\widetilde{\tau}_{\widetilde{\mu}}|_{\widetilde{Z}}$ is a scalar.
But $\widetilde{\sigma}_\mu$ is contained in $\widetilde{\tau}_{\widetilde{\mu}}|_{\widetilde{K}}$, it follows that
$\widetilde{\tau}_{\widetilde{\mu}}|_{\widetilde{Z}}=\id$. Hence $\widetilde{\tau}_{\widetilde{\mu}}$ defines
a representation of $G$ and $\widetilde{\mu}\in \Lambda^+(G)$. The multiplicity one assertion is
clear because there is no way to write $\mu$ as a non-trivial linear combination $(\mu,0)-\sum
n_\alpha (\alpha^+,\alpha^-)|_{\ft_+}$ of positive roots $(\alpha^+,\alpha^-)$ and $n_\alpha \ge 0$
(and at least one $\not= 0$). The rest is now obvious. 
\end{proof}

\subsection{The special cases}\hfill

We turn to the third type in Table 5. The technique is a variation of $\sigma$-systems from
the previous subsection. Here we  use some results from \cite {Kn} on Vogan diagrams. The 
procedure parallels that followed in the OCCC case. One begins with $\ft=\ft_+\oplus \fa$ a
 fundamental Cartan subalgebra of $\fg$ but here not a maximal split Cartan. Hence again there 
are no real roots. Of course $\ft$ determines a parabolic subalgebra which will play no direct role. 
We have $\Delta = \Delta (\fg^\C,\ft^\C)$ the set of roots of $\ft^\C$ in $\fg^\C$, $W$ the Weyl group 
of $\Delta$; let $ \Delta (\fk^\C,\ft^\C_+)$ be the set of roots of $\ft^\C_+$ in $\fk^\C$ and $W_K$ its 
Weyl group. We choose a lexicographical ordering in $\ft_\R^*$ with respect to  a basis $H_1,\ldots H_{\ell}$ 
so that $H_1,\ldots ,H_{\ell_+}$ is a basis for
$i\ft_+$. Let $\Delta^+$ be the corresponding set of positive roots and $\Pi = \{\alpha_1,\cdots, \alpha_l\}$ 
the set of simple roots. Denote by $\Psi =\{\mu_1,\ldots ,\mu_{\ell}\}$ the set of fundamental weights for $\Pi$. 
As before, for $\lambda \in \ft_{\R}^*$ let
$\lambda^\theta := \lambda \circ \theta. $ Then we have the restriction to 
$\ft_+, \,\lambda^+:= \frac{1}{2}\left(\lambda +\lambda^\theta\right),$ and the 
restriction to $\fa$,  $\lambda^-:=\frac{1}{2}\left(\lambda -\lambda^\theta\right).$  A different
but important feature arises here in that imaginary roots can be compact or noncompact. 
Thus we must examine $\Sigma^+$, the restrictions of $\Delta^+$ to $\ft_+$. Also we 
make a choice of simple roots for $ \Delta (\fk^\C,\ft^\C_+)$ compatible with $\Delta^+$. 
To us it seemed easiest to continue with the remaining details in each case separately.

\begin{example}\label{ex:311}
We start with $\fg = \so (5,5)$ and $\fk = \sp(2) \times \sp(2) = \so (5)\times\so (5)$. Using standard notation and as presented in \cite {Kn} p. 359 we have $\Pi = \{ \alpha_1 = e_1 - e_2, \alpha_2 = e_2 -e_4, \alpha_3 = e_4 - e_5, \alpha_4 = e_5 - e_3, \alpha_5 = e_5 +e_3\}$, where $\ft_+ =<e_1, e_2, e_4, e_5>$ and $\fa = <e_3>$. Clearly $\theta:\Pi\to \Pi$ interchanges $\alpha_4$ and $\alpha_5$, so relative to the involution $\theta$ we have a normal $\sigma$-system with a $\sigma$ order for which $\Pi$ is a $\sigma$-fundamental system.  A computation using the Cartan on \cite[p. 359]{Kn}
 determines the set of restrictions, $\Sigma^+$, of $\Delta^+$ to $\ft_+$ which, from \cite {WrI}, is a (non-reduced) root system and contains the positive roots of the Levi subalgebra $\so (4,4) \oplus \fa$. We let $W_{\Sigma^+}$ be its 
 Weyl group. Now set $W_{\theta} = \{ w\in W \vert w\circ \theta = \theta \circ w\}$. Then $W_\theta$ induces 
 a map on $\ft_+$ and, from \cite {WrI} p. 24, as there are no real roots we have $W_\theta\vert_{\ft_+} = W_{\Sigma^+}$. Finally with regard to Weyl groups (following \cite{Han} p. 1016 and others) we will take a distinguished set 
 of representatives for $W_\theta/W_K$, viz. let $D_{\ft_+}^+$ be the positive Weyl chamber for $ \Delta^+ (\fk^\C,\ft^\C_+)$ and let $D_{\fg}^+ $ be the projection of the positive chamber for $\Delta^+$ in $\ft (=\ft_+\oplus \fa)$ to $\ft_+$. 
 Then set $W^1 = \{ w\in W_\theta\vert w\vert_{\ft_+} (D_{\fg}^+ )\subset D_{\ft_+}^+\}$. Then $W^1 $ gives the required coset representatives.

Yet another computation is necessary to obtain $ \Delta ^+(\fk^\C,\ft^\C_+) = \{e_1, \alpha_1^+, \alpha_2^+ + \alpha_3^+ +\alpha_4^+, \alpha_1^+ +  2\alpha_2^+ + 2\alpha_3^+ +2\alpha_4^+\} \cup \{e_4, \alpha_3^+, \alpha_4^+, \alpha_3^+ +2\alpha_4^+\}$,  with compatible basis of simple roots $\{\alpha_1^+ = e_1 - e_2, \alpha_2^+ + \alpha_3^+ +\alpha_4^+ = e_2\} \cup\{\alpha_3^+ = e_4 - e_5, \alpha_4^+ = e_5\}.$ As for the fundamental weights for $\fg$, one gets

\[
\mu_1 = e_1, \quad \mu_2 = e_1 + e_2, 
\quad \mu_3 = e_1+e_2+e_4, \]
\[\mu_4 = \frac{e_1 + e_2-e_3+e_4+e_5}{2},\quad
\mu_5 = \frac{e_1 + e_2+e_3+e_4+e_5}{2};
\]
 while for $\fk$, 
$$\mu_1^+ = e_1 = \mu_1, \mu_2^+ = e_1 +e_2 = \mu_2,\mu_3^+ =  e_4, \mu_4^+ = e_4+e_5 .$$

In terms of the $e_i$, for a highest weight $\widetilde{\mu}$ we have $\widetilde{\mu} =\sum_1^5 m_i\mu_i = (m_1+m_2+m_3+\frac {m_4+m_5}{2}) e_1 + (m_2+m_3+\frac {m_4+m_5}{2})e_2 + (\frac{m_4-m_5}{2}) e_3 +
(m_3+\frac {m_4+m_5}{2}) e_4 + (\frac {m_4+m_5}{2})e_5$.  So, similar to the procedure in the Theorem above, to obtain $\widetilde{\mu}$ as a natural lift from $\ft^+ $ we take $m_4  = m_5$ giving $\widetilde{\mu}^+  = \widetilde{\mu} = M_1e_1+ M_2 e_2 + M_4e_4 + M_5 e_5$ with $M_1\ge M_2\ge M_4 \ge M_5 \ge 0$. 
Now take a candidate highest weight $\mu = \sum_1^4 n_i\mu_i^+$ of $\fk$ to lift to $\widetilde{\mu}$. In terms of the $e_i$ we have $\mu = (n_1 + n_2) e_1 + n_2e_2 + (n_3 + n_4) e_4 + n_4 e_5 = N_1e_1 +N_2 e_2 +N_4e_4 +N_5 e_5$ and $N_1\ge N_2\ge 0, N_4\ge N_5\ge 0$.  Clearly when $N_2 = m_2 + N_4$, i.e. $N_2 \ge N_4$, we have a $\mu$ to lift to $\widetilde{\mu}$ . However this determines a chamber in $\ft^+$ for the action of $W^1$. Now 
$\fg = \so (5,5)$ is type $D_5$ so the Weyl group contains all permutations of the $e_i$. We summarize in Table 1
Case \ref{ex:311} various possibilities for the chamber and an element of $W^1 $ that maps the chamber to the original one. We use the abbreviation $i \longleftrightarrow N_i$ and $e_i -e_j \longleftrightarrow s_{e_i -e_j}\in  W^1 $.

 \begin{table}[h]
\center
\begin{tabular}{|c|c|c|}
\hline
$\mu$ & $w$& orbit\\
\hline
$1\ge 2\ge 4\ge 5$ & id& $1\, 2\, 4\, 5$\\
$1\ge 4\ge 2\ge 5$ & $e_2 -e_4$& $1\, 2\, 4\, 5 \to 1\, 4\, 2\, 5$\\
$1\ge 4\ge 5\ge 2$& $e_2-e_4 \circ e_2 -e_5$&$1\, 2\, 4\, 5 \to 1\, 5\, 4\, 2 \to 1\, 4\, 5\, 2$\\
$4\ge 1\ge2\ge 5$& $e_2 -e_4 \circ e_1 -e_4$ & $1\, 2\, 4\, 5\to 4\, 2\, 1\, 5 \to 4\, 1\, 2\, 5$\\
$4\ge 1\ge 5\ge 2$& $e_2 - e_4 \circ e_2 - e_5 \circ e_1 - e_4$& $1\, 2\, 4\, 5\to 4\, 2\, 1\, 5\, \to 4\, 5\, 1\, 2\to 4\, 1\, 5\, 2$\\
$4\ge 5\ge 1\ge 2$& $e_2 - e_5 \circ e_1 -e_4$& $1\, 2\, 4\, 5\to 4\, 2\, 1\, 5\to 4\, 5\, 1\, 2$\\
\hline
\end{tabular}
 \medskip
 
\caption{Case \ref{ex:311}} 
\end{table}

 So given $\mu\in\Lambda^+(K)$ one finds it in the first column, applies $w^{-1}$ to it obtaining a highest weight of the form $1\ge 2\ge 4\ge 5$ which can be lifted to a natural $\widetilde{\mu} \in\Lambda^+(\widetilde{G})$. It is clear the $w^{-1} $ belongs to $W^1$ as it takes a chamber of dominant $K$-weights to another. The result then follows from the multiplicity one Theorem in \cite{Han} which says the $K$-type $w(\widetilde{\mu})\vert_ {\ft^+}$ occurs with multiplicity 1 in $V_{\widetilde{\mu}}$.
 
 An alternative approach to the existence of the $K$-submodule is to use the generalization of the PRV conjecture (\cite {MPR}), but this does not yet give multiplicity 1.
 \end{example}
 
 \begin{example}\label{ex:412}
Next we consider $\fg = \fe_{6(6)}$ and $\fk = \sp (4)$. We shall use the notation of \cite{Bo} so that we have a basis 
$\{ \alpha_1,\alpha_2,\alpha_3,\alpha _4,\alpha_5,\alpha_6\}$ for $ \Delta^+ (\fg^\C,\ft^\C)$. We use Table C p. 532 in \cite{Kn} for a compatible basis of the simple roots for  $\fk = \sp (4)\subset \fe_{6(6)}$. In particular, 
node 2 is black, and under $\theta$, nodes $3\leftrightsquigarrow 5$ and $1\leftrightsquigarrow 6$. 
This suggests the following basis for $ \Delta^+ (\fk^\C,\ft^\C_+)$: $\{ \gamma_1 = \alpha_2 + \alpha_4 + \frac{\alpha_3 + \alpha_5}{2}, \gamma_2 = \frac {\alpha_1 + \alpha_6}{2}, \gamma_3 = \frac{\alpha_3 + \alpha_5}{2}, \gamma_4 = \alpha_4 \}$. Note that we use $\gamma$ because these are not always the projections to $\ft^+$, e.g. $\gamma_2 \ne \alpha_2^+$.

From \cite {Bo} one computes that $\langle \alpha_i,\alpha_i\rangle = 2$, and since the fundamental weights satisfy $2\frac {\langle \mu_i,\alpha_j\rangle}{\langle \alpha_j,\alpha_j\rangle} = \delta_{i,j}$ we have that the fundamental weights $\mu_i$ are the dual basis to the simple roots $\alpha_i$. Similarly one obtains that $1 = \langle \gamma_1,\gamma_1\rangle = \langle \gamma_2,\gamma_2\rangle = \langle \gamma_3,\gamma_3\rangle$ while $\langle \gamma_4,\gamma_4\rangle = 2$. Then for the fundamental weights of $ \Delta^+ (\fk^\C,\ft^\C_+)$ we can take $\omega_1 = \frac{\mu_2}{2}, \omega_2  = \frac {\mu_1 + \mu_6}{2}, \omega_3 = \frac{\mu_3 + \mu_5 - \mu_2}{2}, \omega_4  = \mu_4 - \mu_2$.

Let $\mu \in \Lambda^+(K)$. Then $\mu = \sum_1^4 n_i \omega_i$ with $n_i\ge 0$ and integers. In terms of the $\mu_i$ we have 
\begin{align}
\mu &= n_1 \frac{\mu_2}{2} + n_2 \frac{\mu_1 +\mu_6}{2} + n_3 \frac{\mu_3 +\mu_5 -\mu_2}{2} + n_4(\mu_4 - \mu_2)\\
\mu &= n_2\frac {\mu_1 + \mu_6}{2} + n_3 \frac{\mu_3 +\mu_5}{2} + n_4 \mu_4 + (\frac{n_1 - n_3}{2} - n_4) \mu_2
\notag\end{align}

Here we must make the assumption that $n_1 - n_3$ is an even integer. Then, as before, we are left with a few cases which will be handled using the Weyl group, i.e. $W^1$. We begin with the case $n_1 -n_3 -2n_4 \ge 0$. Here we lift $\mu$ to $\widetilde{\mu} = n_2 \mu_1 + n_3 \mu_3 + n_4 \mu_4 +(\frac{n_1 - n_3}{2} - n_4) \mu_2$. Then $\widetilde{\mu}^+ = \mu$ so we have a valid lift. In Table 2  Case \ref{ex:412}, similar to that above, the first column contains the various cases for $\mu$, the second the sequence of roots whose reflections give $w$, and the third the lift to $\Lambda^+(\widetilde{G})$ to which you apply $w^{-1}$ and the restriction gives $\mu$.

 \begin{table}[!h]
 \center
 \scriptsize{
\begin{tabular}{|c|c|c|}
\hline
$\mu$ & $w$& Lift $\widetilde{\mu}$\\
\hline
$n_1 -n_3 -2n_4 \ge 0$ & id& $n_2 \mu_1 + n_3 \mu_3 + n_4 \mu_4 +(\frac{n_1 - n_3}{2} - n_4) \mu_2 $\\
\hline
$n_1 -n_3 -2n_4 < 0, n_1 -n_3\ge 0$ & $\alpha_2$& $n_2 \mu_1 + n_3 \mu_3 + \frac{n_1 - n_3}{2}  \mu_4 -(\frac{n_1 - n_3}{2} - n_4) \mu_2 $\\
\hline
$n_1 -n_3 -2n_4 < 0, n_1 -n_3 < 0$& $\alpha_5\circ\alpha_4\circ\alpha_6\circ\alpha_5\circ\alpha_4\circ\alpha_2$&$ n_2 \mu_1 +(n_3 +\frac{n_1 -n_3}{2} -n_4)\mu_3$\\
$n_3\ge n_4-\frac{n_1-n_3}{2} $& $$& $+n_4 \mu_5 +(-\frac{(n_1 - n_3)}{2})\mu_6 $\\
\hline
$n_1 -n_3 -2n_4 < 0, n_1 -n_3 < 0$& $\alpha_5\circ\alpha_4\circ\alpha_6\circ\alpha_5\circ\alpha_4\circ\alpha_2$& $ [n_2 + n_3+(\frac{n_1 - n_3}{2} - n_4)]\mu_1 + (\frac{n_1 + n_3}{2})\mu_5$\\
$n_3<n_4-\frac{n_1-n_3}{2} $& $\alpha_2\circ\alpha_4\circ\alpha_3\circ$& $-(n_3 + \frac{n_1 -n_3}{2} -n_4)\mu_2 + (-\frac{n_1 - n_3}{2} )\mu_6$\\
$n_2 + n_3 \ge  n_4- \frac{n_1 -n_3}{2}$&$$&$$\\
\hline
$n_1 -n_3 -2n_4 < 0, n_1 -n_3 < 0$& $\alpha_5\circ\alpha_4\circ\alpha_6\circ\alpha_5\circ\alpha_4\circ\alpha_2$& $n_2\mu_2 - [n_2 + n_3 + (\frac{n_1 -n_3}{2} -n_4)]\mu_4$\\
$n_3<n_4-\frac{n_1-n_3}{2} $& $\alpha_2\circ\alpha_4\circ\alpha_3\circ$& $[n_1 +n_2+n_3 -n_4]\mu_5 + (-\frac{n_1 - n_3}{2} )\mu_6$\\
$n_2 + n_3  <n_4-\frac{n_1-n_3}{2} $&$\alpha_4\circ\alpha_3\circ\alpha_1\circ$&$$\\
$\frac{n_1+n_3}{2}+ n_2+n_3 \ge n_4 -\frac{n_1-n_3}{2}$ &$ $ & $ $\\
\hline
$n_1 -n_3 -2n_4 < 0, n_1 -n_3 < 0$& $\alpha_5\circ\alpha_4\circ\alpha_6\circ\alpha_5\circ\alpha_4\circ\alpha_2$& $n_2\mu_2 + \frac{n_1 + n_3}{2} \mu_4$\\
$n_3<n_4-\frac{n_1-n_3}{2} $& $\alpha_2\circ\alpha_4\circ\alpha_3\circ$& $  -[\frac{n_1+n_3}{2}+ n_2+n_3 - n_4 +\frac{n_1-n_3}{2} ]\mu_5 $\\
$n_2 + n_3  <n_4-\frac{n_1-n_3}{2}$&$\alpha_4\circ\alpha_3\circ\alpha_1\circ$&$+ [n_2 + 2n_3 - n_4 +\frac{n_1-n_3}{2} ]\mu_6 $\\
$\frac{n_1+n_3}{2}+ n_2+n_3 < n_4 -\frac{n_1-n_3}{2} $ &$  \alpha_5\circ$ & $ $\\
$n_2 + 2 n_3 -n_4 +\frac{n_1-n_3}{2}\ge 0$&$$&$$\\
 \hline
 $n_1 -n_3 -2n_4 < 0, n_1 -n_3 < 0$& $\alpha_5\circ\alpha_4\circ\alpha_6\circ\alpha_5\circ\alpha_4\circ\alpha_2$& $n_2\mu_2 + \frac{n_1 + n_3}{2} \mu_4$\\
$n_3<n_4-\frac{n_1-n_3}{2} $& $\alpha_2\circ\alpha_4\circ\alpha_3\circ$& $  -\frac{n_1 - n_3}{2} \mu_5  $\\
$n_2 + n_3  <n_4-\frac{n_1-n_3}{2}$&$\alpha_4\circ\alpha_3\circ\alpha_1\circ$&$- [n_2 + 2n_3 - n_4 +\frac{n_1-n_3}{2} ]\mu_6 $\\
$\frac{n_1+n_3}{2}+ n_2+n_3 < n_4 -\frac{n_1-n_3}{2} $ &$  \alpha_6\circ\alpha_5\circ$ & $ $\\
$n_2 + 2 n_3 -n_4 +\frac{n_1-n_3}{2}< 0$&$$&$$\\
\hline
 
\end{tabular}
}
\medskip

\caption{Case \ref{ex:412}}
\end{table}

So here, given $\mu = \sum_1^4 n_i \omega_i \in\Lambda^+(\widetilde{K})$ ($n_1 - n_3$ an even integer) one finds it in the first column, applies $w^{-1}$ to it obtaining a highest weight  $\widetilde{\mu} \in\Lambda^+(\widetilde{G})$. It is clear the $w^{-1} $ belongs to $W^1$ as it takes a chamber of dominant $K$-weights to another. The result then follows from the multiplicity one Theorem in \cite{Han} which says the $K$-type $w^{-1} (\widetilde{\mu})\vert_ {\ft^+}$ occurs with multiplicity 1 in $V_{\widetilde{\mu}}$.
Life would be easier if one knew more about the action of $W^1$ on the chambers $D_{\fg}^+ $; unfortunately, we were unable to obtain the result we needed which necessitated the lengthy computations. These computations were facilitated by having the expressions of the simple roots of $ \fe_{6(6)}$ expressed in terms of the fundamental weights.
\end{example}

\begin{example}
The next case $\fg = \so(1,n-1)$ and $\fk = \so (n - 1)$ is elementary and surely in several places in the literature. Assume that $n-1 \ge 3$ to avoid the Abelian case. Base extend the Lie algebras to $\mathbb C$. The fundamental representations of $\fk$ are either exterior powers of the standard representation or spin. All these are known to occur with multiplicity one in the similar representation of $\fg$. Then define a length function on highest weights in the usual way: $l(\mu) = l(\sum_1^l n_i\omega_i) = \sum_1^ln_i$. Induction and using Cartan composition provides a natural lift.

One can be more precise using standard material on highest weights and branching, e.g. as in \cite{GoWa} p. 351. Say relative to a suitable Cartan subalgebra the highest weight of $\fg^\C$ is given by a decreasing sequence $\Lambda_i$ while the highest weight of $\fk^\C$ is given by a similar sequence $\mu_i$. Then depending on the parity of $n-1$, i.e. $n-1 = 2k$ or $n-1 = 2k-1$, either one takes $\Lambda_i = \mu_i, i<k \text {and} \Lambda_k = \vert \mu_k\vert$, or $\Lambda_i = \mu_i, i<k \text {and} \Lambda_k = 0.$ 

For $\fg = \so(1,2) \cong \sl(2,\R)$ and $\fk = \so (2)$ the procedure is the same as in the previous examples, viz., each irreducible unitary representation of $\so(2)$ occurs as the highest (lowest) weight of an irreducible finite dimensional representation of $\sl(2,\R)$. 
\end{example}

\begin{example}
The remaining special case is $\fg = \so (p,q)$  $\fk = \so (p)\times \so (q)$ and $p,q \ne 2$. If at least one of the factors in $\fk$ has even parity then we have an equal rank situation. Then we have in $W^1$ all reflections generated by noncompact roots, in particular we have transpositions between $e_i,  1 \le i\le p$, and $e_j, p+1\le j\le q$. By means of these we can arrange the highest weights of the factors to be in decreasing order for $\fg$ and thus obtain a lift for any highest weight of $\fk$. If both $p,q$ are odd then we are not equal rank but $\fg$ is still of type $D_l$ whose Weyl group contains enough reflections to accomplish the same goal.

\end{example}

\subsection{The Isometries}\hfill

We turn to the remaining type in Table 5. Here $\fk$ is the Lie algebra of the isometries of a standard representation on a finite dimensional vector space while $\fg$ is the Lie algebra of all automorphisms of the vector space. For the cases at hand we will have no need of the spin representation of $\fk$. It is classical that all other such representations are obtainable from exterior powers of the standard representation together with Cartan composition, all of which have natural lifts to $\fg$.

\section{Extension from $(L_c^\prime)_0$ to $L_c$}

\noindent
In the previous section we considered the extension of representations from $K$ to $(L_c^\prime)_0$. In this
section first we discuss the extension from the connected group
$(L_c^\prime)_0$ to $L^\prime_c$. For that we need more information about
$L_c/(L_c)_0$. Let $P_{min}^0=M_{min}^0A_c^0N^0$ be a minimal parabolic subgroup in $(L_c^\prime)_0$, where
$\fa_c=\fa_c^0\oplus \R H_0$, so that $\fa_c^0$ is maximal abelian in $\fl_c^\prime\cap \fs_c$. Then   $P_{min}=M_{min}A_{c}N_{min}$ is a minimal parabolic subgroup in $G_c$ where
$N_{min}=N^0N_-$, $A=\exp \R H_0$,  $A_{c}=A_c^0A$ and $M_{min}=Z_{K_c}(\fa_c)$.  Note that
$M_{min}$ has the same Lie algebra as $M_{min}^0$ and hence $(M_{min})_0 = (M_{min }^0)_0$.

We now use well known results about the connected components of $M_{min}$ to describe the
connected components of $L_c$. As $\fa_c=\fa_c^0\oplus \R H^0$ where
$\fa_c^0$ is maximal abelian in $\fl_c^\prime$, the roots $\Sigma_{cc}$ can be identified with
$\Sigma (\fl^\prime_c,\fa_c^0)$ via restriction. 

\begin{lemma} We have $L_c^\prime =M_{min}(L_c^\prime)_0$
\end{lemma}
\begin{proof} This follows from \cite[Lem. 2.2.8]{Wa1}.
\end{proof}

Let $F_1:=\exp (i\fa_c)\cap K_c$. We note that
if $f\in F_1$ then
\[f=\eta (f)=\eta (\exp iH)=\exp (- i\dot\eta H)= \exp (-i H)=f^{-1}.\]
Thus $f^2=e$ and $F_1\simeq \Z_2^s$ for some $s$. We remark that were $F_1$ cyclic then the desired extension can be found  in \cite{Kn2} Lemma 14.22.  
Choose generators $f_1,\ldots ,f_u\in F_1$ so that with $F=\prod \{e,f_j\}$ we have
$M_{min}=F(M_{min})_0\simeq F\times( M_{min})_0$, see  \cite[Ch. VII]{He78} for details, in particular Theorem 8.5. But we
will not need the exact form of $F_1$. The following lemma now follows:
\begin{lemma} Let $F$ be as above. Then $L_c^\prime = F(L_c^\prime)_0$. 
\end{lemma}

\begin{lemma} Let $\tilde\mu \in \Lambda^+ ((L_c^\prime)_0 )$ and denote by $(\tau_{\tilde{\mu}},V_{\tilde\mu})$ the
corresponding irreducible representation. Let $f\in F$. Then the representations $\tau_{\tilde{\mu}}$ and
$f\cdot \tau_{\tilde{\mu}} : m\mapsto \tau_{\tilde{\mu}} (fmf)$ are equivalent.
\end{lemma}

\begin{proof} Clearly $f\cdot \tau_{\tilde{\mu}}$ is an irreducible representation of $(L_c^\prime)_0$. Let $\ft=\ft_+\oplus \fa_c^0$
be a Cartan subalgebra of $\fl_c^\prime$. Then for $H\in \ft$ we have
\[f\cdot \tau_{\tilde{\mu}} (\exp H )= \tau_{\tilde{\mu}} (f\exp H f)= \tau_{\tilde{\mu}} (\exp \Ad (f)H)=\tau_{\tilde{\mu}} (\exp H) .\]
Thus $f\cdot \tau_{\tilde{\mu}}$ and $\tau_{\tilde{\mu}}$ have exactly the same weights. In particular the highest
weights are the same. Hence $f\cdot \tau_{\tilde{\mu}}\simeq \tau_{\tilde{\mu}}$.
\end{proof}

It follows that for each $f\in F$ there exists  $T_f\in \GL (V_{\tilde\mu})$ such that for all $m\in (L_c^\prime)_0$, 
$T_f \tau_{\tilde{\mu}} (fmf)=\tau_{\tilde{\mu}} (m)T_f$. If $f=e$ we take $T_f=\id$.
Note that $T_f$ is unique up to a scalar $\lambda\in\T$. Let $V_{\tilde\mu} (\tilde\mu)$ be the highest-weight space. Then
$\dim V_{\tilde\mu} (\tilde\mu )=1$. Hence there exists $0\not= v_{\tilde\mu}$ such that $V_{\tilde\mu} (\tilde\mu )=\C v_{\tilde\mu}$.
\begin{lemma}\label{lem:Tf} For $f\in F$ let $T_f$ be as above. Then we can choice $T_f$ such that
\begin{enumerate}
\item $T_f^2=\id$,
\item $T_f (v_{\tilde\mu} )=v_{\tilde\mu}$.
\end{enumerate}
$T_f$ is uniquely determined by (1) and (2).
\end{lemma}
\begin{proof} We have for $m\in (L_c^\prime )_0$ by repeating the definition twice that
\[T_f^2\tau_{\tilde{\mu}} (m)=T_f^2  \tau_{\tilde{\mu}} (f^2mf^2) = \tau_{\tilde{\mu}} (m) T_f^2  .\]
As $\tau_{\tilde{\mu}}$ is irreducible there exists $c_f\in \T$ such that $T_f^2=c_f\id$. (1) now follows 
by replacing $T_f$ by $c_f^{-1/2}T_f$. As $\dim V_{\tilde\mu}(\tilde\mu) =1$ and $T_f$ leaves the weight spaces
invariant, it follows that $T_f|_{V_{\tilde\mu}(\tilde\mu)} $ is scalar, say multiplication by $d_f\not= 0 $. By (1) it follows
that $d_f^2=1$. Hence we can replace $T_f$ by $d_f^{-1}T_f$ to obtain (2) and (1).  If
$T_f$ and $S_f$ satisfy (1) and (2) then $S_f^{-1}=S_f$ and $S_fT_f = c\id$ for some $c\in \C$. But by 
(2) it follows that $S_fT_f ( v_{\tilde\mu})=v_{\tilde\mu}=c v_{\tilde\mu}$. Hence $c=1$.
\end{proof}

From now on we always assume that $T_f$, $\in F$, satisfies (1) and (2).

\begin{lemma} Let $f,g\in F$.  Then $T_{f}T_g=T_g T_f$. 
\end{lemma}
\begin{proof}  As $fg=gf$ it follows that $fgfg=f^2g^2=e$. As above this implies that
$S=T_fT_gT_fT_g=(T_fT_g)^2$ is an $\tau_{\tilde{\mu}}$-intertwining operators.  Hence there
exists $d\in \C^*$ such that $S=d\,\id$. But $T_f|_{V_{\tilde\mu} (\tilde\mu )}=T_g|_{V_{\tilde\mu} (\tilde\mu )}=1$.  Hence
\[d = S|_{V_\mu (\mu )}=1 .\]
Thus $d=1$ and $S=\id$. As $T_f^2=T_g^2= \id$ it follows, by multiplying $S$ first by
$T_f$ and then by $T_g$ that $T_fT_g=T_gT_f$. Hence, the claim.
\end{proof} 
\begin{corollary} Let $f_1,\ldots , f_u$ be generators for $F$ and let $f=f_1^{i_1}\cdots f_u^{i_u}$, $i_j \in \{0,1\}$.
Then $T_f=T_{f_1}^{i_1}\cdots T_{f_u}^{i_u}$.
\end{corollary}
\begin{proof} The operator $S_f=T_{f_1}^{i_1}\cdots T_{f_u}^{i_u}$ satisfies $S\tau_{\tilde{\mu}} (fmf)=
\tau_{\tilde{\mu}} (m)S$ as well as (1) and (2) in Lemma \ref{lem:Tf}. Hence $S_f=T_f$.
\end{proof}

\begin{theorem} Let $F$ be as above, let $\tilde\mu\in \Lambda^+((L_c^\prime)_0)$ and let $T_f$, $f\in F$, as
in Lemma \ref{lem:Tf}. Define
\[\tau_{\tilde\mu}(fm):=T_f\tau_{\tilde{\mu}} (m),\quad f\in F, m\in (L_c^\prime)_0.\]
Then $\tau_{\tilde\mu}$ is an irreducible representation of $L_c^\prime$.
\end{theorem} 

\begin{proof} We need only show that $\tau_{\tilde\mu} : L_c^\prime \to \GL (V_{\tilde\mu})$
is a homomorphism, $\tau_{\tilde\mu}(fmgn) = \tau_{\tilde\mu}(fm)\tau_{\tilde\mu} (gn)$, $f,g\in F$ and $m,n\in (L_c^\prime)_0$.
But we have
\begin{align*} \tau_{\tilde\mu} (fmgn) &= \tau_{\tilde{\mu}} (fg(gmg)n)\\
&= T_{fg}\tau_{\tilde{\mu}} (gmg)\tau_{\tilde{\mu}} (n)\\
&=T_fT_g\tau_{\tilde{\mu}} (gmg)\tau_{\tilde{\mu}} (n)\\
&=T_f\tau_{\tilde{\mu}} (m)T_g\ws (n)\\
&= \tau_{\tilde\mu} (fm)\tau_{\tilde\mu} (gn) .
\end{align*}
\end{proof}  

The final step, the extension to all of $L_c$ is now easy. We use that $L_c\simeq L_c^\prime \times A$. Hence we can
take any character $\chi$ on $A$ and define
\[\tau_{\tilde\mu,\chi}(ma)=\tau_{\tilde\mu} (m)\chi (a) .\]

\begin{remark}
If one needs to extend $\tau_{\tilde\mu}$ to the complexification $L_c^\C$ of $L_c$, a common compatibility issue arises. $L_c^\C$ is not the direct
product $L_c^{\prime \C}\times A^\C$, one needs to be more careful with the choice of $\chi$. Then the requirement is
that each $T_f$ has to be scalar $c(f)$ and $c(f)=e^{i\dot \chi (H)}$ where $f=\exp iH$. For that one needs to use the
exact form of $F$ to determine possible choices of $\chi$.

On the other hand, since $\fl_c\otimes \C \cong \fk_h\otimes \C$ and we work with finite dimensional representations, a lift from $\fk$ to $\fl_c$ gives a lift from $\fk$ to $\fk_h$.
\end{remark}

 \section{Extension of sections of homogeneous vector bundles}

\noindent
We return to the notation of \S 2.   We consider the generalized flag manifold $ \cP_c=G_c/ L_cN_{-}$ and a basepoint $x_o=eL_cN_{-}$. The $G$ orbit of the basepoint, $G\cdot x_o$, is $\cD\cong G/K$, an open domain in $G_c/ L_cN_{-}$. 

For a unitary representation $(\sigma ,V)$ of $K$ on the complex vector space $V$ we let 
$\mathbb V$ denote the associated homogeneous vector bundle over $\cD$.  Without loss of generality, we can assume that $\sigma$ is 
irreducible, in which case we shall denote by $\mu$ a
highest weight and, as before, by $V_\mu$ its representation space. In \cite{Br} and \cite {Ka} homogeneous vector bundles over certain {\it complex} homogeneous spaces were shown to have an extension to natural compactifications, e.g. the wonderful compactification. In \cite {MSIII} again in the {\it complex setting} in somewhat greater generality homogeneous holomorphic vector bundles over Hermitian (locally) symmetric manifolds were extended to the Borel compactification and a detailed analysis of their restriction to the boundary orbits was obtained. 
We shall give a version of this for the real domain $\cD\subset G_c/ L_cN_{-}$. Here, we just give the extension of  $\mathbb V$  to $\widetilde{\mathbb V}$ over the 
compactification $\overline{\cD}\subset G_{c}/L_cN_{-}$, subsequently we shall analyze the restriction to the boundary orbits.

In the previous section for such $(\sigma_{\mu} ,V_{\mu})$ we produced a natural lift $(\tau_{\widetilde{\mu}},V_{\widetilde{\mu}})$ from $K$ to $L_{c}$ (with some minor exceptions). Then extending the representation trivially on $N_-$ we have an irreducible finite dimensional representation of $L_cN_{-}$. Denote the associated homogeneous bundle over $ \cP_c=G_c/ L_cN_{-}$ by $\widetilde{\mathbb V}_{\widetilde{\mu}}$. Since  $\tau_{\widetilde{\mu}}$ contains $\sigma_{\mu}$ with multiplicity one we have that $\mathbb V$ is a subbundle of $\widetilde{\mathbb V}_{\widetilde{\mu}}$. In particular, $\widetilde{\mathbb V}_{\widetilde{\mu}}$ is defined over $\partial \cD $ and gives an extension of $\mathbb V$ to the boundary of $\cD\cong G/K$.

\section{Analytic extension of $K$-finite matrix coefficients of $G$ to $G_c$}

\noindent
The first task is to construct a  $G$-invariant  domain in $G_c$ that will serve as the domain of
\lq para\rq-analytic (or split-holomorphic) extension of $K$-finite matrix coefficients of $G$.
In \cite{Ma} he provides a general setup for cycle spaces. We shall show that this also gives the target domain in $G_c$.

To prepare for this we recall some previous notation related to various involutions that have played
a role here; to simplify the notation we will omit the dot on involutions on the Lie algebra as it will
always be clear whether we are discussing the Lie algebra or the group. Then we
recall some facts about the crown of a semisimple Lie group, in particular for $G_h$ whose crown will
be denoted by $\Xi_h$, see 
\cite{AG90,KS,KSII}  and especially  \cite[Sec. 7]{KSII}. Once we recall the construction from
\cite{Ma} of a real analytic cycle domain $\Xi_M$ for $G/K$ inside
$G_c/L_c$, we then show that $\Xi_M = (\Xi_h^\eta)_o$ is a totally real submanifold
of $\Xi_h$. We also discuss the connection between the crowns
of $\Gh/\Kh$ and $G/K$, in particular in the case $r=r_h/2$ we show
that $\Xi=(\Xi_h^\tau)_o$. We
conclude the section by proving analytic extension of orbit maps of representations
to the real analytic cycle space thereby justifying the name real analytic crown.

The involution basic to this paper is $ \tau : \gh \to \gh$, giving the real form
$\fg=\fg_h^\tau$ and  $G=(\Gh^\tau)_o$. The eigenspace decomposition w.r.t. $\tau$
is $\gh=\fg \oplus \fq_h$. Recall that the complex linear extension of $\tau$ (or $\thh$) is still denoted $\tau$ (or $\thh$), while the conjugate linear extension of $\tau$ to $\ghC$ is $ \eta= \sigma_h\circ\tau=\tau\circ \sigma_h$. Then
$\fg_c=(\ghC)^\eta$ while $G_c=\GhC^{\eta}$.  $\fg_c$ is a semisimple Lie algebra stable under $\tau$ and $\thh$.  
The resulting eigenspace decompositions are $\fg_c = \fg\oplus \fq_c=\fl_c \oplus \fg_c^{-\thh}$, where 
$\fl_c = \fk \oplus i\fq_{hk}$ and $\fg_c^{-\thh} = \fp \oplus i\fq_{hp}$ (see also the
discussion after Lemma \ref{le-tauZo}). 
We
have $G=(G_h\cap G_c)_o$.
For the restrictions to $\fg_c$, resp. $G_c$, we will still use the notation $\tau$ but introduce  $\tau^a=\thh\vert_{\fg_c} $. 
Notice that $\tau$ and $\tau^a$ commute (because $\tau$ and $\thh$ do). 

The involution
$ \theta_c= \tau\circ \thh\vert_{\fg_c}$ defines a Cartan involution on $\fg_c$ with corresponding Cartan decomposition
$\fg_c=\fk_c\oplus \fs_c$. We have $\fk_c=\fk\oplus i\fq_{hp}$ and $\fs_c=\fp\oplus i\fq_{hk}$ showing
that $\theta_c$  agrees with the conjugate linear 
extension of $\theta_h$ restricted to $\fg_c$. Then it is consistent to denote this on $\fg_h$ by $\tau^a = \tau\circ\theta_h$. It should always be clear which involution is being discussed.

We have $\tau^a=\theta_c\circ\tau$ so our notation
agrees with the standard notation for the involution on $\fg_c$ associated with $\tau$.  
As is standard in this $\R$-form setup $\fl_c^\C=\fk_h^\C$, $\fl_c=\fz_{\fg_c}(H_0)$ and $\tau^a=\Ad (\exp (\pi i H_0))$.

Let, as before, $\fa_h$ be a maximal abelian subgroup of $\fp_h$. Let $\Sigma_h=\Sigma (\fg_h,\fa_h)$, 
 let $\Sigma_h^+$ be a positive system, and take the basepoint to be $x_o=e\KhC\in \GhC/\KhC$. Define  
 
\[\Omega_h=\left\{X\in \fa_h\,\left|\, ( \forall \alpha \in\Sigma_h )\,\, |\alpha (X)|<\frac{\pi}{2}\right.\right\},
\quad \widetilde \Xi_h=G_h\exp (i\Omega_h)\KhC,\quad \text{and}\quad \Xi_h = \widetilde \Xi_h\cdot x_o\, .\]
The $G_h$-invariant set $\Xi_h$ was dubbed by Gindikhin the \textit{crown} of $\Gh/\Kh$. Motivated by the results in \cite{KS} we call
$\widetilde \Xi_h$ the crown of $\Gh$.
The set $\widetilde\Xi_h$ is an open $\Gh$-invariant complex submanifold of $\GhC$.
Similarly, $\Xi_h $ is a $\Gh$-invariant complex domain in $\GhC/\KhC$. $\widetilde{\Xi}_h$ and
$\Xi_h$ are independent of the choice of $\fa_h$ as any two such are $K_h$ conjugate.
Write  $\Omega$, $\widetilde{\Xi}$ and
$\Xi$ for the corresponding sets obtained by this construction for $G$ and $G/K$.

We denote by $\partial \Xi_h$, resp. $\partial \Omega_h$, the topological boundary of
$\Xi_h$, resp. $\Omega_h$. Set
\[\Omega_h^+=\Omega_h\cap \fa_h^+ = \{X\in \Omega_h\mid \forall \alpha \in \Sigma_h^+,\,\, 
\alpha (X)>0\}\quad\text{and}\quad
 \Xi_h^+=\Gh \exp i\Omega_h^+\cdot x_o .\]
Then $\Xi_h^+$ is an open $\Gh$-invariant subset of $\Xi_h$ such that $\overline\Xi_h =(\overline{N_{\Kh}(\fa_h)\Xi_h^+})_o
=(\overline{W_h\Xi_h^+})_o$.

For restricted roots we keep the
 notation from Lemma \ref{le-rnot} and (\ref{def-aq}). Thus
 $\beta_1,\ldots ,\beta_r\in \Sigma( \fg_h,\fa_h)$ are strongly orthogonal roots
 (up to sign they are the Cayley transform of the strongly orthogonal roots $\alpha_j$, per
 the discussion after Theorem \ref{th:Orb}).
We denote by $X_j$, $j=1,\ldots ,r$, the dual basis and as usual we have $\fa_h=\bigoplus \R X_j$.
We also define $Y_j\in \fq_h\cap \fp_h$ as in Lemma \ref{cor:2.11} and let then $\fa_{hq}=\bigoplus \R Y_j$. 

If $r=r_h$
then $\fa_h=\fa$ is maximal abelian in $\fp_h $ and $\fp$, and $\fa_{hq}$ is maximal abelian
in $\fp_h$ and $\fq_h\cap \fp_h$.

If $r=r_h/2$ we choose the ordering so that $\beta_{2j}=\beta_{2j-1}\circ \tau
 =\tau^t \beta_{2j-1}$ and assume, as we may, that $\tau X_{2j-1}=X_{2j}$. Let $X_j^\prime = X_{2j-1}+X_{2j}$,
 $X_{j}^-= X_{2j-1}-X_{2j}$, $\fa=\bigoplus_{j=1}^r \R X_j'$
 and $\fa_h^q=\bigoplus_{j=1}^r \R X_j^-$. Then $\fa$ is maximal abelian in $\fp$ and $\fa_h^q$
 is maximal abelian in $\fp_h\cap \fq_h$.
 We let 
 \[\gamma_{2j-1}=\frac{1}{2}\left(\beta_{2j-1}+\beta_{2j}\right)\quad\text{and}\quad \gamma_{2j}=
 \frac{1}{2}\left(\beta_{2j-1}-\beta_{2j}\right),\quad j=1,\ldots ,r\]
 and note, that according to Moore's theorem $\gamma_k \in \Sigma^+(\fg_h,\fa_h)$.
Note that previously the notation $\gamma_j$ was used
for strongly orthogonal roots in $\Delta$. We
note that $\gamma_{2j-1}|_{\fa}=\beta_{2j-1}|_{\fa}=\beta_{2j}|_{\fa}\not= 0$ and
 $\gamma_{2j}|_{\fa_h^q}=\beta_{2j-1}|_{\fa_h^q}=-\beta_{2j}|_{\fa_h^q}\not= 0$.
 
Let's recall that  $\theta_h$ is inner, in particular $\theta_h=\Ad (\exp \pi Z_h)$. Then
\[\tau^a =\Ad (\exp \frac{\pi}{2} Z_h)\circ \tau\circ \Ad (\exp \frac{\pi}{2} Z_h) .\]
Thus $\fg$ and $\fg_h^{\tau^a}$, resp. $\fq_h$ and $\fg_h^{-\tau^a}$, are conjugate. Statements
that are formulated for $\tau$ and its eigenspaces are therefore also valid for $\tau^a$ and
its eigenspaces.

The next result can be gleaned from \cite{KSII}. 
\begin{theorem}\label{the:crown} Let the notation be as above. Then the following holds true:
\begin{itemize}
\item[(a)] $\Omega_h=\{\sum_{j=1}^r t_j X_j\mid (\forall j\in \{1,\ldots ,r_h\}) \, \, |t_j|< \pi/2\}$.
\item[(b)] We have  $g_1\exp iX_1\cdot x_o=
g_2\exp iX_2\cdot x_o$ for some  $g_1,g_2 \in \Gh$ and $Y_1,Y_2\in \Omega_h$, if and only if there exists
$k\in Z_{K_h}(Y_1)$ and $w\in N_{K_k}(\fa_h)$ such
that $g_1=g_2wk$ and $Y_1=\Ad(w^{-1})Y_2$.
\item[(c)] If $g_1\exp iX_1\cdot x_o=
g_2\exp iX_2 \cdot x_o\in \Xi_h^+$ then $X_1=X_2$ and  there exists $m\in Z_{K_h}(\fa_h)$ such that
$g_1=g_2m$.
\item[(d)] If $x_n =g_n\exp Y_n\cdot x_o\in \Xi_h $ is a sequence
such that $x_n \to \partial \Xi_h\subset \GhC/\KhC$ then $Y_n \to Y\in \partial \Omega_h$.
\item[(e)] $\widetilde{\Xi}_h\subset \mathbb{N}_{h} \mathbb{A}_h \KhC$. \end{itemize}
\end{theorem}
 \begin{proof}
 (a) is the comment after \cite[Lem. 7.4]{KSII} and follows
easily from Moore's Theorem; (b) is  \cite[Prop. 3.1]{KSII}; (c)  is   \cite[Cor. 4.2]{KSII};
  (d) is \cite[Lem. 2.3]{KSII} and (e) is (1.1). 

\end{proof}

Next we recall the construction of the cycle space or \textit{Matsuki crown} of
$G/K$ in $G_c/L_c$, per \cite{Ma}.

\begin{remark} To assist the reader we give the correspondence between the notation in \cite{Ma} with our setup. Here the left hand side lists Matsuki's notation and the right hand side the corresponding
object in this article: $\fg\leftrightarrow \fg_c$, $\fh\leftrightarrow \fg$, $\fh^\prime \leftrightarrow \fl_c$, $\fk\leftrightarrow 
\fk_c$, $\fm\leftrightarrow \fs_c$, $\fq\leftrightarrow \fq_c$. Similarly for the groups. In particular
$G\leftrightarrow G_c$ and 
$H\leftrightarrow G$. As $L_c=Z_{G_c}(H_0)$ might be  disconnected, so $H^\prime \leftrightarrow L_{c0}$.
\end{remark}

 Let $\ft $ be a maximal abelian subspace of
$\fk_c\cap \fq_c = i\fq_{hp}$. Denote by $\widetilde \Sigma (\fg_c^\C,\ft^\C)$
the roots of $\ft^\C$ in $\ghC \cong \fg_c^{\C}$. As $\theta_c|_{\ft^\C}=\id_{\ft^\C}$, given a root space
$\fg_c^{\C\alpha}=\fg_c^{\C}(\ft, \alpha)$ we have $\theta_c(\fg_c^{\C}(\ft, \alpha))=\fg_c^{\C}(\ft, \alpha)$ and one
decomposes it according to the 
eigenvalues of $\theta_c$ getting $\fg_c^{\C}(\ft, \alpha) =
 \fk_c^{\C}(\ft, \alpha)\oplus \fs_c^{\C}(\ft, \alpha)$. Let $\widetilde 
\Sigma_c(\fs_c^{\C},\ft) = \{ \alpha\in i\ft^*\vert \fs_c^{\C}(\ft, \alpha)\ne \{0\}\}$. Finally set 
\[\Omega_M =\left\{ Y\in \ft\, \left|\, (\forall \alpha \in \widetilde
 \Sigma_c(\fs_c^{\C}, \ft^\C))\,\, 
|\alpha(Y)| <\frac{\pi}{2}\right.\right\}.\]
As before we define $\Omega^+_M$ as the intersection of $\Omega_M$ with a positive Weyl chamber. 

Let  $T(\Omega_M ) = \exp \Omega_M \subset T =\exp  \ft$,  $T(\Omega_h) =\exp i\Omega_h,\,$ and define 
$$\wXi_M = GT(\Omega_M )L_c \subset G_c \text{ and } \Xi_M = \wXi_M\cdot x_o\subset G_c/L_c.$$ 

\begin{theorem}[Matsuki] $\wXi_M$ is open in $G_c$ and $\Xi_M $ is connected and open in $G_c/L_c$.
\end{theorem}

\begin{proof}
This will follow from \cite[Prop. 1]{Ma} using the dictionary above. For that we need some material about 
$\tau=\tau^a\circ \theta_c$-stable parabolic subalgebras in $\fg_c$.    Let, as
before, $\fa\subset \fp$ be maximal abelian. Let $\Sigma=\Sigma (\fg_c,\fa)$ denote the set
of roots of $\fa$ in $\fg_c$ and let $\Sigma^+  $ be a set of positive roots. Define
$\widetilde\fn =\bigoplus_{\alpha \in\Sigma^+ (\fg_c,\fa)}\fg_{c\alpha} $ and
$\widetilde{\fm}=$ the orthogonal complement of $\fa$ in $\fz_{\fg_c}(\fa)$.
Then $\widetilde{\fp}_c=\widetilde\fm \oplus \fa \oplus \widetilde{\fn}$ is
a minimal $\theta_c\circ \tau^a$ stable parabolic subalgebra in $\fg_c$, see \cite{Ma79} or \cite{vdB}.
Let $\wt{P}_c=\wt{M}_cA\wt{N}_c$ be the corresponding minimal $\theta_c\circ \tau^a$ stable
parabolic subgroup. 
That $L_c\wt{P}_c$ is open in $G_c$ follows from \cite{Ma79}. Hence by Matsuki duality \cite{Ma79} $G\wt{P}_c$
is closed. Now compare this with the assumption on \cite[p. 565]{Ma} and we see that we can take $\wt{P}_c$ for
the parabolic $P$ in \cite{Ma} or \cite{vdB}, i.e. $P\leftrightarrow\wt{P}_c$.
\end{proof}

The main result in \cite{Ma} is 
\begin{theorem}[Matsuki]\label{th:Ma1} Set $S=S(G_c\widetilde{P}_c;L_c\widetilde{P}_c) = \{ x\in G_c\mid x^{-1} 
G\widetilde{P}_c \subset L_c\widetilde{P}_c\}.$ Then $S $ is open and if $S_0$ denotes a connected 
component, we have
$\wXi_M \subseteq S_0$.
\end{theorem}

We mention a slightly different interpretation of $\Omega_M $. We refer to \cite[Chap. 5]{HO96} for a
more detailed discussion. The abelian Lie algebra $\ft $ is a maximal abelian subspace of 
$\fk_c\cap \fq_c = i\fq_{hp}$. But the generalized flag manifold $G_{c}/L_cN_{-}$ is diffeomorphic to 
$K_c/K_c\cap L_c \cong K_c/FK$ which is a Riemannian symmetric space. We have
even more, $L_cN_-$ is a maximal parabolic subgroup with abelian nilradical, $\fn_-$. Hence
$K_c/K_c\cap L_c\simeq G_c/L_cN_-$ is a symmetric $R$-space. Note that $\Ad (F)$ 
normalizes $\fk$ and hence $FK$ is a group. Furthermore, $(FK)\cap G=K$. Since $\fk_c=\fk\oplus i\fq_{hp}$ 
we have the tangent space at $eFK$ is given by $i\fq_{hp}$. Thus $\ft$ is the Lie algebra of a maximal torus 
(an Iwasawa torus) in the tangent space. But we have the open embedding $\cD\simeq G/K\subset G_c/ L_cN_{-}$ 
so the tangent space at eK can be identified with $\fp$ which has maximal abelian subalgebra $\fa$. On 
the other hand, $G_c/ L_cN_{-} \simeq K_c/FK$ is, up to covering, the compact dual symmetric space to $G/K$. 
Thus within $\fg^{\C}$ there is an $\R$-isomorphism $\phi:\fa_c \oplus \ft \to \fa_c^{\C}$, i.e. between the split-complexification and the complexification. 

As $\eta (G_h)=G_h$, $\eta (K_h)=K_h$ and we can choose $\fa_h$ so that $\eta (\fa_h)= \fa_h$ it
follows that $\eta (\wXi_h ) = \wXi_h$ and $\eta (\Xi_h) = \Xi_h$. As
$G_c/L_c = G_c\cdot x_o$ it follows that we can view $G_c/L_c$ as a real 
form of $\GhC/\KhC$.  We note that $\wXi_M$ is not connected unless $L_c$ is, but
$(\wXi_M)_o = G_c\exp \Omega_M (L_c)_o$.  

\begin{remark} The Matsuki crown is defined with respect to $G_c$. To connect the 
notation to $\fg_h$ we make some additional observations. First notice that if $\fa_1$ is maximal abelian in
$\fq_h\cap \fp_h=\fp_h^{\tau^a}$ if and only if  $\ft=i\fa_1$  is maximal abelian in $\fq_c\cap \fk_c=i(\fq_h\cap \fp_h)$.
As $\fs_c=i(\fq_h\cap \fk_h)\oplus \fp$ we have $\widetilde{\Sigma}(\fs_c^\C,\ft^\C)
=\Sigma(\fg_h^{-\tau^a},\fa_1)$ and
\[\Omega_M=i\Omega_{hq}\quad \text{where}
\quad \Omega_{hq}=\{X\in \fa_1\mid (\forall \alpha \in \Sigma(\fg_h^{-\tau^a},\fa_1)\,\, |\alpha (X)|< \pi/2\},\]
 quite analogous to the construction in the group case. This shows that there is
a fundamental difference between the case $r=r_h/2$ and $r=r_h$. In the first case we can
take $\fa=\bigoplus \R X_j^\prime$ as before, and $\fa_1=\fa_h^q=\bigoplus \R X_j^-$. In
particular $\fa$ and $\ft$ commute as in the group case. For $r=r_h$ the space $\fa$ is already maximal
abelian so there is no way to chose $\ft$ so that $\fa$ and $\ft$ commute.
\end{remark}

If $r=r_h$ we always have  $\Sigma (\fg,\fa)\subseteq \Sigma(\fg_h,\fa)$ and
$\Sigma (\fg_h^{-\tau^a},\fa_{hq})\subseteq \Sigma(\fg_h,\fa_{hq})$ which implies that
$\Omega_h\subseteq \Omega$.
So if we define $\Omega_h$ as a
subset of $\fa_{hq}$, $\Omega_h\subseteq \Omega_{hq}$.  Similarly, if $r=r_h/2$,
as $\Omega$ and $\Omega_{hq}$ are defined by via restriction of roots in $\Sigma (\fg_h,\fa_h)$ to
$\fa$, resp. $\fa_h^q$, and because $\Omega_h$ is invariant under $\tau$ and $-\tau$ it follows that
$\Omega_h\cap \fa=\pr_{\fg}(\Omega_h)\subseteq \Omega$ and
$\Omega_h\cap \fa_{h}^q=\pr_{\fq_h}(\Omega_h)\subseteq \Omega_{hq}$. Here $\pr_\fg$
is the projection along $\fq_h$ onto $\fg$ and
$\pr_{\fq_h}$ is the projection along $\fg$ onto $\fq_h$. This clearly implies that we always have
$\Xi\subseteq (\Xi_h^{\tau})_o$ and $\Xi_M \subseteq (\Xi_h^\eta)_o$.
\begin{lemma}\label{le:CrEq} Let the notation be as above.
\begin{itemize}
\item[(a)] Assume that $r=r_h$ and $\Omega_h=\Omega$. Then
$\Xi=(\Xi^\tau_h)_o$.
\item[(b)] Assume that $r=r_h$ and $\Omega_h=\Omega_{hq} \subset \fa_{hq}$. Then
$\Xi_M = (\Xi^\eta_h)_o$.
\item[(c)] Assume that $r=r_h/2$ and $\Omega_h\cap \fa=\Omega$. Then
$\Xi=(\Xi^\tau_h)_o$.
\item[(d)] Assume that $r=r_h/2$ and $\Omega_h\cap \fa_h^q=\Omega_{hq} $. Then
$\Xi_M = (\Xi^\eta_h)_o$.
\end{itemize}
\end{lemma}
\begin{proof} We prove only (c) and (d). The proofs of (a) and (b) are simpler following the same line
of argument.

(c): We have $\Xi=G\exp (\Omega)\cdot x_o\subset G_h\exp \Omega_h\cdot x_o=\Xi_h\subset \GhC/\KhC$.
Taking $\tau$ fixed points implies the inclusion $\Xi\subseteq (\Xi_h^\tau)_o$. As $\GhC$ is simply connected it
follows that $\GC=\GhC^\tau$. Hence $(\Xi_h^\tau)_o = (\Xi_h\cap \GC/\KC)_o$. As $\Xi$ is open in
$\GC/\KC$ it follows that $\Xi$ is open in $(\Xi_h^\tau)_o$. Assume that
$\Xi$ is not closed in $(\Xi_h^\tau)_o$. Then there
exists a sequence $\xi_j=g_j\exp Y_j\cdot x_o$, $g_j\in G$, $Y_j\in\Omega$, such
that $\xi_j\to \xi \in \partial \Xi \cap (\Xi_h^\tau)_o$.  According to Theorem \ref{the:crown} part (d)
there exists $Y\in \partial \Omega$ such that $Y_j\to Y$. Hence there exists $\alpha \in \Sigma (\fg,\fa)$
such that $|\alpha (Y_j)|\to \pi/2$. Let 
\[\fg_{h\alpha}=\{X\in \fg_h\mid (\forall H\in\fa)\,\, [H,X]=\alpha (H)X\}\not= \{0\}\, .\]
Then $\fg_{h\alpha}$ is $\ad (\fa_h)$ invariant. It follows that there exists
$\widetilde{\beta}\in\Sigma(\fg_h,\fa_h)$ such that $\widetilde{\beta}_\fa=\alpha$.
Thus $Y\in\partial\Xi_h$ contradicting the assumption that $\xi \in \Xi_h$. Thus $\Xi$ is
closed in $(\Xi_{h}^\tau)_o$.

Part (d) follows in the same way replacing $\tau$ by $\eta$ and in the last argument replacing $\fa$ by
$\fa_h^q$.
\end{proof}

\begin{lemma}\label{prRootSp}  
Assume that $r=r_h/2$. Write $\fa_h =\fa \oplus \fa_h^q$ and let $\beta\in\Sigma (\fg_h,\fa_h)$.  If 
$\beta|_{\fa_h^q}\not= 0$ and $H\in \fa_h^q$ is so that $\beta (H)=1$ then $\ad H : \pr_\fg (\fg_{h\beta})\to 
\pr_{\fq_h}(\fg_{h\beta})$ is an isomorphism. In particular, if $\beta|_{\fa}\not= 0$, then
$\{0\}\not= \pr_{\fg} (\fg_{h\beta})\subseteq \fg_{\beta|_{\fa}}$ and $\{0\}\not=\pr_{\fq_h}(\fg_{h\beta})
\subseteq (\fg_{h}^{-\tau^a})_{\beta|_{\fa_h^q}}$.
\end{lemma}
\begin{proof} Let $X=X_g+X_q\in \fg_{h\alpha}$ with $X_g=\pr_\fg (\fg_{h\beta})$ and
$X_q=\pr_{\fq_h}(\fg_{h\beta})$. Then $\ad H(X)=X=[H,X_g]+[H,X_q]$. As $[H,X_g]\in \fq$ and $[H,X_q]\in\fg$
it follows that $[H,X_g]=X_q$ and $[H,X_q]=X_g$. The last part follows by replacing $\tau$ by $\tau^a$ which
interchanges the role of $\fa$ and $\fa_h^q$.
\end{proof}

\begin{lemma}\label{le:OmegaEqual} We have the following.
\begin{itemize}
\item[(a)] Assume  that $r=r_h$ then we have:
\begin{itemize}
\item[(a-i)] If $\beta_j\in \Sigma (\fg,\fa)$ for all $j=1,\ldots ,r$ then $\Omega=\Omega_h$.
\item[(a-ii)] If $\beta_j\in\widetilde{\Sigma}(\fg^{-\tau^a},\fa_{hq})$ then $\Omega_M = i\Omega_{hq}$.
\end{itemize}
\item[(b)] If $r=r_h/2$ then we have:
\begin{itemize}
\item[(b-i)] If $\gamma_{2j}|_{\fa}=\beta_{2j}|_{\fa} \in \Sigma (\fg,\ga)$, $j=1,\ldots ,r$ then
$\Omega=\Omega_h\cap \fa$.
\item[(b-ii)] If $\gamma_{2j-1}|_{\fa_h^q}=\beta_{2j}|_{\fa_h^q}\in \widetilde{\Sigma}(\fg_h^{-\tau^a},\fa_h^q)$ then
$\Omega_M = i \Omega_{hq}$.
\end{itemize}
\end{itemize}
\end{lemma}
\begin{proof} This follows directly from Moore's Theorem. 
For example consider (b-ii). We
only have to show that $\Omega_h\cap \fa_h^q\subseteq \Omega_{hq}$. Let $X=\sum_{j=1}^{r_h}
t_j X_j\in \Omega_h\cap \fa_h^q$. Then $|t_j|<\pi/2$ for all $j$. Furthermore
$X = -\tau X$. Hence $X=\sum_{j=1}^r t_{2j-1} (X_{2j-1}-X_{2j})$ and hence $|\beta_{2j-1}(X)|
<\pi/2$. The claim now follows from Moore's
Theorem as all the roots in $\widetilde{\Sigma}(\fg_h^{-\tau^a},\fa_h^q)$ are restrictions of
roots in $\Sigma (\fg_h,\fa_h)$
\end{proof}

Finally we come to the relationship of various crowns.

 \begin{theorem}\label{th:Omega}
 Let the notation be as above. Then the following holds:
 \begin{itemize}
 \item[(a)] If $r=r_h$ then $\Omega_h =\Omega_{hq}$ and
 $\Xi_M =(\Xi_h^\eta)_o$.
 \item[(b)] If  $r=r_h/2$. Then $\Omega=\Omega_h\cap \fa$ and
 $\Omega_M =\Omega_h\cap \fa_h^q$, Furthermore $\Xi_M =(\Xi_h^\eta)_o$ and
  $\Xi =(\Xi_h^\tau)_o$.
  \end{itemize}
 \end{theorem}
 \begin{proof} Lemma  \ref{le:CrEq} and Lemma \ref{le:OmegaEqual} imply that we  have to show that $\beta_j|_{\fa}\in \Sigma (\fg,\fa)$ respectively
 $\beta_j|_{\fa_h^q} \in\widetilde{\Sigma}(\fg_h^{-\tau^a},\fa_h^q)$, for
 $j=1,\ldots ,r_h$. For (a) we use 
 $\su(1,1)$-reduction for $\tau^a$ to show that $\beta_j\in \widetilde{\Sigma}(\fg_h^{-\tau^a},\fa_{hq})$.
 For (b) this follows from Lemma \ref{prRootSp} as each $\beta_j$ has a non-zero restriction to
 $\fa$ and $\fa_h^q$. 
 \end{proof}

\begin{Basic Example}  $\SU (1,1)$ - cont.\hfill

Recall that $\fg_h = \su( 1,1) = \fk_h \oplus \fp_h = \fg \oplus \fq_h$, where 
\[\fk_h=\R \begin{pmatrix} i & 0 \\ 0 & -i\end{pmatrix}, \quad \fp_h =\R\begin{pmatrix} 0 & 1 \\ 1 & 0\end{pmatrix} 
\oplus \R\begin{pmatrix}0 & i\\-i &0\end{pmatrix} = \R\cdot X\oplus \R\cdot Y,\]
while $\fg = \R\cdot X$ and $\fq_h = \fk_h \oplus \R\cdot Y $. As before, $\SU (1,1)$ 
has two natural choices of Iwasawa: $A = \exp \R X$ or $A_{hq} = \exp \R Y.$

From \cite {KS} we know that either choice gives, with the obvious notation,
\[ G_h \exp \Omega  \,\K_h \subset \N \mathbb {A}\K_h \quad \text{and}\]
\[G_h \exp i\Omega_{hq} \,\K_h \subset \N_{hq} \mathbb {A}_{hq}\K_h .\]

Also from before we have $T (\Omega ) = \exp i\Omega = \exp i\Omega_{hq}.$ Since $\fa$ and
$\fa_{hq}$ are conjugate via $K_{h}$ we have 
 \[ G_h \exp i\Omega \,\K_h = G_h \exp i\Omega_{hq} \,\K_h = G_h T (\Omega) \,\K_h.\]
 
Taking fixed points of the conjugate linear $\eta$ gives \[G T (\Omega ) L_c \subset
(G_h T(\Omega) \,\K_h)^\eta\subset (\N\mathbb{A} \K_h)^{\eta} \subset G_c  ,\]
where 
\[N = \exp \R\begin{pmatrix} i & -i \\ i & -i\end{pmatrix} = \left\{\left.\begin{pmatrix} 1+iv & -iv \\ iv & 1-iv\end{pmatrix}
\, \right|\, v\in \R\right\}.\]
Now $G T (\Omega) (L_c)_0$ is connected and contains the identity. 

Take $nak \in \N\mathbb{A}\K_h$. Then 
\[nak = \begin{pmatrix} 1+iv & -iv \\ iv & 1-iv\end{pmatrix} \begin{pmatrix} \cosh (z) & \sinh (z) \\ \sinh (z) & \cosh (z)\end{pmatrix}\begin{pmatrix} w & 0 \\ 0 & w^{-1}\end{pmatrix}, \quad v,z \in \C, w\in \C^*.\]

Multiplication gives
\[nak = \begin{pmatrix} w\cosh(z) +w iv e^{-z} & w^{-1}\sinh(z)-w^{-1}iv e^{-z} \\ w iv e^{-z} + w \sinh(z)
 & w^{-1}\cosh(z) - w^{-1} iv e^{-z}\end{pmatrix} =\begin{pmatrix} a & b \\ c & d \end{pmatrix}.\]

On the other hand, if $gtl \in GT (\Omega) (L_c)_0$ then
\[ gtl =  \begin{pmatrix}\cosh (t) & \sinh (t) \\ \sinh (t) & \cosh (t)\end{pmatrix} \begin{pmatrix}\cos (\theta) &
\sin (\theta) \\ -\sin (\theta) & \cos (\theta)\end{pmatrix}\begin{pmatrix} e^r & 0 \\ 0 & e^{-r}\end{pmatrix}.\]
 
 Since in this example $G=A$ and $L_c\subset \K_h$ it suffices to express $t$ in terms of $nak$. An elementary, though tedious, computation gives the following solutions provided $0\le \vert\theta\vert<\frac {\pi}{4}$:

\begin{align*}
 e^{-2u} &= \cos(2\theta), & z=u\in\R\\
e^x &= \frac{\cos(\theta) + \sin(\theta)}{\cos(2\theta)^\frac{1}{2}}, &w = e^x\in\R \\
iv=s &= -\frac{e^u\sin(2\theta)}{2\cos(2\theta)^\frac{1}{2}}, & s\in\R.
\end{align*}

With these substitutions it is straightforward to verify that $t = nak$ with $n\in \N\cap 
G_c, a\in A, k\in (L_c)_0$.  Also notice that $[Y, \begin{pmatrix} i & -i \\ i & -i\end{pmatrix}] = 2 \begin{pmatrix} i & 
-i \\ i & -i\end{pmatrix}$, thus $\vert\theta\vert<\frac {\pi}{4}$ is the full range to describe 
$T(\Omega )$. Thus  $G T (\Omega ) L_c \subset (\N\cap G_c) A L_c$.
\end{Basic Example}
 
\begin{example}[Cayley Type Spaces] There are examples where $r=r_h$ and
$\Omega=\Omega_h$. The simplest case is the rank one case $(\so (1,n), \so(1,n-1))$ with $n\ge 3$. 
But the following example shows that we have no general statement in this case.
 Assume that $\fg =\fg^\prime \oplus \R H_0$ is
not simple with $\fg^\prime =[\fg,\fg]$ simple. Then $\fa_h = \fa = \fa^\prime
\oplus \R H_0$ with $\fa^\prime =\fa\cap \fg^\prime$. We have by
Moore's Theorem, Theorem \ref{th-Moore}, we have
\[\Sigma (\fg,\fa)=\{\frac{1}{2}(\beta_i-\beta_j)\mid i\not= j\}
\quad
\text{and}\quad \Sigma(\fg_h^{-\tau},\fa )=
\pm\{ \frac{1}{2}(\beta_i+\beta_j)\mid i,j=1,\ldots ,r=r_h\}.\]

Let again $X_1,\ldots ,X_r$ be so that $\alpha_i(X_j)=\delta_{ij}$ and use those as coordinate axes. Then
$\Omega_h=(-\pi/2,\pi/2)^r$.
On the other hand the condition for $\Omega$ is $\frac{1}{2}|x_i-x_j|<\pi/2$. Thus $\Omega_h\subsetneqq
\Omega$.
Interchanging $\tau$ and $\tau^a$ we see that $\Omega_h=\Omega$ which again leads to
$\Xi=\Xi_h^\eta$.
\end{example}
   
 Given $(\pi, E)$ an irreducible Banach representation of $G$ and a $K$-finite vector $v\in E$ 
Theorem 3.1 in \cite{KS} states that the orbit map $g\to \pi(g)v$ has a holomorphic extension to the domain
$\wXi\subset \mathbb{G}$. There is an analogous result here with the domain $\wXi_M$ just constructed and the group $G_c$ in place of $\mathbb{G}$. 
 
 \begin{theorem} Let $(\pi,E)$ be an irreducible Banach representation of $G$, and let $v\in E$ be
 a $K$-finite vector. Then the map $g\to \pi(g)v$ has an analytic extension to $(\wXi_M)_0=
 GT (\Omega_M ) (L_c)_0\subset G_c.$
 \end{theorem}
\begin{proof}   The key to the result is that $(L_c)_0$ and $G$ have the same maximal compact subgroup $K$. First we consider the case $r = r_h$. Then $\fa_h$ and $\fa_{hq}$ are $K_h$ conjugate, so  $ G_h \exp (i\Omega_h) \,\K_h = G_h \exp i\Omega_{hq} \,\K_h = G_h T (\Omega ) \,\K_h.$ From \cite{KS}, $G_h T (\Omega_h ) \,\K_h\subset
 \N_h \mathbb {A}_h\K_h $ is open and the projection maps to  $\mathbb {A}_h$ and $\K_h $ are holomorphic.  Now $\wXi_{M0} =GT (\Omega_M ) L_{c0}\subset S_0\subset [(G_h T(\Omega_h \,\K_h))^\eta]_0 \subset G_c.$
The restriction of the projection maps to  $\mathbb {A}_h$ and $\K_h $ gives analytic maps to $A = (\mathbb {A}_h)^\eta$ and $L_{c0} = (\K_h)^\eta$ but as $\wXi_M$ is connected, to $L_{c0}$. 
 Since $r = r_h$, $\fa\cong \fa_h$ is also an Iwasawa for $G$,  Denote the map to $L_{c0}$ by  $\ell$. 
 Since both $L_{c0}$ and $G$ have the same maximal compact subgroup, $K$,  composition of $\ell$ with the usual $\kappa $ projection of  $L_{c0}$ to $K$ gives an analytic map from $ \wXi_{M0} =GT(\Omega_M) L_{c0}
 \subset [(G_h T (\Omega)\,\K_h)^\eta]_0$ to $K$. With these analytic maps from $\wXi_{M0} $
 to $A$ and $K$ we are now in the position of the proof of Theorem 3.1 in \cite{KS} and can continue it verbatim to obtain the result.

If $r\ne r_h$ then as we have seen $r = \frac{r_h}{2}$. As in Lemma 2.12 $\fa_h = \fa \oplus \fa_q^h$ as a Lie algebra direct sum, i.e. $\fa$, $\fa_q^h$ are abelian and $[\fa,\fa_q^h] = 0$. Also from the Lemma we have $\eta$, restricted to $\fa_h$,  is one on $\fa$ and $-1$ on $\fa_q^h$. Then the conjugate linear extension $\eta$ is one on $\fa \oplus i\fa_q^h$, i.e. $\mathbb A_h^\eta = A\exp i\fa_q^h$ with $A\subset G$. Thus $\wXi_{M0} =GT (\Omega_M ) L_{c0}\subset [ (G_h T(\Omega_h \,\K_h)^\eta]_0\subset G_c$ Again has the restriction of the holomorphic projection maps taking values in $K$ and $ \exp i\fa_q^h$ with the latter isomorphic to $\exp i\fa$. Thus here to we are in the position of Theorem 3.1 of \cite{KS}.

\end{proof} 
 \begin{remark} In the \textbf{Basic Example} $G\cong \R^*$, the representations of $G$ are just characters, so from the above expression the continuation of the characters to $G T^+ L_c $ as just translation in the variable by  $-\frac{1}{2}$ log cos($2\theta$). 

\end{remark}

\begin{example}[The case of  $\SU(m,1)$]
We will show that the computations for \SU(m,1) reduce to those of the Basic Example. Here $\fg_h = \su( m,1) = \fk_h \oplus \fp_h = \fg \oplus \fq_h$, where 
\[\fk_h= \begin{pmatrix} A & 0 \\ 0 & -tr (A)\end{pmatrix}, ({\rm with} \, A= -A^*), \quad \fp_h =\begin{pmatrix} 0 & Z \\ Z^* & 0\end{pmatrix} ({\rm with}\,  Z\in \C^m)\] and $\fg = \so(m,1)$. Using obvious block matrices let 
$$X =\begin{pmatrix}0 & 0 & 0\\ 0 & 0 & 1\\0 & 1 & 0 \end{pmatrix}, \, Y = \begin{pmatrix} 0 & 0 & 0\\ 0 & 0 & i\\0 & -i & 0 \end{pmatrix}.$$
As before, $\SU (m,1)$ has two natural choices of Iwasawa: 
\[ A = \exp \R X  (X \in \fg ) \text{ or } A_{hq} = \exp \R Y (Y\in \fq_h).\]

We will do the computations for \SU(3,1) for then the procedure for \SU(m,1) will be clear. Either choice of Iwasawa gives \[ \qquad G_h \exp i\Omega \,\K_h \subset \N \mathbb {A}\K_h \quad \text{and}\]
\[G_h \exp i\Omega_{hq} \,\K_h \subset \N_{hq} \mathbb {A}_{hq}\K_h ,\]
\[ \text { moreover} G_h \exp i\Omega \,\K_h = G_h \exp i\Omega_{hq} \,\K_h = G_h T (\Omega ) \,\K_h\]
where $T (\Omega )= \exp i\Omega = \exp i\Omega_{hq}$.   Taking fixed points of
the conjugate linear $\eta$ gives
\[G T (\Omega ) L_c \subset (G_h T (\Omega ) \,\K_h)^\eta\subset (\N\mathbb{A} \K_h)^{\eta} \subset G_c  ,\]
where 
\begin{align*}
N &= \exp \left\{\begin{pmatrix} 0 & 0 & -Z_1 & Z_1 \\ 0 & 0 & -Z_2 & Z_2 \\ 
\overline{Z_1} & \overline{Z_2} & iv & -iv\\ \overline{Z_1} & \overline{Z_2} & iv & -iv \end{pmatrix}\right\} \\
&= \left\{
\begin{pmatrix} 1 & 0 & -Z_1 & Z_1 \\ 0 & 1 & -Z_2 & Z_2 \\ 
\overline{Z_1} & \overline{Z_2}  & 1 + iv -\frac{1}{2} (|Z_1|^2 + |Z_2|^2) & -iv + \frac{1}{2} (|Z_1|^2 + |Z_2|^2)\\ \overline{Z_1} 
& \overline{Z_2} & iv -\frac{1}{2}(|Z_1|^2 +|Z_2|^2) & 1 -iv +\frac{1}{2}(|Z_!|^2 + |Z_2|^2) \end{pmatrix}\right\},
\end{align*}
($Z_i \in \C, v\in \R)$ while $L_c = \begin{pmatrix} GL(3,\R) & 0 \\ 0&1/ \det\end{pmatrix}$.  Again in $G T^+ L_c $ it suffices to consider only the $t$ term. Now in $(G_h T^+ \,\K_h)^\eta $ the right action by $L_c$ has the effect of multiplying the last column by $det^{-1}$,  but
\[t = \begin{pmatrix}1 & 0 & 0 & 0\\ 0 & 1 & 0 & 0\\ 0 & 0 &\cos (\theta) & -\sin (\theta)\\ 0 & 0 & \sin (\theta)
  & \cos (\theta)\end{pmatrix}.\]
 Consequently we must have $Z_1 = 0 = Z_2$, reducing  the computations to the case \SU(1,1) thus obtaining essentially the same formulae for $t = nak$ as before. In particular, $G T (\Omega) L_c \subset (\N\cap G_c) A L_c \subset (\N\mathbb{A} \K_h)^{\eta}.$
 \end{example}
 
\appendix 
\section{The Classification}\label{se-Classification}
In the following tables we set $\gl_+(n,\C)= \sl (n,\C)\oplus \R\id$ and $\ft =i \R=$ the Lie algebra of the torus $\mathbb{T}=\{z\in\C\mid |z|=1\}$.

\begin{table}[!h]
\center
\begin{tabular}{|c|c|c|c|c|}
\hline
$\fg_c$ & $\fg_h$&$\fg$ \\
\hline
$\sl (p+q,\C)$ & $\su (p,q)\times \su (p,q)$& $\su (p,q)$\\
$\so (2n,\C)$ & $\so^*(2n)\times \so^* (2n)$& $\so^*(2n)$\\
$\so(n+2,\C)$& $\so (2,n)\times \so (2,n)$&$\so (2,n)$\\
$\sp (n,\C)$& $\sp (n,\R)\times \sp (n,\R)$ & $\sp (n,\R)$\\
$\fe_6$& $\fe_{6(-14)}\times \fe_{6(-14)}$& $\fe_{6(-14)}$\\
$\fe_7$& $\fe_{7(-25)}\times \fe_{7(-25)}$& $\fe_{7(-25)}$\\
\hline
\end{tabular} 
\medskip

\caption{$\fg$ with complex structure (group case)}
\end{table}
\medskip

In Table 4 the items listed below the line are those where
$G_h/K_h$ is a tube type domain and $\fg_c \cong\fg_h$. That happens if and only if $\fg\simeq \fl_c$ if and only if $\fg$ has a one-dimensional center. We denote the compact real form of $E_6$ by $\fe_6$. We also note that $\sl (n,\R)\times \R=\gl (n,\R)$ but we
write it using $\sl (n,\R)\times \R$ so that it fits better into the general picture. Same comments hold for $\fu (n)$ and $\su (n)\times \ft$.

\begin{table}[!h]
\center
\scriptsize{
\begin{tabular}{|c|c|c|c|c|}
\hline
$\fg_c$ & $\fg_h$&$\fg$ &$ \fl_c $ & $\fk_h$\\
\hline
$\sl (p+q,\R)$& $\su (p,q)$& $\so (p,q)$&$\fs (\gl (p,\R)\times \gl (q,\R) )$ &$\fs (\fu (p)\times \fu (q))$\\
$\su^*(2(p+q))$&$\su (2p,2q)$& $\sp (p,q)$& $\su^* (2p)\times \su^* (2q)\times \R$ &$\fs (\fu (2p)\times \fu (2q))$\\
$\so (n,n)$& $\so^*(2n)$& $\so (n,\C)$& $\sl (n,\R)\times\R$ & $ \su (n)\times \ft$\\
$\so (1,q+1) (q\ge 3) $ & $\so (2,q)$ & $\so (1,q)$ &  $\so (q)\times \R$ & $ \fs(\so (2)\times \so (q))$\\
$\so (p+1,q+1) (q\ge p\ge 2)$& $\so (2,p+q)$&$\so (1,p)\times \so (1,q)$& $\so (p,q)\times \R$ & $\so (2)\times \so (p+q)$\\
$\sp (n,n)$& $\sp (2n,\R)$& $\sp (n,\C)$& $\su^* (2n)\times \R$ & $\fu (2n)=\su (2n)\times \ft$\\
$\fe_{6(6)}$&$\fe_{6(-14)}$&$\sp (2,2)$& $\so (5,5)\times \R$ & $\so (10)\times \ft$\\
$\fe_{6(-26)}$&$\fe_{6(-14)}$ & $\ff_{4(-20)}$& $ \so(1,9)\times \R$ &$\so (10)\times \ft$ \\
$\fe_{7(7)}$&$\fe_{7(-25)}$&$\su^*(8)$& $\fe_{6(6)}\times \R$  & $\fe_6\times \ft$ \\
\hline
$\su (n,n)$& $\su (n,n)$& $\sl (n,\C)\times \R$&$ \sl (n,\C)\times \R $ & $\fs (\fu (n)\times \fu (n))$\\
$\so^*(4n)$& $\so^*(4n)$& $\su^*(2n)\times \R$& $\su^*(2n)\times \R$ &$\su (2n)\times \ft$\\
$\so (2,n)$& $\so (2,n)$ & $\so(1,n-1)\times \R$& $\so(1,n-1)\times \R$ & $\so (n)\times \ft$ \\
$\sp (n,\R)$& $\sp (n,\R)$& $\sl (n,\R)\times \R$& $\sl (n,\R)\times \R$ & $\su (n)\times \ft$  \\
$\fe_{7(-25)}$&$\fe_{7(-25)}$&$\fe_{6(-26)}\times\R$&$\fe_{6(-26)}\times\R$ &$\fe_{6}\times \ft$\\
\hline
\end{tabular}
 }
\medskip

\caption{$\fg$ without complex structure}
\end{table}

\begin{table}[!h]
\center
\scriptsize{
\begin{tabular}{|c|c|c|c|c|}
\hline
$\fg_c$ & $\fg_h$&$\fg$ &$ \fl_c $ & $\fk$\\
\hline
$\su (n,n)$& $\su (n,n)$& $\sl (n,\C)\times \R$&$ \sl (n,\C)\times \R $ & $\fu (n)$\\
$\su^*(2(p+q))$&$\su (2p,2q)$& $\sp (p,q)$& $\su^* (2p)\times \su^* (2q)\times \R$ &$ (\sp (p)\times \sp (q))$\\
$\so^*(4n)$& $\so^*(4n)$& $\su^*(2n)\times \R$& $\su^*(2n)\times \R$ &$\sp (n)$\\
$\sp (n,n)$& $\sp (2n,\R)$& $\sp (n,\C)$& $\su^* (2n)\times \R$ & $\sp (n)$\\
$\so (1,q+1)  (q\ge 3) $ & $\so (2,q)$ & $\so (1,q)$ &  $\so (q)\times \R$ & $ \so (q)$\\
$\so (2,n),(n=2k)$& $\so (2,n)$ & $\so(1,n-1)\times \R$& $\so(1,n-1)\times \R$ & $\so (n - 1)$ \\
$\fe_{6(-26)}$&$\fe_{6(-14)}$ & $\ff_{4(-20)}$& $ \so(1,9)\times \R$ &$\so (9)$ \\
$\fe_{7(-25)}$&$\fe_{7(-25)}$&$\fe_{6(-26)}\times\R$&$\fe_{6(-26)}\times\R$ &$\ff_4$\\
\hline
$\sl (p+q,\R)$& $\su (p,q)$& $\so (p,q)$&$\fs (\gl (p,\R)\times \gl (q,\R) )$ &$\so (p)\times \so (q))$\\
$\so (n,n)$& $\so^*(2n)$& $\so (n,\C)$& $\sl(n,\R)\times \R$ & $\so (n)$\\
$\sp (n,\R)$& $\sp (n,\R)$& $\sl (n,\R)\times \R$& $\sl (n,\R)\times \R$ & $\so (n)$  \\
\hline

$\so (2,n),( n=2k +1)$& $\so (2,n)$ & $\so(1,n-1)\times \R$& $\so(1,n-1)\times \R$ & $\so (n - 1)$ \\
$\so (p+1,q+1) (q\ge p\ge 2)$& $\so (2,p+q)$&$\so (1,p)\times \so (1,q)$& $\so (p,q)\times \R$ & $\so (p)\times \so (q)$\\
$\fe_{6(6)}$&$\fe_{6(-14)}$&$\sp (2,2)$& $\so (5,5)\times \R$ & $\sp(2) \times \sp(2)$\\
$\fe_{7(7)}$&$\fe_{7(-25)}$&$\su^*(8)$& $\fe_{6(6)}\times \R$  & $\sp (4)$ \\
\hline
\end{tabular}
}
\medskip

\caption{$\fg$ by type}
\end{table}

In Table 4 and Table 5 we can assume the $q\ge p$ because interchanging the role of $p$ and $q$ leads to
isomorphic cases. The case   $\fg_h = \so(2,q)\supset \fg =\so (1,q)$ corresponds to the case $p=0$,
and the case $p=1$ corresponds to the case $\fg_h =\so (2,n)\supset \fg=\so (1,n-1)\times \R$.
The case $\fg_h=\so (2,2)$ is excluded because $\so (2,2)$ is not simple. 
 
In Table 5 we have reorganized Table 4 into three groups. The first group consists of those $\fg$ for which the $\fl_c$ has 
one conjugacy class of Cartan subalgebra (denoted OCCC). The second group consists of those $\fg$ for which 
$\fl_c$ consists of automorphisms of a vector space while the maximal compact, $\fk$, of $\fg$ corresponds to 
isometries of the space. The third group consists of exceptions that will be treated individually. Of course there 
are ways, say using the octonions, to incorporate some of the third group into the second but 
we prefer this way. Notice that in all groups $\fk$ is the maximal compact for both $\fg$ and $\fl_c$.


\section*{Acknowledgements} The research by {\'O}lafsson was partially supported by NSF grants DMS-0801010
and DMS-1101337; both authors are grateful for support provided by the Max-Planck-Institut f\"ur Mathematik, Bonn

\end{document}